\documentclass[reqno]{amsart}

\usepackage{amsmath}
\usepackage{amsthm}
\usepackage{amsfonts}
\usepackage{amssymb}
\usepackage{enumerate}
\usepackage{graphicx}
\usepackage{todonotes}
\usepackage{epstopdf}
\usepackage{color}
\usepackage{tikz}
\usepackage{subcaption}

\renewcommand{\ln}{\log} 

\newcommand{\1}{\mathbf{1}}

\DeclareMathOperator{\cov}{Cov}
\DeclareMathOperator{\var}{Var}
\DeclareMathOperator{\tr}{tr}

\DeclareMathOperator{\diag}{diag}
\DeclareMathOperator{\dist}{dist}
\DeclareMathOperator{\rank}{rank}
\DeclareMathOperator{\supp}{supp}

\newcommand{\M}{\operatorname{M}}

\newcommand{\Prob}{\mathbb{P}}
\newcommand{\E}{\mathbb{E}}
\newcommand{\C}{\mathbb{C}}
\renewcommand{\P}{\mathbb{P}}
\renewcommand\Re{\operatorname{Re}}
\renewcommand\Im{\operatorname{Im}}
\newcommand{\eps}{\varepsilon}
\newcommand*\lap{\mathop{}\!\Delta}
\newcommand{\one}{\mathbf{1}}

\def\R{\mathbb{R}}

\theoremstyle{plain}
  \newtheorem{theorem}{Theorem}[section]
  
  \newtheorem{lemma}[theorem]{Lemma}
  \newtheorem{corollary}[theorem]{Corollary}

\theoremstyle{definition}
  \newtheorem{definition}[theorem]{Definition}
  \newtheorem{example}[theorem]{Example}
  \newtheorem{assumption}[theorem]{Assumption}

\theoremstyle{remark}
  \newtheorem{remark}[theorem]{Remark}

\usepackage{bbm}
\newcommand{\abs}[1]{\left\vert#1\right\vert} 
\newcommand{\set}[1]{\left\{#1\right\}} 

\newcommand{\ind}{\mathbbm 1} 

\begin{document}

\title[Local behavior of critical points and roots of random polynomials]{On the local pairing behavior of critical points and roots of random polynomials} 

\author{Sean O'Rourke}
\address{Department of Mathematics\\University of Colorado\\Campus Box 395\\Boulder, CO 80309-0395\\USA}
\email{sean.d.orourke@colorado.edu}
\thanks{S. O'Rourke has been supported in part by NSF grants ECCS-1610003 and DMS-1810500.}

\author{Noah Williams}
\address{Department of Mathematical Sciences\\Appalachian State University\\342 Walker Hall\\121 Bodenheimer Dr.\\Boone, NC 28608\\USA}
\email{williamsnn@appstate.edu}

\begin{abstract}
We study the pairing between zeros and critical points of the polynomial $p_n(z) = \prod_{j=1}^n(z-X_j)$, whose roots $X_1, \ldots, X_n$ are complex-valued random variables. Under a regularity assumption, we show that if the roots are independent and identically distributed, the Wasserstein distance between the empirical distributions of roots and critical points of $p_n$ is on the order of $1/n$, up to logarithmic corrections.  The proof relies on a careful construction of disjoint random Jordan curves in the complex plane, which allow us to naturally pair roots and nearby critical points.  In addition, we establish asymptotic expansions to order $1/n^2$ for the locations of the nearest critical points to several fixed roots. This allows us to describe the joint limiting fluctuations of the critical points as $n$ tends to infinity, extending a recent result of Kabluchko and Seidel. Finally, we present a local law that describes the behavior of the critical points when the roots are neither independent nor identically distributed.
\end{abstract}

\maketitle

\section{Introduction}

This paper concerns the nature of the pairing between the critical points and roots of random polynomials in a single complex variable.  In particular, we consider polynomials of the form 
\begin{equation} \label{eq:def:pn}
p_n(z):= \prod_{j=1}^n(z-X_j),
\end{equation}
where $X_1, \ldots, X_n$ are complex-valued random variables (not necessarily independent or identically distributed). While much is known about the locations of the critical points of $p_n$ when the roots are deterministic (see for example Marden's book \cite{M} which contains the Gauss--Lucas theorem and Walsh's two circle theorem among other results), Pemantle and Rivin \cite{PR}, Hanin \cite{H1, H2, H3}, and Kabluchko \cite{K} first demonstrated that the random version of this problem admits greater precision, especially when the degree $n$ is large. 

In particular, Pemantle and Rivin conjectured that when $X_1, \ldots, X_n$ are chosen to be independent and identically distributed (iid) with distribution $\mu$, then the empirical distribution constructed from the critical points of $p_n$ converges weakly in probability to $\mu$. They proved their conjecture in \cite{PR} for measures satisfying some technical assumptions, and Subramanian \cite{S} refined their work for $\set{X_j}_{j=1}^n$ on the unit circle.  Kabluchko first proved the conjecture in full generality in \cite{K} to obtain the following result. 
\begin{theorem}[Kabluchko \cite{K}] \label{thm:kabluchko}
Let $X_1, X_2, \ldots$ be iid complex-valued random variables with distribution $\mu$.  Then for any bounded and continuous function $\varphi: \mathbb{C} \to \mathbb{C}$, 
\[ \frac{1}{n-1} \sum_{j=1}^{n-1} \varphi(w_j^{(n)}) \longrightarrow \int \varphi(z) d \mu(z) \]
in probability as $n \to \infty$, where $w_1^{(n)}, \ldots, w_{n-1}^{(n)}$ are the critical points of the polynomial
$ p_n(z) =  \prod_{j=1}^n(z-X_j). $
\end{theorem}
 
Inspired by such results, the first author established several versions of Theorem \ref{thm:kabluchko} for random polynomials with dependent roots that satisfy some technical conditions \cite{O}. For example, the conclusion of Theorem \ref{thm:kabluchko} holds for characteristic polynomials of certain classes of matrices from the classical compact  matrix groups. Additionally, in \cite{OW}, the authors adapted  Kabluchko's strategy to the situation where $p_n$ is perturbed to have $o(n)$ deterministic roots.  Two other relevant works include Reddy's thesis \cite{TRR} and the recent paper of Byun, Lee, and Reddy \cite{BLR}, who showed that under some mild assumptions, Kabluchko's result holds when $p_n$ has mostly deterministic roots and several (potentially dependent) random ones. Byun, Lee, and Reddy proved several additional results including that the sequence of empirical measures constructed from the zeros of $p^{(k)}_n$ converges weakly in probability to the distribution $\mu$, for \textit{any} fixed choice of $k$, as well as a version of Theorem \ref{thm:kabluchko} when the roots $X_1, \ldots, X_n$ are given by a $2D$ Coulomb gas density.  

Theorem \ref{thm:kabluchko} and most of the cited works above focus on the macroscopic, or global, behavior of the critical points of $p_n$.  For example, by combining Theorem \ref{thm:kabluchko} with the Law of Large Numbers, one obtains that, for any bounded and continuous function $\varphi: \mathbb{C} \to \mathbb{C}$, 
\begin{equation} \label{eq:kabluchko}
	\sum_{j=1}^n \varphi(X_j) = \sum_{j=1}^{n-1} \varphi(w_j^{(n)}) + o(n) 
\end{equation}
with high probability\footnote{See Section \ref{sec:notation} for a complete description of the asymptotic notation used here and in the sequel.}.  
In contrast to Theorem \ref{thm:kabluchko}, this paper focuses on describing the local behavior of the critical points.  

One important aspect of the local critical point behavior 
is that the critical points and roots of $p_n$ appear to pair with one another. Theorem \ref{thm:kabluchko} and \eqref{eq:kabluchko} describe this phenomenon at the macroscopic level by comparing the global behaviors of the critical points and roots. However, a glance at Figures \ref{fig:twoDisks} and \ref{fig:norm} suggests that a stronger pairing phenomenon exists.
In particular, one sees that nearly every critical point is paired closely with a root of $p_n$, an indication that the local behavior of the critical points should be extremely similar to the local behavior of the roots. 
\begin{figure}
\includegraphics[width =.8\columnwidth]{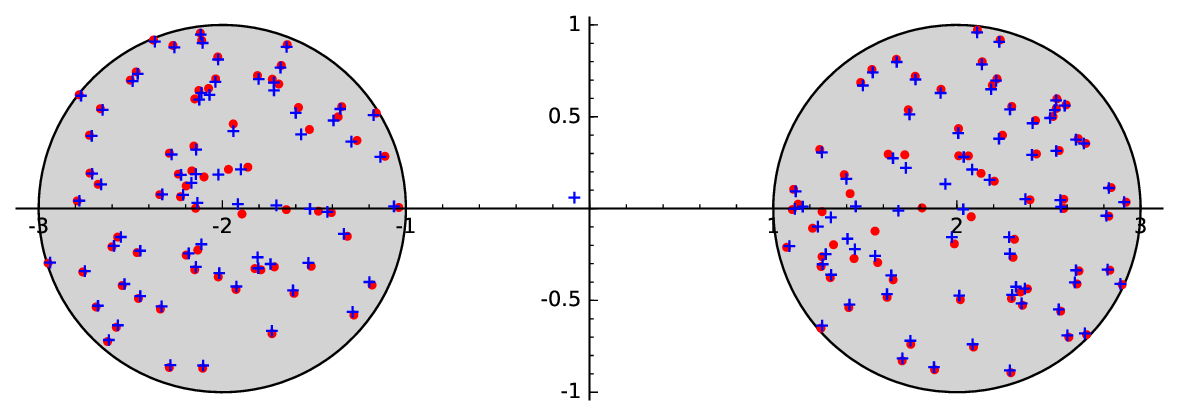}
\caption{The roots (red dots) and critical points (blue crosses) of a random, degree $150$ polynomial, where all $150$ roots are chosen independently and uniformly on two disks. See Example \ref{ex:twoDisks}.}
\label{fig:twoDisks}
\end{figure}

\begin{figure}
\includegraphics[width =.8\columnwidth]{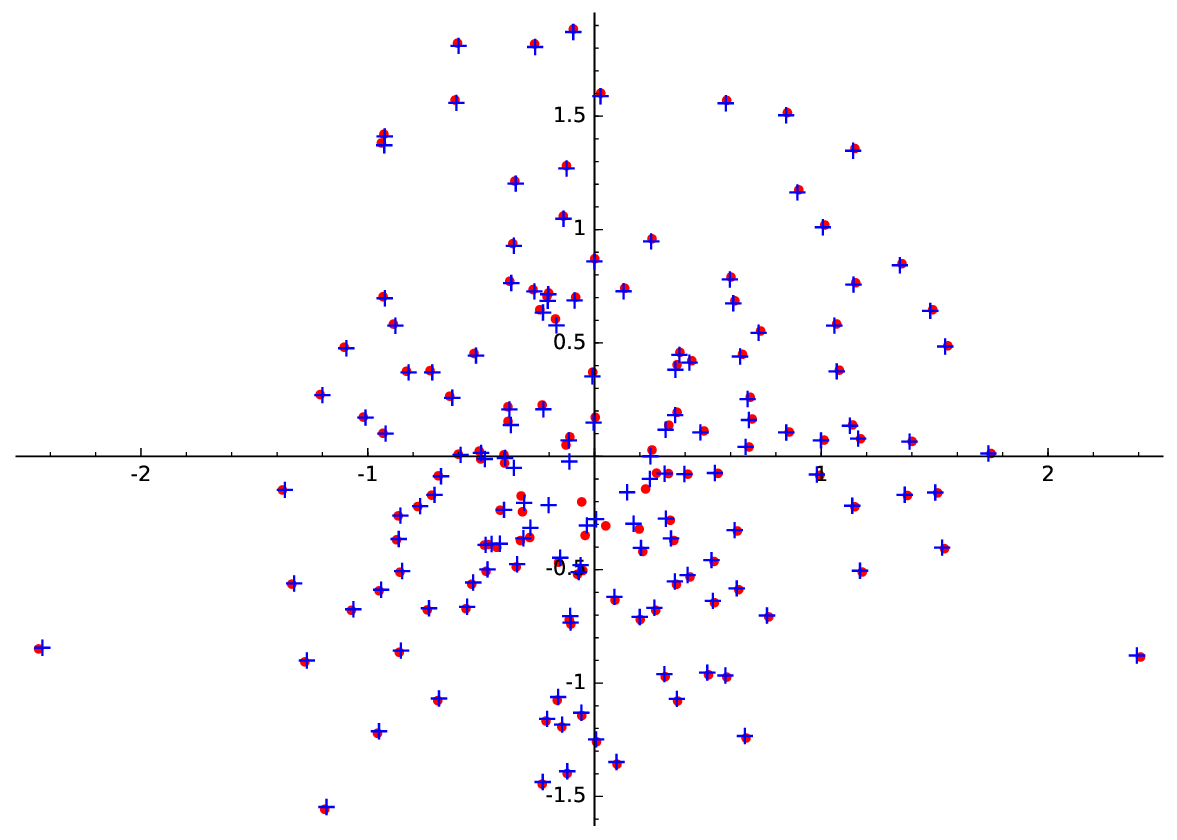}
\caption{The roots (red dots) and critical points (blue crosses) of a random, degree $150$ polynomial, where all $150$ roots are chosen independently according to a standard complex normal distribution. See Example \ref{ex:norm}.}
\label{fig:norm}
\end{figure}

Hanin investigated the pairing phenomenon between roots and critical points for several classes of random functions \cite{H1,H2,H3}, including random polynomials with independent roots. He proved that the distance between a fixed, deterministic root and its nearest critical point is roughly $1/n$ in the case where $\mu$ has a bounded density supported on the Riemann sphere \cite{H3}. The root-and-critical point pairing for random polynomials was also explored in \cite{OW, OWo}, and Dennis and Hannay gave an electrostatic explanation of the phenomenon in \cite{DH}.  Most recently, Steinerberger showed that the pairing phenomenon also holds for some classes of deterministic polynomials \cite{SS}, and Kabluchko and Seidel determined the asymptotic fluctuations of the critical point of $p_n$ that is nearest a given root \cite{KS}. Kabluchko and Seidel's results are similar to some of our conclusions below and appear to have been concurrently derived using different methods. We present a detailed comparison between \cite{KS} and our work in the next section. 


In this paper, we refine the results mentioned above to obtain a more complete picture of the pairing that occurs between zeros and critical points of $p_n$. We begin by exhibiting a bound on the Wasserstein, or ``transport,'' distance between the collections of roots and critical points of $p_n$. While this result explains the nearly 1-1 pairing between roots and critical points in Figures \ref{fig:twoDisks} and \ref{fig:norm}, it does not allow one to describe the behavior near any particular root. We accomplish this feat in the next section of the paper, where we discuss the joint fluctuations  for a fixed number of critical points of $p_n$. We conclude our analysis by establishing a local law that describes the mesoscopic behavior of the critical points of $p_n$.  
Many of our results focus on the cases where the roots $X_1, \ldots, X_n$ of $p_n$ are iid, but for some of our results, we do not even require that the roots be independent.

\subsection{Notation} \label{sec:notation}

Throughout the paper, we use asymptotic notation, such as $O$ and $o$, under the assumption that $n \to \infty$.  We write $X_n = O(Y_n)$ , $Y_n = \Omega(X_n)$, $X_n \ll Y_n$, or $Y_n \gg X_n$ to denote the bound $|X_n| \leq C Y_n$ for some constant $C > 0$ and for all $n > C$.  If the implicit constant depends on a parameter $k$, e.g., $C = C_k$, we denote this with subscripts, e.g., $X_n = O_k(Y_n)$ or $X_n \ll_k Y_n$.  By $X_n=o_k(Y_n)$, we mean that for any $\eps > 0$, there is a natural number $N_{\eps,k}$ depending on $k$ and $\eps$ for which $n \geq N_{\eps,k}$ implies $\abs{X_n} \leq \eps {Y_n}$.  
In general, $C,c, K$ are constants which may change from one occurrence to the next.  We often use subscripts, such as $C_{P_1,P_2,\ldots}$, to denote that the constant depends on some parameters $P_1, P_2,\ldots$.  

We use the following set-theoretic conventions. For $z_0 \in \mathbb{C}$ and $r \geq 0$, we define
\[
B(z_0, r) := \{ z \in \mathbb{C} : |z - z_0| < r \}
\] 
to be the open ball of radius $r$ centered at $z_0$, and $\overline{B}(z_0,r)$ to be its closure. The notations $\#S$ and $\abs{S}$ denote the cardinality of the finite set $S$. The natural numbers, $\mathbb{N}$, do not include zero. 

For a probability measure $\mu$, we use $X \sim \mu$ to mean that the random variable $X$ has distribution $\mu$ and $\supp(\mu)$ to denote its support. We say that a probability measure $\mu$ on $\mathbb{C}$ has \emph{density} $f$ if $\mu$ is absolutely continuous with respect to Lebesgue measure on $\mathbb{C}$ and the Radon--Nikodym derivative of $\mu$ with respect to Lebesgue measure is $f$. The random variable $\ind_E$ is the indicator supported on the event $E$, and we say an event $E$ (which depends on $n$) holds with \emph{overwhelming probability} if for every $\alpha > 0$, $\Prob(E) \geq 1 - O_\alpha(n^{-\alpha})$. 

Finally, we use $d^2z$ to denote integration with respect to the Lebesgue measure on $\C$ to avoid confusion with complex line integrals, where we integrate against $dz$. We use $\sqrt{-1}$ to denote the imaginary unit and reserve $i$ as an index.

\subsection*{Acknowledgements} 
The authors thank Boris Hanin for calling their attention to this line of research and for many useful conversations.  The authors also thank the anonymous referees for useful feedback and corrections.  

\section{Main Results}\label{sec:results}

We begin by introducing the Wasserstein metric in order to discuss the pairing between the roots and critical points of $p_n$ that one sees in Figures \ref{fig:twoDisks} and \ref{fig:norm}.  

\subsection{Wasserstein distance} \label{sec:wasserstein}
For probability measures $\mu$ and $\nu$ on $\mathbb{C}$, let $W_1(\mu, \nu)$ denote the $L_1$-Wasserstein distance between $\mu$ and $\nu$ defined by
\[ W_1(\mu, \nu) := \inf_{\pi} \int |x - y| d \pi(x,y), \]
where the infimum is over all probability measures $\pi$ on $\mathbb{C} \times \mathbb{C}$ with marginals $\mu$ and $\nu$ (see e.g. \cite{V}, Chapter 6). 
Theorem \ref{thm:wasserstein} below gives a bound on the Wasserstein distance between the empirical measures constructed from the roots $X_1, \ldots, X_n$ and the critical points $w_1^{(n)}, \ldots, w_{n-1}^{(n)}$ of the polynomial $p_n$ defined in \eqref{eq:def:pn}. We denote these empirical measures by
\begin{equation} \label{eq:def:mun}
\mu_n := \frac{1}{n}\sum_{j=1}^n\delta_{X_j}\quad \text{and}\quad \mu'_n := \frac{1}{n-1}\sum_{j=1}^{n-1}\delta_{w_j^{(n)}},
\end{equation}
respectively.
%
Before we state Theorem \ref{thm:wasserstein}, we motivate some regularity assumptions $\mu$ must satisfy in the hypothesis.
%

Consider that
\begin{equation}
\frac{1}{n}p_n'(z) = \prod_{j=2}^n(z-X_j)\left((z-X_1)\frac{1}{n}\sum_{j=2}^n\frac{1}{z-X_j} + \frac{1}{n}\right),
\label{eqn:whyCauchy}
\end{equation}
where the sum on the right-hand side is an empirical mean of iid random variables. Provided $\mu$ is sufficiently nice, the Law of Large Numbers implies $\frac{1}{n}\sum_{j=2}^n\frac{1}{z-X_j}$ converges in distribution to the Cauchy--Stieltjes transform of $\mu$, which is given by
\begin{equation} \label{eq:def:m_mu}
m_\mu(z) := \int_\C \frac{d\mu(x)}{z-x},
\end{equation}
and defined for those values of $z \in \mathbb{C}$ for which the integral exists. Heuristically speaking, if $m_{\mu}(z)$ is finite and bounded away from zero near $z=X_1$, then $p'_n(z) \approx 0$ for some $z$ satisfying $\abs{z-X_1} = O(1/n)$. If, on the other hand, $\abs{m_\mu(z)}$ is close to $0$ for $z$ near $X_1$, we have $p'_n(z) \approx \prod_{j=2}^n(z-X_j)$, so $p'_n(z)$ need not have any zeros near $X_1$. Similar heuristic intuition applies if we replace $X_1$ in turn with $X_2, \ldots, X_n$. 

In light of the discussion above, conditions on the Cauchy--Stieltjes transform of $\mu$ feature prominently in this paper, particularly in the hypothesis of Theorem \ref{thm:wasserstein}, which requires at least one of the assumptions below.

%
%
%
%
%
%

\begin{assumption}
Suppose there are positive constants, $C_1,C_2$, so that the following conditions hold when $X_1, \ldots, X_n$ are iid complex-valued random variables with common  distribution $\mu$:
\begin{enumerate}[(i)]
\item \label{it:subquad} for any $\eps > 0$, 
$
\P(\abs{m_\mu(X_1)} < \eps) \leq C_1\eps^2;
$
\item \label{it:maxBd} the random variable $\eta_n:=\max_{1\leq j\leq n}\abs{X_j}$ satisfies 
$
\P\left(\eta_n \geq n^{C_2}\right) = o(1). 
$
\end{enumerate}
\label{assum:subquad}
\end{assumption}

\begin{assumption}[Alternative to Assumption \ref{assum:subquad} for radially symmetric distributions] Suppose $\mu$ has two finite absolute moments and a continuous density, $f$, that is radially symmetric about $z=z_0$ and that satisfies $f(z_0) > 0$. 
\label{assum:radSym}
\end{assumption}

We can now state the main result of this subsection.

\begin{theorem} \label{thm:wasserstein}
Let $X_1, \ldots, X_n$ be iid, complex random variables whose distribution, $\mu$, has a bounded density and satisfies either Assumption \ref{assum:subquad} or Assumption \ref{assum:radSym}. Then, there is a positive constant $C$, depending on $\mu$, so that with probability $1 - o(1)$,
\begin{equation} \label{eq:wasserstein:bound}
W_1(\mu_n, \mu'_n) \leq \frac{C\eta_n(\ln{n})^{9}}{n},
\end{equation}
where $\eta_n := \max_{1 \leq j\leq n}\abs{X_j}$, and $\mu_n, \mu'_n$ (defined in \eqref{eq:def:mun})
are the empirical measures constructed from the roots and critical points of $p_n(z) = \prod_{j=1}^n(z-X_j).$
\end{theorem}

In the case where $\mu$ has sub-exponential tails, one can show that with probability tending to $1$, $\eta_n = O(\ln{n})$. Consequently, Theorem \ref{thm:wasserstein} immediately implies the following corollary.  
\begin{corollary} \label{cor:wasserstein}
Let $X_1, \ldots, X_n$ be iid, complex random variables whose distribution, $\mu$, has a bounded density and satisfies Assumption \ref{assum:subquad} part \eqref{it:subquad} in addition to the following condition:
\begin{enumerate}[(i')]
\addtocounter{enumi}{1}
\item there exist $C,c>0$ such that if $X \sim \mu$, then, $\P(\abs{X} > t) \leq Ce^{-ct}$ for every $t> 0$.
\end{enumerate}
Then, there is a positive constant $C_\mu$, depending only on $\mu$, so that with probability $1 - o(1)$,
\begin{equation*}
W_1(\mu_n, \mu'_n) \leq \frac{C_\mu(\ln{n})^{10}}{n},
\end{equation*}
where $\mu_n, \mu'_n$ (defined in \eqref{eq:def:mun})
are the empirical measures constructed from the roots and critical points of $p_n(z) = \prod_{j=1}^n(z-X_j).$
\end{corollary}

Theorem \ref{thm:wasserstein} and Corollary \ref{cor:wasserstein} show that the roots and critical points can be paired in such a way that the typical spacing between a critical point and its paired root is $O(n^{-1})$, up to logarithmic corrections. This precisely describes the phenomenon observed in Figures \ref{fig:twoDisks} and \ref{fig:norm}, and the authors believe that these bounds are optimal (up to logarithmic factors) based on the theorems of Section \ref{sec:microscopic} below and the results in \cite{KS}.  

A couple of remarks concerning Theorem \ref{thm:wasserstein} and its corollary are in order. Due to the heuristic that motivates our proof of Theorem \ref{thm:wasserstein} (see Figure \ref{fig:loops}), the authors conjecture that Assumption \ref{assum:subquad} part \eqref{it:subquad} can be weakened to require that for some fixed $\delta > 0$, $\P(\abs{m_\mu(X_1)}< \eps) \leq C_1 \eps^{1+\delta}$. At present, we require $\delta = 1$ to obtain some technical bounds in the proof.  An examination of the proof reveals exactly where this condition is needed. 
The second remark concerns the appearance of $\eta_n$ on the right-hand side of \eqref{eq:wasserstein:bound}.  The authors believe this term is at least partially necessary.  Indeed, based on numerical experiments, the Wasserstein distance $W_1(\mu_n, \mu'_n)$ appears larger for distributions $\mu$ with extremely heavy tails.  The precise dependence of $\eta_n$ on the Wasserstein distance remains an open question.  

\subsection{Examples of Theorem \ref{thm:wasserstein} and Corollary \ref{cor:wasserstein}}
The assumptions of Theorem \ref{thm:wasserstein} and Corollary \ref{cor:wasserstein} are rather technical, so this subsection is devoted to several specific examples worked out in detail.  

\begin{example}[$\mu$ is uniform on a disk]\label{ex:oneDisk}
If $\mu$ has a uniform distribution on the disk of radius $R$ centered at $z_0$, then, $\mu$ has density 
$
f(z) = \frac{1}{\pi R^2}\ind_{\abs{z-z_0} \leq R}
$
and Cauchy--Stieltjes transform
\[
m_\mu(z) = \begin{cases}\frac{1}{R^2}\left(\overline{z-z_0}\right) & \text{if $\abs{z-z_0} \leq R$},\\\frac{1}{z-z_0} &\text{if $\abs{z-z_0} \geq R$}.\end{cases}
\]
(Lemma \ref{lem:radSymMu} facilitates the computation of $m_\mu(z)$ when $\mu$ is radially symmetric. For this example, apply Lemma \ref{lem:radSymMu} when $z=0$, $R = 1$, and apply a linear transformation.) It follows that if $X \sim \mu$, then for any $\eps < 1$,
\[
\P\left(\abs{m_\mu(X)} < \eps\right) \leq \P\left(\abs{X-z_0} < R^2\eps\right) = R^2\eps^2,
\]
so $\mu$ satisfies Assumption \ref{assum:subquad}, and by Theorem \ref{thm:wasserstein}, with probability $1-o(1)$, $W_1(\mu_n, \mu'_n) =O((\ln{n})^9/n)$. (Note that almost surely, $\eta_n \leq \abs{z_0} + R$).
\label{ex:disks}
\end{example}

\begin{example}[$\mu$ is supported on all of $\C$]\label{ex:norm}
Assumption \ref{assum:radSym} is easy to verify for a large class of measures that do not necessarily have compact support. For example, suppose $\mu$ has a standard complex normal distribution with density
$
f(z) = \frac{1}{\pi} e^{-\abs{z}^2}.  
$
Clearly, $\mu$ is radially symmetric about the origin, and $f(z)$ is continuous with $f(0) = \pi^{-1} > 0$. Furthermore, $\mu$ has sub-exponential tails, so by Corollary \ref{cor:wasserstein}, with probability tending to $1$, $W_1(\mu_n, \mu_n') \leq O((\ln{n})^{10}/n)$. Figure \ref{fig:norm} illustrates this example.
%
\end{example}

\begin{example}[$\mu$ is not radially symmetric]\label{ex:twoDisks}
In this last example, we consider a situation where $\mu$ does not exhibit radial symmetry. Suppose $\mu$ is uniform on the two disks $B(-2, 1)$ and $B(2,1)$ with  density
\[
f(z) = \frac{1}{2\pi}\left(\ind_{\abs{z+2}< 1}(z) + \ind_{\abs{z-2} < 1}(z)\right),
\]
which is depicted in Figure \ref{fig:twoDisks}. By separately considering the cases $\abs{z+2} < 1$, $\abs{z-2} < 1$, and $\abs{z\pm 2} \geq 1$, we can use the calculations from Example \ref{ex:disks} to obtain the Cauchy--Stieltjes transform:
\begin{equation}
m_\mu(z) = \begin{cases} \frac{1}{2}\left(\overline{z+2} + \frac{1}{z-2}\right)&\text{if $\abs{z+2} < 1$,}\\ \frac{z}{z^2 -4} &\text{if $\abs{z\pm 2} \geq 1$,}\\
\frac{1}{2}\left(\overline{z-2} + \frac{1}{z+2}\right) &\text{if $\abs{z-2} < 1$}.\end{cases}
\label{eqn:twoCirc}
\end{equation}
Since $\mu$ has compact support, Assumption \ref{assum:subquad} part \eqref{it:maxBd} holds trivially. In Appendix \ref{sec:appendixA}, we establish part \eqref{it:subquad}, so by Theorem \ref{thm:wasserstein}, with probability $1-o(1)$, $W_1(\mu_n,\mu'_n) = O((\ln{n})^9/n)$. 
\end{example}

\subsection{Fluctuations of the critical points} \label{sec:microscopic}
While Theorem \ref{thm:wasserstein} describes the typical distance between a root and its paired critical point, it does not allow one to study any particular root or critical point. Toward this end, we now fix several of the roots and treat them as deterministic: consider the polynomial  
\[
p_n(z) :=\prod_{l=1}^s(z-\xi_l)\prod_{j=1}^{n+1-s}(z-X_j), \] 
where $X_1, \ldots, X_{n+1-s}$ are iid complex-valued random variables with distribution $\mu$, and $\vec{\xi} = (\xi_1, \ldots, \xi_s)$  is a deterministic vector in $\C^s$.  Our goal is to simultaneously study the behavior of the critical points closest to $\xi_l$, $1 \leq l \leq s$.

Our first result, Theorem \ref{thm:multiInCLT}, covers the situation where some of the values $\xi_1, \ldots, \xi_s$ are allowed to be inside the support of $\mu$. In particular, for each $1 \leq l \leq s$, equation \eqref{eqn:multiInLoc} locates the critical point, $w_l^{(n)}$, that is near $\xi_l$ to within $O(n^{-2})$  (up to logarithmic corrections). This bound indicates that each $w_l^{(n)}$ is centered at
\begin{equation}
\hat{w}_l^{(n)} := \xi_l - \frac{1}{n+1}\frac{n}{\sum_{j\neq l}\frac{1}{\xi_l-\xi_j} +\sum_{j=1}^{n+1-s}\frac{1}{\xi_l - X_j}},
\label{eqn:wHat}
\end{equation}
rather than $\xi_l$.  This observation allows us to express the fluctuations of each critical point as a sum of independent random variables (up to some lower order error terms), and we use this to show that the fluctuations of the vector $(w_1^{(n)}, \ldots, w_s^{(n)})$ converge in distribution to a multivariate normal distribution. See Figure \ref{fig:BlobLLNx4}.

\begin{figure}
\includegraphics[width =1\columnwidth, trim = {0 0 0 1.5cm}, clip]{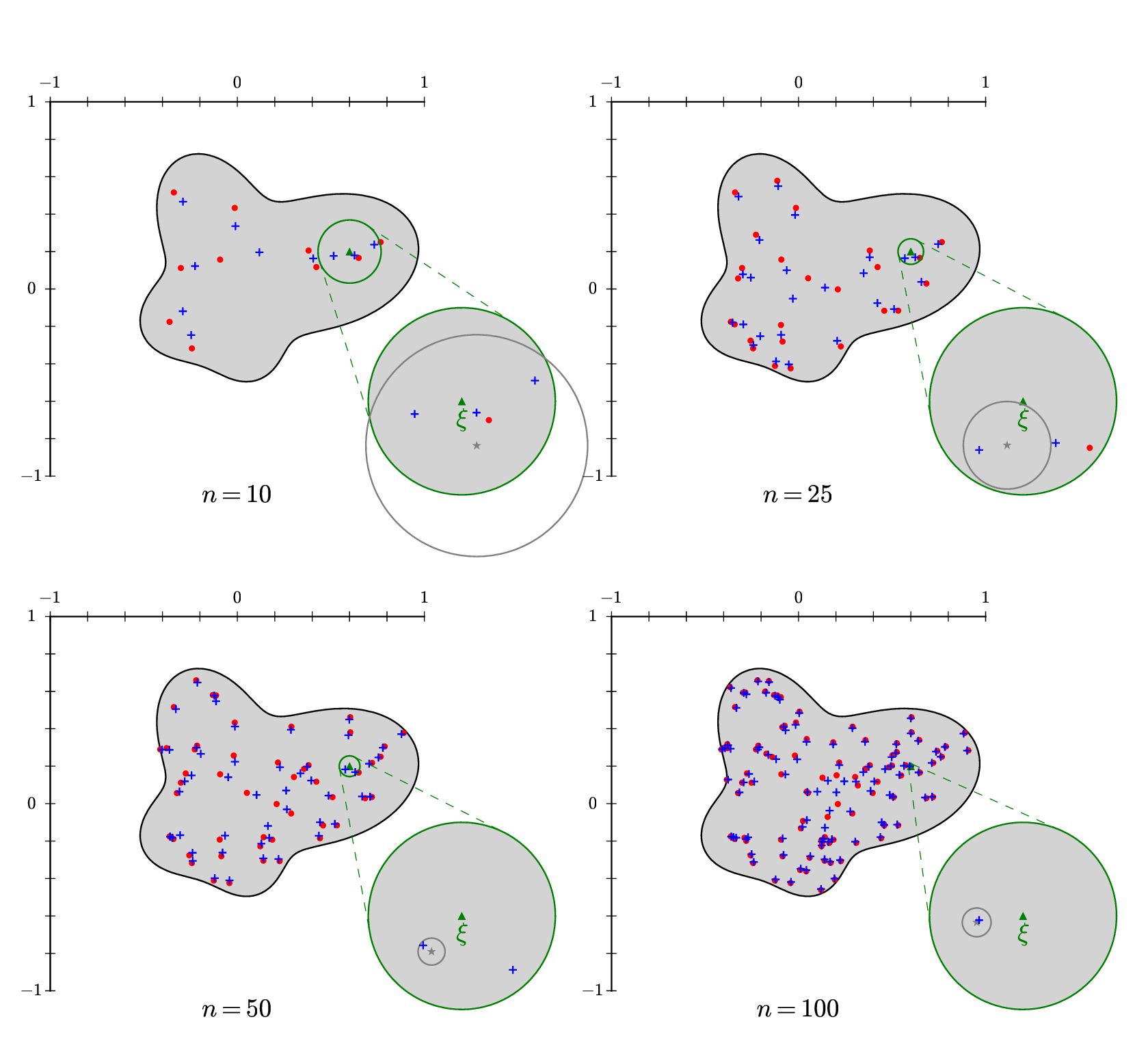}
\caption{Simulation to illustrate Theorem \ref{thm:multiInCLT}. The roots (red dots) and critical points (blue crosses) of $p_n(z) = (z-\xi_1)\prod_{j=1}^n(z-X_j)$ for increasing values of $n$, where the roots, $X_1, \ldots, X_{100}$, are chosen independently and uniformly on the outlined region. The green circle centered at $\xi_1$ is of radius $\frac{2n}{n+1}\left(\sum_{j=1}^n\frac{1}{\xi_1-X_j}\right)^{-1}$ and the gray circle has radius $\frac{20}{n^2}$ and center $\hat{w}_1^{(n)}$ (see \eqref{eqn:wHat}).} 
\label{fig:BlobLLNx4}
\end{figure}

In order to state Theorem \ref{thm:multiInCLT} we need the following definitions. Let
\begin{align*}
\M_{\mu} &:= \left\{ z \in \mathbb{C} : m_{\mu}(z) = 0 \right\}
\end{align*}
denote the set of zeros of $m_\mu$. We say that a measure $\mu$ has a \textit{density in a neighborhood of $z_0$} if there exists a $\rho > 0$ so that the restriction of $\mu$ to the open ball $B(z_0,\rho)$ is absolutely continuous with respect to the Lebesgue measure on $B(z_0,\rho)$.

\begin{theorem}[Locations and fluctuations of critical points when $p_n$  has several deterministic roots]
Let $X_1, X_2, \ldots$ be iid complex-valued random variables with distribution $\mu$, fix $s$ and the distinct, deterministic values $\xi_1, \ldots, \xi_s \notin \M_\mu$, and suppose that in a neighborhood of each $\xi_l$, $1 \leq l \leq s$, $\mu$ has a bounded density, $f$. Then, with probability $1 - o(1)$, the polynomial
\[p_n(z) = \prod_{l=1}^s(z-\xi_l)\prod_{j=1}^{n+1-s}(z-X_j)\]
has $s$ critical points, $w_1^{(n)}, \ldots, w_s^{(n)}$, such that for $1 \leq l \leq s$, $w_l^{(n)}$ is the unique critical point of $p_n$ that is within a distance of $\frac{3}{\abs{m_\mu(\xi_l)}n}$ of $\xi_l$, and  
\begin{equation}
\abs{w_l^{(n)} - \xi_l + \frac{1}{n+1}\frac{n}{\sum_{j\neq l}\frac{1}{\xi_l - \xi_j} + \sum_{j=1}^{n+1-s}\frac{1}{\xi_l - X_j}}} = O_{\mu,\vec{\xi}}\left(\left(\frac{\ln n}{n}\right)^2\right).
\label{eqn:multiInLoc}
\end{equation}
In addition, if $f$ is continuous at $\xi_1, \ldots, \xi_s$, then we have
\begin{equation}\label{eqn:multiInCLT}
\left(\frac{n^{3/2}}{\sqrt{\ln{n}}}\cdot m_\mu(\xi_l)^2\cdot\left(w_l^{(n)} -\xi_l + \frac{1}{n+1}\frac{1}{m_\mu(\xi_l)}\right)\right)_{l=1}^s \longrightarrow (N_1, \ldots, N_s) 
\end{equation}
in distribution as $n \to \infty$, where $(N_1, \ldots, N_s)$ is a vector of complex random variables whose real and imaginary components $(\Re(N_1), \Im(N_1), \ldots, \Re(N_s), \Im(N_s))$ have a multivariate normal distribution with mean zero and covariance structure characterized by
\begin{equation}
\begin{aligned}
\cov(\Re(N_j),\Re(N_l)) &= \begin{cases}\frac{\pi f(\xi_j)}{2}&\text{if $l=j$,}\\0&\text{else}\end{cases}\\
\cov(\Im(N_j),\Im(N_l)) &= \begin{cases}\frac{\pi f(\xi_j)}{2}&\text{if $l=j$,}\\0&\text{else}\end{cases}\\
\cov(\Re(N_j), \Im(N_l)) &= 0.
\end{aligned}
\label{eqn:multiInCov}
\end{equation}
%
\label{thm:multiInCLT}
\end{theorem}

\begin{remark}
Theorem \ref{thm:multiInCLT} can also be extended to the case where $\xi_1, \ldots, \xi_s$ are independent random variables (rather than deterministic values).  This can be seen by conditioning on $\xi_1, \ldots, \xi_s$ and applying Theorem \ref{thm:multiInCLT}; a similar argument was used in \cite{KS}.  
\end{remark}

Compare Theorem \ref{thm:multiInCLT} to Theorem 2.2 of \cite{KS}, which describes the same phenomenon when $s=1$. Both theorems identify the same fluctuations of $w_1^{(n)}$ about $\xi_1$, however, the two results locate the critical point $w_1^{(n)}$ on different scales. While Theorem 2.2 from \cite{KS} shows that $w_1^{(n)}$ is the unique critical point of $p_n$ within a distance of order $o(1/\sqrt{n})$ of $\xi_1$, Theorem \ref{thm:multiInCLT} refines the location of $w_1^{(n)}$ to within order $O(n^{-2})$ up to logarithmic corrections. In fact, since $\frac{1}{n}\sum_{j=1}^{n}\frac{1}{\xi_1 - X_j}$ converges almost surely to $m_\mu(\xi_1)$, the results of the two theorems can be combined to give a stronger picture of the local behavior of $w_1^{(n)}$. Note that in contrast to the method of proof used by Kabluchko and Seidel in \cite{KS}, our approach is based on a deterministic argument (see Theorem \ref{thm:detLoc}).  



For values of $\xi_1, \ldots, \xi_s$ outside the support of $\mu$, \eqref{eqn:multiInCLT} and \eqref{eqn:multiInCov} demonstrate that the scaling factor $n^{3/2}/\sqrt{\ln{n}}$ is too small to achieve a meaningful result. (Indeed, $f$ may be chosen to be identically zero outside $\supp(\mu)$, so the random vector $(N_1, \ldots N_s)$ is almost surely the zero vector.) The following result refines the analysis in this situation and is depicted in Figure \ref{fig:UnitCirc9x9}.

\begin{theorem}[Locations and Fluctuations of critical points when $p_n$ has several roots outside $\supp(\mu)$]
Let $X_1, X_2, \ldots$ be iid complex-valued random variables with common distribution $\mu$, fix $s\in \mathbb{N}$, and suppose $\xi_1, \ldots, \xi_s \notin \supp(\mu) \cup \M_\mu$ are distinct, fixed deterministic values. Then, there exist constants $C, c_{\mu,\vec{\xi}}, C_{\mu,\vec{\xi}} > 0$, so that with probability at least $1 - C\exp(-c_{\mu,\vec{\xi}}\,n)$, the polynomial
\[p_n(z) = \prod_{l=1}^s(z-\xi_l)\prod_{j=1}^{n+1-s}(z-X_j)\]
has $s$ critical points, $w_1^{(n)}, \ldots, w_s^{(n)}$, such that for $1 \leq l \leq s$, $w_l^{(n)}$ is the unique critical point of $p_n$ that is within a distance of $\frac{3}{\abs{m_\mu(\xi_l)}n}$ of $\xi_l$, and  
\begin{equation}
\abs{w_l^{(n)} - \xi_l + \frac{1}{n+1}\frac{n}{\sum_{j\neq l}\frac{1}{\xi_l - \xi_j} + \sum_{j=1}^{n+1-s}\frac{1}{\xi_l - X_j}}} < \frac{C_{\mu,\vec{\xi}}}{n^2}.
\label{eqn:multiOutLoc}
\end{equation}
In addition, we have
\begin{equation}
\left(n^{3/2}\cdot m_\mu(\xi_l)^2\cdot\left(w_l^{(n)} -\xi_l + \frac{1}{n+1}\frac{1}{m_\mu(\xi_l)}\right)\right)_{l=1}^s \longrightarrow (N_1, \ldots, N_s)
\label{eqn:multiOutCLT}
\end{equation}
in distribution as $n \to \infty$, where $(N_1, \ldots, N_s)$ is a vector of complex random variables whose real and imaginary components $(\Re(N_1), \Im(N_1), \ldots, \Re(N_s), \Im(N_s))$ have a multivariate normal distribution with mean zero and covariance structure
\begin{equation}
\begin{aligned}
\cov(\Re(N_j),\Re(N_l)) &= \cov\left(\Re\left(\frac{1}{\xi_j - X_1}\right),\Re\left(\frac{1}{\xi_l - X_1}\right)\right)\\
\cov(\Im(N_j),\Im(N_l)) &= \cov\left(\Im\left(\frac{1}{\xi_j - X_1}\right),\Im\left(\frac{1}{\xi_l - X_1}\right)\right)\\
\cov(\Re(N_j), \Im(N_l)) &= \cov\left(\Re\left(\frac{1}{\xi_j - X_1}\right),\Im\left(\frac{1}{\xi_l - X_1}\right)\right).
\end{aligned}
\label{eqn:multiOutCov}
\end{equation}
\label{thm:multiOutCLT}
\end{theorem}

\begin{figure}
\includegraphics[width =.85\columnwidth, trim = {0 0 2.5cm 1.3cm}, clip]{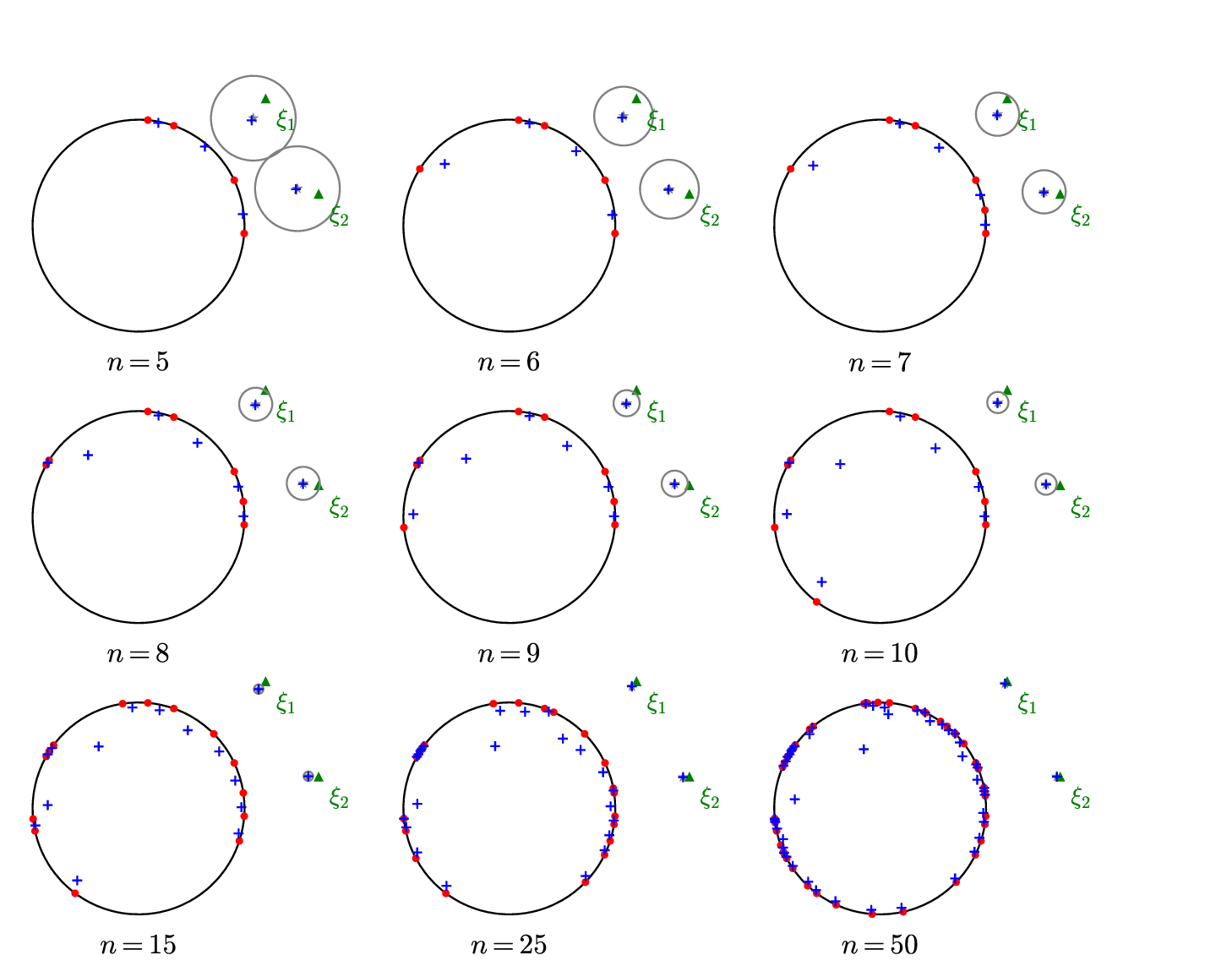}
\caption{Simulation to illustrate Theorem \ref{thm:multiOutCLT}. The roots (red circles) and critical points (blue crosses) of $p_n(z) = \prod_{l=1}^2(z-\xi_l)\prod_{j=1}^{n-1}(z-X_j)$ for increasing values of $n$, where the roots, $X_1, \ldots, X_{49}$, are chosen independently and uniformly on the unit circle. The gray circles are of radius $\frac{10}{n^2}$ and are centered at $\hat{w}_1^{(n)}$, $\hat{w}_2^{(n)}$ defined in \eqref{eqn:wHat}.
}
\label{fig:UnitCirc9x9}
\end{figure}


\begin{remark}
After an application of the Borel--Cantelli lemma, Theorem \ref{thm:multiOutCLT} can be combined with Theorem 2.9 of \cite{OW} to establish the following: when $\mu$ has compact support, almost surely, for $n$ sufficiently large, $w_1^{(n)}, \ldots, w_s^{(n)}$, characterized by \eqref{eqn:multiOutLoc}, are the only critical points of $p_n$ outside an $\eps$-neighborhood of $\supp(\mu) \cup M_{\mu}$. 
\end{remark}

In Section \ref{sec:microProof}, we provide a generalization of Theorem \ref{thm:multiOutCLT} to a situation where $p_n$ has a number of deterministic roots that may depend on $n$ (see Theorem \ref{thm:manyDet} below). The proofs of Theorems \ref{thm:multiInCLT} and \ref{thm:multiOutCLT} are based on a technical, deterministic argument that applies to cases where $X_1, \ldots, X_n$ are random variables that are not independent (see Theorem \ref{thm:detLoc}). To illustrate this point, we conclude the subsection with a result that demonstrates pairing between individual roots and critical points of $p_n$ when $p_n$ is the characteristic polynomial of a random matrix.

\begin{theorem} \label{thm:eig}
Fix $\eps > 0$ and $\lambda \in \mathbb{C}$ with $|\lambda| \geq 1 + 3 \eps$.  Let $M$ be an $n \times n$ random matrix whose entries are iid copies of a random variable with mean zero, unit variance, and finite fourth moment.  Let $A$ be an $n \times n$ deterministic matrix with operator norm $\|A\| = O(1)$, $\rank(A) = O(1)$\footnote{We continue to use asymptotic notation, such as $O$ and $o$, under the assumption that $n \to \infty$.  In this theorem, $n$ represents the dimension of the matrices $M$ and $A$.}, and whose only nonzero eigenvalue is $\lambda$.  Then almost surely, for $n$ sufficiently large, the characteristic polynomial\footnote{Here, $I$ denotes the identity matrix.} 
\[
p_n(z) := \det \left( zI - \frac{1}{\sqrt{n}} M  - A \right) 
\]
of $\frac{1}{\sqrt{n}} M  + A$ has a factorization 
$p_n(z) = (z - \xi) \prod_{i=1}^{n-1} (z - X_i)$,
where 
\begin{enumerate}[(i)]
\item \label{item:inside} The roots $X_1, \ldots, X_{n-1}$ lie inside the disk $B(0, 1 + 2\eps)$.  
\item \label{item:outside} The root $\xi$ lies outside the disk $B(0,1+2\eps)$ and satisfies $\xi = \lambda + o(1)$.  
\item \label{item:eigcp} $p_n$ contains a unique critical point, $w_\xi^{(n-1)}$, which satisfies 
\begin{equation} \label{eq:eigcp}
	\left| w_\xi^{(n-1)} - \xi + \frac{1}{n}\cdot\frac{1}{ \frac{1}{n-1} \sum_{i=1}^{n-1} \frac{1}{\xi - X_i} } \right| = O_{\lambda, \eps} \left( \frac{1}{n^2} \right). 
\end{equation}
and hence
\begin{equation} \label{eq:eigo1}
	w_\xi^{(n-1)} = \lambda + o(1) 
\end{equation}
\end{enumerate}
\end{theorem}
\begin{remark}
The conclusion in \eqref{eq:eigo1} can be deduced from properties \eqref{item:inside} and \eqref{item:outside} and Walsh's two circle theorem (see, for example, \cite[Theorem 4.1.1]{RS}).  However, the conclusion in \eqref{eq:eigcp} cannot be deduced from Walsh's two circle theorem and instead follows from Theorem \ref{thm:detLoc}. We prove Theorem \ref{thm:eig} in Appendix \ref{sec:appendixA}.
\end{remark}

\subsection{A local law for the critical points} \label{sec:locallaw}

In this subsection, we consider a local law that describes the behavior of the critical points of
\[ p_n(z) := \prod_{i=1}^n (z - X_i). \]
We begin with the case where $X_1, \ldots, X_n$ are arbitrary random variables (not assumed to be independent nor identically distributed) and then specialize our main result to several applications and examples.

\begin{theorem}[Local law] \label{thm:main}
Fix $C > 0$, and let $X_1, \ldots, X_n$ be complex-valued random variables (not necessarily independent nor identically distributed) which satisfy the following axioms.  
\begin{enumerate}[(i)]
\item (Upper bound) With overwhelming probability, 
$ \max_{1 \leq i \leq n} |X_i| \leq e^{n^C}, $
\item \label{item:cond:anti} (Anti-concentration) For every $a > 0$, there exists $b > 0$ such that
\begin{equation} \label{eq:anti}
	\left| \sum_{i=1}^n \frac{1}{Z - X_i} \right| \geq n^{-b} 
\end{equation}
with probability $1 - O_{a}(n^{-a})$, where $Z$ is uniformly distributed on $B(0, n^C)$, independent of $X_1, \ldots, X_n$. 
\end{enumerate}
Let $\varphi: \mathbb{C} \to \mathbb{R}$ be a twice continuously differentiable function (possibly depending on $n$) which is supported on $B(0, n^{C})$ and which satisfies the pointwise bound
\begin{equation} \label{eq:pointwise}
	| \lap \varphi(z) | \leq n^{C} 
\end{equation}
for all $z \in \mathbb{C}$.  Then, for every fixed $c > 0$ and every $\alpha > 0$, 
\[ \sum_{j=1}^{n-1} \varphi(w_j^{(n)}) = \sum_{i=1}^n \varphi(X_i) + O_{\alpha}(\| \lap \varphi \|_1 \log n) + O_{\alpha}(n^{-c}) \]
with probability $1 - O_{\alpha}(n^{-\alpha})$, where $w_1^{(n)}, \ldots, w_{n-1}^{(n)}$ are the critical points of the polynomial 
\[ p_n(z) := \prod_{i=1}^n (z - X_i) \]
and $\| \lap \varphi \|_1$ is the $L_1$-norm of $\lap \varphi$.  Here, the implicit constants in our asymptotic notation depend on $C, c$, and $\alpha$.  
\end{theorem}

\begin{remark} \label{rem:alt}
Condition \eqref{item:cond:anti} on the random variables $X_1, \ldots, X_n$ from Theorem \ref{thm:main} is implied by the following: 
\begin{enumerate}
\item[(ii')] for every $a > 0$, there exists $b > 0$ such that, for almost every $z \in B(0, n^C)$, 
\[ \left| \sum_{i=1}^n \frac{1}{z - X_i} \right| \geq n^{-b} \]
with probability $1 - O_{a}(n^{-a})$.  
\end{enumerate}
Indeed, the implication follows by simply conditioning on the random variable $Z$ (which avoids a set of Lebesgue measure zero with probability $1$).  
\end{remark}

The assumptions of Theorem \ref{thm:main} are fairly technical, and we derive some simpler conditions that guarantee when the hypotheses of Theorem \ref{thm:main} are met in Section \ref{sec:guarantee}.  We now specialize Theorem \ref{thm:main} to the case where $X_1, \ldots, X_n$ are independent random variables.  

\begin{theorem}[Local law for independent roots] \label{thm:indep}
Fix $C > 0$, and let $X_1, \ldots, X_n$ be independent complex-valued random variables which satisfy
$ \max_{1 \leq i \leq n} \E |X_i | \leq n^C. $
In addition, assume $X_1$ is absolutely continuous (with respect to Lebesgue measure on $\mathbb{C}$) and has density bounded by $n^C$.  Let $\varphi: \mathbb{C} \to \mathbb{R}$ be a twice continuously differentiable function (possibly depending on $n$) which is supported on $B(0, n^{C})$ and which satisfies the pointwise bound given in \eqref{eq:pointwise} for all $z \in \mathbb{C}$.  Then, for every fixed $c > 0$ and every $\alpha > 0$, 
\[ \sum_{j=1}^{n-1} \varphi(w_j^{(n)}) = \sum_{i=1}^n \varphi(X_i) + O_{\alpha}(\| \lap \varphi \|_1 \log n) + O_{\alpha}(n^{-c}) \]
with probability $1 - O_{\alpha}(n^{-\alpha})$, where $w_1^{(n)}, \ldots, w_{n-1}^{(n)}$ are the critical points of the polynomial 
$ p_n(z) := \prod_{i=1}^n (z - X_i) $
and $\| \lap \varphi \|_1$ is the $L_1$-norm of $\lap \varphi$.  Here, the implicit constants in our asymptotic notation depend on $C, c$, and $\alpha$.
\end{theorem}

Theorem \ref{thm:indep} can be viewed as a local version of Theorem \ref{thm:kabluchko} and \eqref{eq:kabluchko}.  Indeed, since the functions in the theorem above can depend on $n$, one can approximate an indicator function of Borel sets which changes with $n$.  
In addition, the error bound in Theorem \ref{thm:indep} is significantly better then the error term from \eqref{eq:kabluchko}.  

Interestingly, Theorem \ref{thm:indep} only requires a single root ($X_1)$ to actually be random; the rest may be deterministic.  In particular, since the density of $X_1$ is bounded by $n^C$, $X_1$ can itself be quite close to deterministic.   Obviously, though, the result fails for deterministic polynomials.  For example, consider $q_n(z) := z^n-1$.  The conclusion of Theorem \ref{thm:indep} fails for this polynomial since all of the critical points are located at the origin while the roots are the $n$-th roots of unity, located on the unit circle.  However, Theorem \ref{thm:indep} does apply to $p_n(z) := q_n(z) (z - X)$, where $X$ is uniformly distributed on $B(z_0, n^{-C/2})$ for any fixed $z_0 \in \mathbb{C}$.  Theorem \ref{thm:indep} strengthens Theorem 1.6 of \cite{BLR} for the empirical distribution associated with the zeros of $p'_n$ by providing a rate of convergence. As a consequence of Theorem \ref{thm:indep}, we have the following central limit theorem (CLT).  


\begin{theorem}[Central limit theorem for linear statistics] \label{thm:clt}
Let $X_1, X_2, \ldots$  be iid random variables which are absolutely continuous (with respect to Lebesgue measure on $\mathbb{C}$) and have a bounded density.  In addition, assume $\E|X_1| < \infty$.  Let $\varphi: \mathbb{C} \to \mathbb{R}$ be a twice continuously differentiable function with compact support which does not depend on $n$.  Then,
\[
\frac{1}{\sqrt{n}} \left(\sum_{j=1}^{n-1} \varphi(w^{(n)}_j) - \E\left[\sum_{j=1}^{n-1} \varphi(w^{(n)}_j)\right] \right) \longrightarrow N(0, v^2) \]
in distribution as $n \to \infty$, where $w^{(n)}_1, \ldots, w^{(n)}_{n-1}$ are the critical points of the polynomial 
$ p_n(z) := \prod_{i=1}^n (z - X_i) $
and $v^2$ is the variance of $\varphi(X_1)$.  
\end{theorem}

We now state a version of Theorem \ref{thm:indep} that applies when the function $\varphi$ is analytic.  While analyticity is a much more rigid assumption, the next result does not contain the extra factor of $\log n$ present in the error term from Theorem \ref{thm:indep}.  


\begin{theorem}[Local law for analytic test functions] \label{thm:analytic}
Fix $C, c, \eps > 0$.  Let $\mu$ be a probability measure on $\mathbb{C}$ supported on $B(0, C)$, and assume
\begin{equation} \label{eq:mmulowbnd}
	| m_\mu(z)| \geq c 
\end{equation}
for all $z \in \Gamma$, where $\Gamma$ is the boundary of $B(0,C+\eps)$.  Then for any function $\varphi$ (possibly depending on $n$), analytic in a neighborhood containing the closure of $B(0,C+\eps)$, one has
\[ \sum_{j=1}^{n-1} \varphi(w_j^{(n)}) = \sum_{i=1}^n \varphi(X_i) + O \left( \oint_\Gamma |\varphi(z)| |dz| \right), \]
where $w_1^{(n)}, \ldots, w_{n-1}^{(n)}$ are the critical points of the polynomial
$ p_n(z) := \prod_{i=1}^n (z - X_i) $
and $X_1, \ldots, X_n$ are iid random variables with distribution $\mu$.  Here, the implicit constants in our asymptotic notation depend on $C, c$, and $\eps$.  
\end{theorem}



\subsection{Guaranteeing the assumptions in the local law} \label{sec:guarantee}

In this section, we provide some criteria for assuring the assumptions in Theorem \ref{thm:main} are met. 

\begin{lemma}[Simple criterion for an upper bound] \label{lemma:critup}
Fix $C, \eps > 0$, and suppose $X_1, \ldots, X_n$ are complex-valued random variables (not necessarily independent nor identically distributed).  If 
$ \max_{1 \leq i \leq n} \E |X_i|^\eps \leq n^C, $
then
$ \max_{1 \leq i \leq n} |X_i| \leq e^{n^C} $
with overwhelming probability.  
\end{lemma}
\begin{proof}
As 
$ \Prob \left( \max_{1 \leq i \leq n} |X_i| > e^{n^C} \right) \leq \sum_{i=1}^n \Prob (|X_i| > e^{n^C} ), $
the claim follows from a simple application of Markov's inequality.  
\end{proof}

\begin{lemma}[Criterion for anti-concentration] \label{lemma:critanti}
Fix $C > 0$, and let $X_1, \ldots, X_n$ be complex-valued random variables such that $X_1$ is independent of $X_2, \ldots, X_n$.  In addition, assume $X_1$ is absolutely continuous (with respect to Lebesgue measure on $\mathbb{C}$) with density bounded by $n^C$, and suppose that $\E|X_1| \leq n^C$.  Then for every $a > 0$, there exists $b > 0$ such that
\[ \left| \sum_{i=1}^n \frac{1}{Z - X_i} \right| \geq n^{-b} \]
with probability $1 - O_{a}(n^{-a})$, where $Z$ is uniformly distributed on $B(0, n^C)$ and independent of $X_1, \ldots, X_n$.  
\end{lemma}

We prove Lemma \ref{lemma:critanti} in Appendix \ref{sec:appendixA}.


\subsection{Overview and outline}
The remainder of the paper is devoted to proving our main results. In Section \ref{sec:microProof}, we establish Theorems \ref{thm:multiInCLT}, \ref{thm:multiOutCLT}, and \ref{thm:eig} of Subsection \ref{sec:microscopic} by way of Theorem \ref{thm:detLoc} for deterministic polynomials, which we also use to prove a generalization to Theorem \ref{thm:multiOutCLT}. Section \ref{sec:locProof} contains the proofs of the local laws from Subsection \ref{sec:locallaw} including those for Theorems \ref{thm:main}, \ref{thm:indep}, \ref{thm:clt}, and \ref{thm:analytic}. We conclude the paper with a proof of Theorem \ref{thm:wasserstein} in Section \ref{sec:wassProof}.

There are two appendices that contain minor lemmata and supporting calculations. In Appendix \ref{sec:appendixA}, we provide Lemma \ref{lem:radSymMu} to simplify the computation of $m_\mu$ for radially symmetric distributions, we include calculations related to Example \ref{ex:twoDisks}, and we justify Lemma \ref{lemma:critanti}. Appendix \ref{sec:appendixB} contains some classical arguments that establish a Lindeberg CLT that we use to prove part of Theorem \ref{thm:multiInCLT}.

\section{Proof of results in Section \ref{sec:microscopic}} \label{sec:microProof}

The proofs of Theorems \ref{thm:multiInCLT}, \ref{thm:multiOutCLT}, and \ref{thm:eig} rely on the following theorem for deterministic polynomials.  

\begin{theorem}
	Suppose $\xi$ is a complex number, $\vec{X} = (X_1, X_2, \ldots, X_n)$ is a vector of complex numbers, and $C_1, C_2, k_\text{Lip}$ are positive values for which the following three conditions hold:
	\begin{enumerate}[(i)]
		\item \label{it:detLoc1} $C_1 \leq \abs{\frac{1}{n}\sum_{j=1}^n\frac{1}{\xi-X_j}} \leq C_2$;
		\item \label{it:detLoc2} The function $z \mapsto \frac{1}{n}\sum_{j=1}^n \frac{1}{z-X_j}$ is Lipschitz continuous with constant $k_\text{Lip}$ on the set $\set{z \in \mathbb{C} :\abs{z -\xi} \leq \frac{2}{C_1n}}$;
		\item \label{it:detLoc3} $\displaystyle\min_{1\leq j \leq n}\abs{\xi - X_j} > \frac{3}{C_1n}$.
	\end{enumerate}
	Then, if $C > 0$ and $n \in \mathbb{N}$ satisfy 
	\begin{equation} \
	C > \frac{8(1+2C_2^2)}{C_1^3}\quad \text{and} \quad n > 4C_2 \max\set{\frac{1}{C_1}, C(k_\text{Lip} + 1)},
	\end{equation}
	the polynomial $p_n(z) := (z-\xi)\prod_{j=1}^n(z-X_j)$ has exactly one critical point, $w_\xi^{(n)}$, that is within a distance of $\frac{3}{2C_1 n}$ of $\xi$, and  
	\begin{equation}
	\abs{w_\xi^{(n)}-\xi + \frac{1}{n+1}\frac{1}{\frac{1}{n}\sum_{j=1}^n\frac{1}{\xi-X_j}}} < \frac{C(k_\text{Lip}+1)}{n^2}.
	\label{eqn:detLocCP}
	\end{equation}
	\label{thm:detLoc}
\end{theorem}
We remark that criteria \eqref{it:detLoc1} and \eqref{it:detLoc2} appear relevant in view of \eqref{eqn:whyCauchy} and its accompanying heuristic. 
Assumption \eqref{it:detLoc3} helps to guarantee that $p_n(z)$ has only one critical point that is within order $O(1/n)$ of $\xi$, but with respect to establishing equation \eqref{eqn:detLocCP}, \eqref{it:detLoc3} is likely an artificial constraint related to the use of Rouch\'{e}'s theorem in the proof. We prove Theorem \ref{thm:detLoc} in the next subsection.

\subsection{Proof of Theorem \ref{thm:detLoc}}

Our strategy is to compare $p_n(z)$ to the simpler polynomial \[\widetilde{p}(z) = (z-\xi)(z-Y_n)^n,\] where 
\[
Y_n := \xi - \frac{1}{\frac{1}{n}\sum_{j=1}^n\frac{1}{\xi - X_j}}
\]
is chosen so that near $\xi$, the logarithmic derivatives
\[
L_n(z) := \frac{1}{z-\xi} + \sum_{j=1}^n\frac{1}{z-X_j}\quad \text{and} \quad \widetilde{L}_n(z) := \frac{1}{z-\xi} + \frac{n}{z-Y_n}
\]
of $p_n$ and $\widetilde{p}_n$, respectively, are close to each other. In particular, we will use Rouch\'{e}'s theorem to show that $L_n$ and $\widetilde{L}_n$ both have exactly one zero in each of the nested open balls
\[
D_n^\text{sm}:= B\left(c_n, \frac{C(k_\text{Lip}+1)}{n^2}\right)\quad \text{and}\quad D_n^\text{lg} := B\left(\xi, \frac{3}{2C_1n}\right),
\]
where
\[
c_n := \xi- \frac{1}{n+1}\frac{1}{\frac{1}{n}\sum_{j=1}^n\frac{1}{\xi-X_j}}
\]
can be easily verified to be a root of $\widetilde{L}_n$. By ``clearing the denominators'' we will conclude that $p_n$ has exactly one critical point in each of the two balls. The lemma below establishes a few  key facts that we frequently reference throughout the proof.
\begin{lemma}
Under the assumptions of Theorem \ref{thm:detLoc}:
\begin{enumerate}[(i)]
\item \label{it:smBall} For $\abs{z - c_n} \leq \frac{C(k_\text{Lip}+1)}{n^2}$:
\begin{align*}
&\frac{C(k_\text{Lip}+1)}{n^2} < \abs{z-\xi} < \frac{5}{4C_1n},\ \text{so $D_n^\text{sm} \subset D_n^\text{lg}$};\\
&\frac{C(k_\text{Lip}+1)}{n^2} < \abs{z-Y_n} < \frac{2}{C_1};\\
&\frac{C(k_\text{Lip}+1)}{n^2}<\frac{1}{C_1 n}  < \abs{z-X_j}\ \text{for $1\leq j \leq n$}.\\
\end{align*}
\item \label{it:lgBall} For $\abs{z - \xi} \leq \frac{3}{2nC_1}$:
\begin{align*}
&\frac{3}{2C_1n} < \abs{z-Y_n} < \frac{5}{2C_1};\\
&\frac{3}{2C_1n} < \abs{z-X_j}\ \text{for $1\leq j \leq n$};\\
&\frac{1}{2C_1n} < \abs{z-c_n}\ \text{if $\abs{z-\xi} = \frac{3}{2nC_1}$}.
\end{align*}
\end{enumerate}
\label{lem:detIneq}
\end{lemma}
\begin{proof}
To prove \eqref{it:smBall}, suppose $\abs{z - c_n} \leq \frac{C(k_\text{Lip}+1)}{n^2}$. By the triangle inequality, we have
\[
\abs{z-\xi} \geq \abs{c_n - \xi} - \abs{z-c_n} \geq\frac{1}{(n+1)C_2} - \frac{C(k_\text{Lip}+1)}{n^2} \geq \frac{1}{2nC_2} - \frac{C(k_\text{Lip}+1)}{n^2},
\]
and by the hypothesis that $n > 4C_2C(k_\text{Lip}+1)$, it follows that
\[
\abs{z-\xi} > \frac{1}{2nC_2} - \frac{1}{4nC_2} = \frac{1}{4nC_2} > \frac{C(k_\text{Lip}+1)}{n^2}.
\]
On the other hand, we have
\[
\abs{z-\xi} \leq \abs{z-c_n} + \abs{c_n - \xi} \leq \frac{C(k_\text{Lip}+1)}{n^2} + \frac{1}{(n+1)C_1},
\]
and the assumption $n > 4C_2C(k_\text{Lip}+1)$ guarantees that
\[
\abs{z-\xi} < \frac{1}{4nC_1} + \frac{1}{(n+1)C_1} < \frac{5}{4C_1n}
\]
(note: $C_1 \leq C_2$). This establishes the first inequality. The second follows from nearly identical reasoning; we omit the details. To achieve the inequalities $\frac{1}{C_1 n} < \abs{z-X_j}$, we use $\abs{z-\xi} < \frac{5}{4C_1n}$, which we just proved, and the assumption that $\min_{1 \leq j \leq n}\abs{\xi - X_j} > \frac{3}{C_1n}$. Indeed, for $1 \leq j \leq n$, the triangle inequality yields
\[
\abs{z-X_j} \geq \abs{\xi-X_j} - \abs{z-\xi} > \frac{3}{C_1n} - \frac{5}{4C_1n} > \frac{1}{C_2n} > \frac{C(k_\text{Lip}+1)}{n^2}.
\]
This completes the proof of part \eqref{it:smBall}. Part \eqref{it:lgBall} follows from nearly identical reasoning. Note that the assumption $n > 4C_2/C_1$ is useful for achieving the lower bound on $\abs{z-Y_n}$. We omit the remaining details.
\end{proof}

The lower bounds in Lemma \ref{lem:detIneq} imply that under the assumptions of Theorem \ref{thm:detLoc}, $L_n(z)$ and $\widetilde{L}_n(z)$ are holomorphic on the domain $D_n^\text{sm}$ and that $(z-\xi)L_n(z)$ and $(z-\xi)\widetilde{L}_n(z)$ are holomorphic on the domain $D_n^\text{lg}$.
We will show that under the same assumptions, $\abs{L_n(z) - \widetilde{L}_n(z)} < \abs{\widetilde{L}_n(z)}$ for $z$ in the boundaries of $D_n^\text{sm}$ and $D_n^\text{lg}$ in order to justify Rouch\'{e}'s theorem. To that end, assume the hypotheses of Theorem \ref{thm:detLoc} and let $z \in \partial D_n^\text{sm} \cup \partial D_n^\text{lg}$. 
Then, the triangle inequality implies 
\begin{align*}
\abs{L_n(z) - \widetilde{L}_n(z)} &= \abs{\sum_{j=1}^n \frac{1}{z-X_j} - \frac{n}{z-Y_n}}\\
&\leq\abs{\sum_{j=1}^n \frac{1}{z-X_j} - \sum_{j=1}^n \frac{1}{\xi-X_j}} + \abs{\sum_{j=1}^n \frac{1}{\xi-X_j} - \frac{1}{\frac{1}{n}z-\frac{1}{n}Y_n}}\\
&\leq nk_\text{Lip}\abs{z-\xi} + \abs{\sum_{j=1}^n \frac{1}{\xi-X_j} - \frac{1}{\frac{1}{n}(z-\xi)+\frac{1}{\sum_{j=1}^n\frac{1}{\xi-X_j}}}},
\end{align*}
where we have used hypothesis \eqref{it:detLoc2} of Theorem \ref{thm:detLoc} to bound the first term on the left. By factoring $\abs{\sum_{j=1}^n \frac{1}{\xi-X_j}}$ from both terms in the right summand, we obtain 
\[
\abs{L_n(z) - \widetilde{L}_n(z)} \leq nk_\text{Lip}\abs{z-\xi} + n\cdot\abs{\frac{1}{n}\sum_{j=1}^n \frac{1}{\xi-X_j}}\cdot\abs{1 - \frac{1}{(z-\xi)\frac{1}{n}\sum_{j=1}^n\frac{1}{\xi-X_j}+1}},
\]
and then, combining the fractions, factoring out another $\abs{\sum_{j=1}^n \frac{1}{\xi-X_j}}$, and applying hypothesis \eqref{it:detLoc1}  of Theorem \ref{thm:detLoc} twice yields
\begin{align*}
\abs{L_n(z) - \widetilde{L}_n(z)} &\leq nk_\text{Lip}\abs{z-\xi}  + nC_2\cdot\abs{\frac{(z-\xi)\frac{1}{n}\sum_{j=1}^n\frac{1}{\xi-X_j}}{(z-\xi)\frac{1}{n}\sum_{j=1}^n\frac{1}{\xi-X_j}+1}}\\
&\leq nk_\text{Lip}\abs{z-\xi}  + nC_2^2\cdot\frac{\abs{z-\xi}}{\abs{(z-\xi)\frac{1}{n}\sum_{j=1}^n\frac{1}{\xi-X_j}+1}}.
\end{align*}
Finally, we can use the reverse triangle inequality and hypothesis \eqref{it:detLoc1}  of Theorem \ref{thm:detLoc} to show
\begin{equation}
\begin{aligned}
\abs{L_n(z) - \widetilde{L}_n(z)} &\leq nk_\text{Lip}\abs{z-\xi}  + nC_2^2\abs{z-\xi} \cdot\frac{1}{1-\abs{(z-\xi)\frac{1}{n}\sum_{j=1}^n\frac{1}{\xi-X_j}}}\\
&\leq n\abs{z-\xi}\left(k_\text{Lip}  + \frac{C_2^2}{1-\abs{z-\xi}C_2}\right).
\end{aligned}
\label{eqn:detRouche1}
\end{equation}
At this point, we split the argument into two cases: $\abs{z-c_n} = \frac{C(k_\text{Lip}+1)}{n^2}$ and $\abs{z-\xi} = \frac{3}{2nC_1}$. In the first case, Lemma \ref{lem:detIneq} guarantees that $\abs{z-\xi} < \frac{2}{nC_1}$, and the hypotheses of Theorem \ref{thm:detLoc} require that $\frac{1}{2} > \frac{2C_2}{nC_1}$, so we obtain
\begin{equation}
\abs{L_n(z) - \widetilde{L}_n(z)} < \frac{2}{C_1}\left(k_\text{Lip}  + 2C_2^2\right) \leq \frac{2}{C_1}(k_\text{Lip}+1)(1 + 2C_2^2).
\label{eqn:LnSmall}
\end{equation}
On the other hand,
\begin{align*}
\abs{\widetilde{L}_n(z)} &= \abs{\frac{1}{z-\xi} + \frac{n}{z-Y_n}}\\
&= \abs{\frac{z-Y_n + n(z-\xi)}{(z-\xi)(z-Y_n)}}\\
&= (n+1)\cdot \abs{z-\xi}^{-1} \cdot \abs{z-Y_n}^{-1} \cdot \abs{z - c_n}\\
&> n \cdot \frac{nC_1}{2} \cdot \frac{C_1}{2} \cdot \frac{C(k_\text{Lip}+1)}{n^2},
\end{align*}
where the last inequality follows from Lemma \ref{lem:detIneq}. One of the assumptions in Theorem \ref{thm:detLoc} is that $C > \frac{8(1+2C_2^2)}{C_1^3}$, so
\begin{equation}
\abs{\widetilde{L}_n(z)} > \frac{C_1^2}{4}(k_\text{Lip}+1)\frac{8(1+2C_2^2)}{C_1^3} = \frac{2}{C_1}(k_\text{Lip}+1)(1 + 2C_2^2).
\label{eqn:LnBig}
\end{equation}
Combining \eqref{eqn:LnSmall} and \eqref{eqn:LnBig} yields $\abs{L_n(z) - \widetilde{L}_n(z)} < \abs{\widetilde{L}_n(z)}$ for $z$ in the boundary of $D_n^\text{sm}$. In addition, recall (Lemma \ref{lem:detIneq} part \eqref{it:lgBall}) that $L_n(z)$ and $\widetilde{L}(z)$ are holomorphic on the domain $D_n^\text{sm}$, so Rouch\'{e}'s theorem guarantees that $L_n(z)$ and $\widetilde{L}_n(z)$ have the same number of zeros inside $D_n^\text{sm}$. Since $c_n$ is the unique zero of $\widetilde{L}_n(z)$ in $D_n^\text{sm}$, we conclude that $L_n(z)$ has exactly one zero, $w_\xi^{(n)}$, in $D_n^\text{sm}$.  Furthermore, 
$
L_n(z) = \frac{p'_n(z)}{p_n(z)}
$
(which is analytic for $z \in D_n^\text{sm}$ by \eqref{it:smBall} of Lemma \ref{lem:detIneq}), so the zeros of $L_n(z)$ in $D_n^\text{sm}$ are the same as the critical points of $p_n(z)$ in $D_n^\text{sm}$. We conclude that $p_n(z)$ has exactly one critical point in $D_n^\text{sm}$.

Lemma \ref{lem:detIneq} shows that $D_n^\text{sm} \subset D_n^\text{lg}$, so it remains to establish that $p_n(z)$ also has exactly one critical point in $D_n^\text{lg}$, for then, the critical point in both domains must be the same one. Continuing from \eqref{eqn:detRouche1}, in the case where $\abs{z-\xi} = \frac{3}{2C_1 n}$, we obtain
\begin{equation}
\abs{L_n(z) - \widetilde{L}_n(z)} < \frac{3}{2C_1}(k_\text{Lip}+2C_2^2) \leq \frac{3}{2C_1}(k_\text{Lip}+1)(1+2C_2^2),
\label{eqn:LnSmall2}
\end{equation}
where we have once again used the assumption that $\frac{1}{2} \geq \frac{2C_2}{nC_1}$. Similarly to above, we also have
\begin{align*}
\abs{\widetilde{L}_n(z)} &= \abs{\frac{1}{z-\xi} + \frac{n}{z-Y_n}}\\
&= \abs{\frac{z-Y_n + n(z-\xi)}{(z-\xi)(z-Y_n)}}\\
&= (n+1)\cdot \abs{z-\xi}^{-1} \cdot \abs{z-Y_n}^{-1} \cdot \abs{z - c_n}\\
&> n \cdot \frac{2C_1n}{3} \cdot \frac{2C_1}{5} \cdot \frac{1}{2C_1n}\\
&= \frac{2C_1n}{15}
\end{align*}
where the inequality follows from Lemma \ref{lem:detIneq}, \eqref{it:lgBall}. From the assumptions on $n$ and $C$ in Theorem \ref{thm:detLoc}, it follows that 
\[
n > 4C_2C(k_\text{Lip}+1) > \frac{32(1+2C_2^2)(k_\text{Lip}+1)}{C_1^2}\cdot\frac{C_2}{C_1} \geq \frac{32(1+2C_2^2)(k_\text{Lip}+1)}{C_1^2}
\]
(recall $C_1 \leq C_2$), so in the case when $\abs{z-\xi} = \frac{3}{2C_1n}$, 
\begin{equation}
\abs{\widetilde{L}_n(z)} > \frac{64}{15C_1}(k_\text{Lip}+1)(1+2C_2^2).
\label{eqn:LnBig2}
\end{equation}
Combining \eqref{eqn:LnSmall2} and \eqref{eqn:LnBig2} yields $\abs{L_n(z) - \widetilde{L}_n(z)} < \abs{\widetilde{L}_n(z)}$ for $z$ in the boundary of $D_n^\text{lg}$. Consequently, for $z \in \partial D_n^\text{lg}$,
\[
\abs{(z-\xi)L_n(z) - (z-\xi)\widetilde{L}_n(z)} < \abs{(z-\xi)\widetilde{L}_n(z)},
\]
and since $(z-\xi)L_n(z)$, $(z-\xi)\widetilde{L}_n(z)$ are holomorphic in $D_n^\text{lg}$ by Lemma \ref{lem:detIneq}, \eqref{it:lgBall}, Rouch\'{e}'s theorem guarantees that $(z-\xi)L_n(z)$, $(z-\xi)\widetilde{L}_n(z)$ have the same number of zeros in $D_n^\text{lg}$. In fact, $(z-\xi)\widetilde{L}_n(z)$ has exactly one zero in $D_n^\text{lg}$, namely $c_n$, so
\[
(z-\xi)L_n(z) = \frac{p_n'(z)}{\prod_{j=1}^n(z-X_j)}
\]
has exactly one zero in $D_n^\text{lg}$, too. (Note: by Lemma \ref{lem:detIneq}, \eqref{it:smBall}, $D_n^\text{sm} \subset D_n^\text{lg}$.) Hence, $p_n'(z)$ has exactly one root in $D_n^\text{lg}$, and as we showed above, this root lies in $D_n^\text{sm}$. The proof of Theorem \ref{thm:detLoc} is complete.

In the remainder of this section, we use Theorem \ref{thm:detLoc} to prove Theorems \ref{thm:multiInCLT},  \ref{thm:multiOutCLT} and \ref{thm:eig}. We also include a subsection where we sketch how the arguments could be modified to prove Theorem \ref{thm:manyDet}, which generalizes part of Theorem \ref{thm:multiOutCLT} to situations where $p_n$ has many deterministic roots. When  $\xi \in \supp(\mu)$, it is difficult to control $\frac{1}{n}\sum_{j=1}^n\frac{1}{\xi-X_j}$, so we start with the proof of Theorem \ref{thm:multiOutCLT}, which is more straightforward than the justification of Theorem \ref{thm:multiInCLT}.

\subsection{Proof of Theorem \ref{thm:multiOutCLT}}\label{sec:detLoc}

We begin by establishing equation \eqref{eqn:multiOutLoc} via Theorem \ref{thm:detLoc}. To that end, we consider $\set{\xi_l}_{l=1}^s$, one at a time, letting each in turn play the role of $\xi$ in the statement of Theorem \ref{thm:detLoc}. Fix $\xi_l$, $1 \leq l \leq s$. We will show that for large $n$, on the complement of the ``bad'' event 
\[
E_n^l := \set{\abs{\frac{1}{n+1-s}\sum_{j=1}^{n+1-s}\frac{1}{\xi_l - X_j} - m_\mu(\xi_l)} \geq \frac{\abs{m_\mu(\xi_l)}}{3}},
\]
the hypotheses of Theorem \ref{thm:detLoc} are satisfied with $\xi = \xi_l$,
\[
\vec{X}  = (\xi_1, \ldots, \xi_{l-1}, \xi_{l+1}, \ldots, \xi_s, X_1,\ldots, X_{n+1-s}),
\]
and the positive constants
\begin{equation}
C_{1,l} := \frac{\abs{m_\mu(\xi_l)}}{2},\ C_{2,l} := \frac{3\abs{m_\mu(\xi_l)}}{2},\ k_{\text{Lip},l} := \frac{9}{\dist(\xi_l, \supp(\mu) \cup\set{\xi_j: j \neq l})^2}.
\label{eqn:CkOut}
\end{equation}
(Here, $\dist(z, D) := \inf_{w \in D} |z - w|$ is the distance from $z \in \mathbb{C}$ to a set $D \subset \mathbb{C}$.)

For large $n$, on the complement of $E_n^l$, 
\begin{equation}
\begin{aligned}
&\abs{\frac{1}{n}\left(\sum_{\substack{j=1\\j \neq l}}^s\frac{1}{\xi_l-\xi_j} + \sum_{j=1}^{n+1-s}\frac{1}{\xi_l-X_j}\right)}\\
 &\qquad\qquad\leq \abs{\frac{1}{n}\sum_{\substack{j=1\\j\neq l}}^s\frac{1}{\xi_l-\xi_j}} + \frac{n+1-s}{n}\abs{\frac{1}{n+1-s}\sum_{j=1}^{n+1-s}\frac{1}{\xi_l-X_j}}\\
&\qquad\qquad\leq o_l(1) + \abs{\frac{1}{n+1-s}\sum_{j=1}^{n+1-s}\frac{1}{\xi_l-X_j}-m_\mu(\xi_l)} + \abs{m_\mu(\xi_l)}\leq C_{2,l}
\end{aligned}
\label{eqn:multiOutC2}
\end{equation}
(The last inequality holds for large $n$.) Similarly, for large $n$, on the event $\left(E_n^l\right)^c$, 
\begin{equation}
\abs{\frac{1}{n}\left(\sum_{\substack{j=1\\j \neq l}}^s\frac{1}{\xi_l-\xi_j} + \sum_{j=1}^{n+1-s}\frac{1}{\xi_l-X_j}\right)} \geq C_{1,l},
\label{eqn:multiOutC1}
\end{equation}
and condition \eqref{it:detLoc1} of Theorem \ref{thm:detLoc} follows from equations \eqref{eqn:multiOutC2} and \eqref{eqn:multiOutC1}. If $n$ is chosen large enough that
\[
\eps_l:=\dist(\xi_l, \supp(\mu) \cup\set{\xi_j : j \neq l}) > \frac{3}{C_{1,l}n},
\]
then condition \eqref{it:detLoc3} of Theorem \ref{thm:detLoc} holds, and for $\abs{z - \xi_l} \leq \frac{2}{C_{1,l}n}$, 
\begin{align*}
\min_{1\leq j \leq n+1-s}\abs{z-X_j} &\geq \min_{1\leq j \leq n+1-s}\abs{\xi_l-X_j}-\abs{z-\xi_l}\geq \eps_l-\frac{2}{C_{1,l}n}> \frac{\eps_l}{3},\\
\min_{j\neq l}\abs{z-\xi_j} &\geq \min_{j \neq l}\abs{\xi_l-\xi_j}-\abs{z-\xi_l}\geq \eps_l-\frac{2}{C_{1,l}n}> \frac{\eps_l}{3}.
\end{align*}
In particular, this shows that for positive integers $n > 3(C_1\eps_l)^{-1}$ and complex numbers $z, w \in \set{z:\abs{z - \xi} \leq \frac{2}{C_{1,l}n}}$,
\begin{align*}
&\abs{\frac{1}{n}\left(\sum_{\substack{j=1\\j\neq l}}^{s}\frac{1}{z-\xi_j} + \sum_{j=1}^{n+1-s}\frac{1}{z-X_j}\right) - \frac{1}{n}\left(\sum_{\substack{j=1\\j\neq l}}^{s}\frac{1}{w-\xi_j}+\sum_{j=1}^{n+1-s}\frac{1}{w - X_j}\right)}\\
&\qquad= \frac{1}{n}\abs{\sum_{\substack{j=1\\j\neq l}}^{s} \frac{w-z}{(z-\xi_j)(w-\xi_j)}+\sum_{j=1}^{n+1-s}\frac{w-z}{(z-X_j)(w-X_j)}}\\
&\qquad\leq \abs{w-z}\cdot\frac{9}{\eps^2_l}= k_{\text{Lip},l}\cdot\abs{w-z},
\end{align*}
which implies condition \eqref{it:detLoc2} of Theorem \ref{thm:detLoc}.

Now, fix any $C > \max_{1\leq l\leq s}\frac{8(1+2C_{2,l}^2)}{C_{1,l}^3}$. If $n$ is a natural number large enough to guarantee inequalities \eqref{eqn:multiOutC2} and \eqref{eqn:multiOutC1} for $1 \leq l \leq s$ and that satisfies
\begin{equation}
n > \max\set{\frac{4C_{2,l}}{C_{1,l}},4C_{2,l}C(k_{\text{Lip},l}+1), \frac{3}{C_{1,l}\eps_l}: 1 \leq l \leq s},
\label{eqn:nMaxOutLoc}
\end{equation}
Theorem \ref{thm:detLoc} guarantees that on the complement of $\cup_{l=1}^sE_n^l$, the polynomial $p_n$ has $s$ critical points, $w_1^{(n)}, \ldots,w_s^{(n)}$, such that for $1 \leq l \leq s$, $w_l^{(n)}$ is the unique critical point of $p_n$ that is within a distance of $\frac{3}{\abs{m_\mu(\xi_l)}n}$ of $\xi_l$, and 
\begin{equation}
\abs{w_l^{(n)}-\xi_l + \frac{1}{n+1}\frac{n}{\sum_{j\neq l} \frac{1}{\xi_l-\xi_j} + \sum_{j=1}^{n+1-s}\frac{1}{\xi_l-X_j}}} < \frac{C(k_{\text{Lip},l}+1)}{n^2}.
\label{eqn:s>1}
\end{equation}
(Note that for large $n$, $w_1^{(n)}, \ldots, w_s^{(n)}$ are distinct because $\xi_1, \ldots, \xi_s$ are distinct and \eqref{eqn:s>1} implies $w_l^{(n)}\to\xi_l$ for $1 \leq l \leq s$.)
We complete our justification of \eqref{eqn:multiOutLoc} from Theorem \ref{thm:multiOutCLT} by choosing $C_{\mu,\vec{\xi}}$ larger than $\max_{l}C(k_{\text{Lip},l} + 1)$ and applying Hoeffding's inequality to the bounded random variables $(\xi_l-X_j)^{-1}$ to achieve the desired control over $\P(\cup_l E_n^l)$. More specifically, since $\xi_l \notin \supp(\mu)$ for $1 \leq l \leq s$, the random variables $Y_j^l:=\left(\xi_l - X_j\right)^{-1}$ are almost surely uniformly bounded by $K_l:=\dist(\xi_l,\supp(\mu))^{-1}$, and  the following version of Hoeffding's inequality applies with $t_l := \frac{\abs{m_\mu(\xi_l)}}{3}$.
\begin{lemma}[Hoeffding's inequality for complex-valued random variable; Lemma 3.1 from \cite{OW})]
\label{lem:Hoeffding}
Let $Y_1, \ldots, Y_n$ be iid complex-valued random variables which satisfy $|Y_j| \leq K$ almost surely for some $K > 0$.  Then there exist absolute constants $C, c>0$ such that
\[
\Prob \left( \left| \frac{1}{n} \sum_{j=1}^n Y_j - \frac{1}{n} \E \left[ \sum_{j=1}^n Y_j \right] \right| \geq t \right) \leq C \exp \left( - c n t^2 / K^2 \right),
\]
for every $t > 0$.  
\end{lemma}
By Lemma \ref{lem:Hoeffding}, we can find $C,c_{\mu,\vec{\xi}}>0$ such that $\cup_l E_n^l$ occurs with probability at least $1-C\exp(-c_{\mu,\vec{\xi}}\,n)$ as is desired.

We have established with overwhelming probability the existence of the critical points $w_1^{(n)},\ldots, w_s^{(n)}$ characterized by \eqref{eqn:multiOutLoc}. It remains to show that they satisfy the convergence in \eqref{eqn:multiOutCLT}. To that end, apply the Borel--Cantelli Lemma to the events $\cup_lE_n^l$ to see that almost surely, for large enough $n$, $w_l^{(n)}$ satisfies \eqref{eqn:multiOutLoc} for $1 \leq l \leq n$. It follows that with probability $1$, for sufficiently large $n$ and any $l$, $1 \leq l \leq s$, 
\begin{equation}
\begin{aligned}
&\sqrt{n}(n+1)\cdot m_\mu(\xi_l)^2\cdot\left(w_l^{(n)} - \xi_l + \frac{1}{n+1}\frac{1}{m_\mu(\xi_l)}\right)\\
&\quad=m_\mu(\xi_l)^2\cdot\sqrt{n}\left(\frac{1}{m_\mu(\xi_l)} - \frac{n}{\sum_{j\neq l}\frac{1}{\xi_l - \xi_j}+ \sum_{j=1}^{n+1-s}\frac{1}{\xi_l-X_j}}\right) + O_{\mu,\vec{\xi}}(n^{-1/2})\\
&\quad=m_\mu(\xi_l)^2\cdot\sqrt{n}\left(\frac{1}{m_\mu(\xi_l)} - \frac{1}{\frac{1}{n}\sum_{j=1}^{n}\frac{1}{\xi_l-X_j} + O_{\mu,\vec{\xi}}(1/n)}\right) + O_{\mu,\vec{\xi}}(n^{-1/2})\\
&\quad=\frac{m_\mu(\xi_l)}{\frac{1}{n}\sum_{j=1}^n\frac{1}{\xi_l-X_j}+ O_{\mu,\vec{\xi}}(1/n)} \cdot \sqrt{n}\left(\frac{1}{n}\sum_{j=1}^n\frac{1}{\xi_l-X_j} - m_\mu(\xi_l)\right) + O_{\mu,\vec{\xi}}(n^{-1/2}).
\end{aligned}
\label{eqn:outCLTCalc}
\end{equation}
In the case $s > 1$, we have used that
\[
\max_{1 \leq l \leq s}\abs{\sum_{j\neq l}\frac{1}{\xi_l-\xi_k}  - \sum_{j=n+2-s}^n\frac{1}{\xi_l - X_j}}= O_{\mu,\vec{\xi}}(1).
\]
 Now, we will use the Cram\'{e}r--Wold device (see e.g. Theorem 29.4 in \cite{Bill}) to show the convergence \eqref{eqn:multiOutCLT}. 
To start, let $t_1, \ldots, t_s, r_1, \ldots, r_s$ be arbitrary real numbers and define the random variables 
\begin{align*}
Y_{n,l} &:= n^{3/2}\cdot m_\mu(\xi_l)^2\cdot\left(w_l^{(n)} - \xi_l + \frac{1}{n+1}\frac{1}{m_\mu(\xi_l)}\right),\\
Z_{l,j} &:= \Re\left(\frac{1}{\xi_l-X_j}\right),\\
W_{l,j} &:= \Im\left(\frac{1}{\xi_l-X_j}\right),
\end{align*}
for $1 \leq l \leq s$. By \eqref{eqn:outCLTCalc}, we have, with probability tending to 1,
\begin{equation}
\begin{aligned}
Y_{n,l}&=\frac{n}{n+1}\frac{m_\mu(\xi_l)}{\frac{1}{n}\sum_{j=1}^n\frac{1}{\xi_l-X_j} + o(1)} \cdot \sqrt{n}\left(\frac{1}{n}\sum_{j=1}^n\frac{1}{\xi_l-X_j} - m_\mu(\xi_l)\right) + O\left(\frac{1}{\sqrt{n}}\right)\\
&=(1+o(1))\sqrt{n}\left(\frac{1}{n}\sum_{j=1}^n\frac{1}{\xi_l-X_j} - m_\mu(\xi_l)\right) + O\left(\frac{1}{\sqrt{n}}\right)\\
&=\sqrt{n}\left(\frac{1}{n}\sum_{j=1}^n\frac{1}{\xi_l-X_j} - m_\mu(\xi_l)\right) +o(1),
\end{aligned}
\label{eqn:YnlCalc}
\end{equation}
where all of the implied constants depend on $\xi_1, \ldots, \xi_s$ and $\mu$, and we have made ample use of Slutsky's theorem (see e.g. Theorem 11.4 from \cite{G}). To obtain the last line, we also used the classical CLT (see e.g. Theorem 29.5 from \cite{Bill}) in conjunction with Slutsky's theorem.  If we take linear combinations of the real and imaginary parts of $Y_{n,l}$, we obtain that with probability at least $1-o(1)$,
\begin{align*}
&\sum_{l=1}^st_l\Re(Y_{n,l}) + \sum_{l=1}^sr_l\Im(Y_{n,l})\\
&\quad=\sqrt{n}\left(\frac{1}{n}\sum_{j=1}^n\sum_{l=1}^s\left[t_lZ_{l,j} + r_lW_{l,j}- t_l\Re\left(m_\mu(\xi_l)\right) - t_l\Im\left(m_\mu(\xi_l)\right) \right]\right) +o(1),
\end{align*}
which converges by the classical CLT (and Slutsky's theorem) in distribution to a normally distributed random variable with mean $0$ and variance 
$
\var\left(\sum_{l=1}^s\left[t_lZ_{l,1} + r_lW_{l,1}\right]\right).
$
This limiting distribution is the same as the distribution of the random variable 
$
\sum_{l=1}^s\left[t_l\Re(N_l) + r_l\Im(N_l)\right]
$
with covariance structure given by \eqref{eqn:multiOutCov}, so by the Cram\'{e}r--Wold strategy, the proof of Theorem \ref{thm:multiOutCLT} is complete.

The next subsection illustrates how to modify the argument above to prove a generalization of Theorem \ref{thm:manyDet} to the case where $p_n$ has a number of deterministic roots that may grow with $n$.

\subsection{Generalization of Theorem \ref{thm:multiOutCLT}}
The following result shows how Theorem \ref{thm:detLoc} could be used to locate the critical points near a number of outlying deterministic roots that is allowed to depend on $n$. Compare the following theorem to Theorem 2 in \cite{H3}. Both theorems discuss the pairing between $s_n$ roots and critical points of $p_n$, where $s_n = o(n)$ is allowed to depend on $n$. Theorem \ref{thm:manyDet} describes the locations of the critical points with higher precision than Theorem 2 of \cite{H3}, however our theorem requires that the deterministic roots $\xi_1, \ldots, \xi_{s_n}$ be outside the support of $\mu$, while Theorem 2 in \cite{H3} doesn't make this restriction. 

\begin{theorem}[Locations of critical points when $p_n$ has many deterministic roots.]
	Suppose $X_1, X_2, \ldots$ are iid complex-valued random variables with distribution $\mu$, let $\xi_1, \xi_2,\ldots$ be fixed deterministic values, let $s_n, l_n, a_n$ be positive integers less than $n$, and fix $\eps, L > 0$, so that all of these together satisfy:
	\begin{enumerate}[(i)]
		\item \label{it:manyDet1} $1 \leq s_n \leq l_n = o(n)$, $a_nl_n = o(n)$, $a_n = o(\sqrt{n})$;
		\item \label{it:manyDet2}  $\min\set{\abs{m_\mu(\xi_l)}: 1 \leq l \leq s_n} \geq \eps$ and $\max\set{\abs{m_\mu(\xi_l)}: 1 \leq l \leq s_n} \leq L$;
		\item \label{it:manyDet3} $\min\set{\abs{\xi_l - x} : 1\leq l \leq s_n,\ x \in \supp(\mu) \cup \set{\xi_j}_{j=1, j\neq l}^{l_n} } > \frac{6}{\eps\cdot a_n}$.
	\end{enumerate}
	Then, there exist constants $C,c_{\mu,\eps,L}, C_{\mu, \eps,L}>0$ so that 
	with probability at least $1-C\cdot s_n\exp(-c_{\mu,\eps,L}\cdot n/a_n^2)$, the polynomial
	\[p_n(z) = \prod_{l=1}^{l_n}(z-\xi_l)\prod_{j=1}^{n+1-l_n}(z-X_j)\]
	has $s_n$ critical points, $w_1^{(n)}, \ldots, w_{s_n}^{(n)}$, such that for $1 \leq l \leq s_n$, $w_l^{(n)}$ is the unique critical point of $p_n$ within $\frac{3}{2\eps n}$ of $\xi_l$ and
	\begin{equation}
	\max_{1\leq l \leq s_n}\abs{w_l^{(n)} - \xi_l + \frac{1}{n+1}\frac{n}{\sum_{k=1, k\neq l}^{l_n} \frac{1}{\xi_l-\xi_k} + \sum_{j=1}^{n+1-l_n}\frac{1}{\xi_l - X_j}}} < \frac{C_{\mu,\eps,L}\cdot a_n^2}{n^2}.
	\end{equation}
	\label{thm:manyDet}
\end{theorem}

Theorem \ref{thm:manyDet} follows from an argument quite similar to the one provided in the previous subsection. We outline the main differences in the following proof sketch.

Argue as in Subsection \ref{sec:detLoc} for each $l$, $1 \leq l \leq s_n$, separately but in place of the definitions in equation \eqref{eqn:CkOut} choose
\[
C_1 := \frac{\eps}{2},\ C_2 := \frac{3L}{2},\ \text{and}\ k_\text{Lip} := \frac{\eps^2a_n^2}{4}.
\]
Also, modify the events $E_n^l$ into the events 
\[
E_n^l := \set{\abs{\frac{1}{n+1-l_n}\sum_{j=1}^{n+1-l_n}\frac{1}{\xi_l - X_j} - m_\mu(\xi_l)} \geq \frac{\abs{m_\mu(\xi_l)}}{3}},\ 1 \leq l \leq s_n.
\]
Notice that condition \eqref{it:detLoc1} from Theorem \ref{thm:detLoc} now holds for $n$ sufficiently large (depending on the rate of convergence of $a_nl_n/n \to 0$) on the complement of $E_n^l$ because 
\[
\abs{\frac{1}{n}\sum_{k=1,k\neq l}^{l_n}\frac{1}{\xi_l - \xi_k}} \leq \frac{\eps a_nl_n}{6n} = o(1),
\]
and this limit is uniform with respect to $1 \leq l \leq s_n$. The requirements \eqref{eqn:nMaxOutLoc} on $n$ now become
\[
n > \max\set{\frac{4C_2}{C_1},4C_2C\left(\frac{\eps^2a_n^2}{4}+1\right), a_n},
\]
which hold uniformly for $1 \leq l \leq s_n$ by assumption \eqref{it:manyDet1} in the statement of Theorem \ref{thm:manyDet}. By Hoeffding's inequality (Lemma \ref{lem:Hoeffding}), with $Y_j^l := \frac{1}{\xi_l - X_j}$, $K_l:=  \frac{\eps a_n}{6}$, and $t_l := \frac{\abs{m_\mu(\xi_l)}}{3} \geq \frac{\eps}{3}$, there are constants $C,c_{\mu,\eps} > 0$, independent of $l$, $\xi_l$, and $s_n$, so that for large $n$
\[
\P\left(E_n^l\right) \leq C \exp\left(-c_{\mu,\eps}(n+1-l_n)/a_n^2\right).
\]
Taking a union over $l$, $1 \leq l \leq s_n$ establishes the desired result.

\subsection{Proof of Theorem \ref{thm:multiInCLT}}

We now proceed to prove Theorem \ref{thm:multiInCLT}. In order to control the behavior of $\frac{1}{n}\sum_{j=1}^n\frac{1}{\xi-X_j}$, we will rely on the Law of Large Numbers. Lemma \ref{lem:CSnice} below justifies this approach by establishing some regularity properties for $\E(\xi-X_1)^{-1} = m_\mu(\xi)$ that we will continue to use throughout the remainder of the paper. 
We note that Lemma \ref{lem:CSnice} is similar to Lemma 5.7 in \cite{KS}.


\begin{lemma}[Regularity properties of the Cauchy--Stieltjes transform]
Suppose that on $B(\xi, \rho) \subset \C$, $\mu$ has a density with respect to the Lebesgue measure that is bounded by $C_{\mu,\xi,\rho}$. Then, 
\begin{enumerate}[(i)]
\item \label{it:CS1} for any $z \in B(\xi,\rho/2)$,
\[
\abs{m_\mu(z)} \leq \int_\C\frac{1}{\abs{z-w}}\,d\mu(w) \leq 2\pi C_{\mu,\xi,\rho}\min\set{{\rho}/{2},1} + \max\set{{2}/{\rho},1};\]
\item \label{it:CS2} if $\rho = \infty$ so that $\mu$ has a density bounded by $C_\mu$ on all of $\C$, then there exist constants $\kappa_\mu, \eps_\mu>0$, depending on $\mu$, so that the following holds. If $x,y \in \C$ with $\abs{x-y} < \eps_\mu$, then 
\[
\abs{m_\mu(x)-m_\mu(y)} \leq \kappa_\mu\abs{x-y}\ln\left(\abs{x-y}^{-1}\right).
\]
\end{enumerate}
\label{lem:CSnice}
\end{lemma}
\begin{proof}
To prove the first inequality, observe that for any $z \in B(\xi,\rho/2)$, 
\begin{align*}
\abs{m_\mu(z)} &\leq \int_\C\frac{1}{\abs{z-w}}\ind_{\abs{z-w}< \min\set{\rho/2, 1}}\,d\mu(w) + \int_\C\frac{1}{\abs{z-w}}\ind_{\abs{z-w}\geq \min\set{\rho/2, 1}}\,d\mu(w)\\
&\leq 2\pi C_{\mu,\xi,\rho} \int_0^{\min\set{\rho/2, 1}}\frac{1}{r}\cdot r\,dr + \max\set{{2}/{\rho},1}\\
&\leq 2\pi C_{\mu,\xi,\rho}\min\set{{\rho}/{2},1} + \max\set{{2}/{\rho},1},
\end{align*}
where we have used polar coordinates in the integral. To prove \eqref{it:CS2}, let $Z \sim \mu$ and fix $x,y \in \C$ with $\abs{x-y} \leq 1$. We will compute the difference 
\begin{equation}
\abs{m_\mu(x) - m_\mu(y)} = \abs{\E\left[\frac{1}{x-Z}\right] - \E\left[\frac{1}{y-Z}\right]} \leq \abs{x-y}\E\left[{\frac{1}{\abs{x-Z}\abs{y-Z}}}\right] 
\label{eqn:CSnice1}
\end{equation}
by splitting the expectations over each of the events
\begin{align*}
\mathcal{A} &:= \set{\abs{x-Z} \geq \abs{x-y}\ \text{and}\ \abs{y-Z} \geq \abs{x-y}},\\
\mathcal{B} &:= \set{\abs{x-Z} \geq \abs{x-y}\ \text{and}\ \abs{y-Z} < \abs{x-y}},\\
\mathcal{C} &:= \set{\abs{x-Z} < \abs{x-y}\ \text{and}\ \abs{y-Z} \geq \abs{x-y}},\\
\mathcal{D} &:= \set{\abs{x-Z} < \abs{x-y}\ \text{and}\ \abs{y-Z} < \abs{x-y}},
\end{align*}
whose union occurs almost surely. 
The Cauchy--Schwarz inequality implies
\begin{align*}
\E\left[{\frac{\ind_{\mathcal{A}}}{\abs{x-Z}\abs{y-Z}}}\right] 
&\leq 
\sqrt{\E\left[\frac{\ind_{\abs{x-Z}\geq \abs{x-y}}}{\abs{x-Z}^2}\right]\E\left[\frac{\ind_{\abs{y-Z}\geq \abs{x-y}}}{\abs{y-Z}^2}\right]}\\
&\leq 
2\pi C_\mu \int_{\abs{x-y}}^1\frac{1}{r^2}r\,dr + \E[1]= 
2\pi C_\mu\ln\abs{\frac{1}{x-y}} + 1.
\end{align*}
If $\mathcal{B}$ occurs, then $\abs{x-Z} \geq \abs{x-y}$, so the expectation on the right of \eqref{eqn:CSnice1} is bounded by
\begin{align*} 
\E\left[{\frac{\ind_{\mathcal{B}}}{\abs{x-Z}\abs{y-Z}}}\right] \leq\frac{1}{\abs{x-y}} \E\left[\frac{\ind_{\abs{y-Z}< \abs{x-y}}}{\abs{y-Z}}\right] \leq \frac{4\pi C_\mu}{\abs{x-y}}\int_0^{\abs{x-y}}\frac{1}{r}r\,dr = 4\pi C_\mu.
\end{align*}
We can bound $\E[\ind_\mathcal{C}(\abs{x-Z}\abs{x-Y})^{-1}]$ in similar fashion. 
For the expectations over the event $\mathcal{D}$, we have the following bound on the middle expression of \eqref{eqn:CSnice1}: 
\[
\E\left[\frac{\ind_{\mathcal{D}}}{\abs{x-Z}}\right]+\E\left[\frac{\ind_{\mathcal{D}}}{\abs{y-Z}}\right] \leq 8\pi C_\mu\int_0^{\abs{x-y}}\frac{1}{r}r\,dr \leq  8\pi C_\mu\abs{x-y}.
\]
%
In view of \eqref{eqn:CSnice1}, these last few inequalities establish \eqref{it:CS2}.
\end{proof}


We proceed to prove Theorem \ref{thm:multiInCLT}, starting with a justification of \eqref{eqn:multiInLoc} in the case $s=1$ and $\xi_1 = \xi$. 
Choose $\rho_\xi>0$ so that in the disk $B(\xi, 3\rho_\xi)$, $\mu$ has a density $f$ that is bounded by $C_f$. Our plan of attack will be to show that the hypotheses of Theorem \ref{thm:detLoc} are satisfied on the complement of a ``bad'' event whose probability tends to $0$ as $n$ grows. To optimize our control over this event, we allow it to depend on the parameter $\eps_n = o(1)$ that we will choose appropriately to achieve the asymptotic bound in \eqref{eqn:multiInLoc}.  

To that end, suppose $\eps_n \in (0,1)$, let $d_n:=\lceil\ln(\sqrt{n})\rceil$, and 
for each $n \geq 1$ define the annuli 
\begin{align*}
A_n^0 &:= \set{z \in \C:\abs{z - \xi} < \frac{\rho_\xi}{\sqrt{n}}},\\
A_n^k &:= \set{z \in \C: \frac{\rho_\xi e^{k-1}}{\sqrt{n}} \leq \abs{z-\xi} < \frac{\rho_\xi e^k}{\sqrt{n}}},\ 1 \leq k \leq d_n, 
\end{align*}
and the binomial random variables
\[
N_n^k := \#\set{1 \leq j \leq n: X_j \in A_n^k},\ 0 \leq k \leq d_n. 
\]
Consider the ``bad'' events
\begin{align*}
E_n &= \set{\abs{\frac{1}{n} \sum_{j=1}^n\frac{1}{\xi - X_j} - m_\mu(\xi)} \geq \frac{\abs{m_\mu(\xi)}}{2}},\\
F_n^k &= \set{N_n^k \geq \pi C_f\rho_\xi^2e^{2k} + \frac{e^k}{\sqrt{\eps_n}}},\ 0\leq k \leq d_n,\\ 
G_n &= \set{\min_{1\leq j \leq n}\abs{X_j - \xi} < \sqrt{\frac{\eps_n}{n}}}.
\end{align*}
We will demonstrate that if
\begin{equation}
C_1:= \frac{\abs{m_\mu(\xi)}}{2}, \quad C_2:= \frac{3\abs{m_\mu(\xi)}}{2}, \quad \text{and} \quad k_\text{Lip} :=\frac{C_{\mu,\xi}\ln{n}}{\eps_n^{3/2}},
\label{eqn:CkIn}
\end{equation}
for $\eps_n := (\ln{n})^{-2/3}$ and $C_{\mu,\xi}$ defined in Lemma \ref{lem:inLip} below, then the conditions in Theorem \ref{thm:detLoc} hold on the complement of $E_n \cup G_n \cup \bigcup_k F_n^k$ for large enough $n$. Furthermore, we will show that the union of these events occurs with probability tending to $0$. Notice that events $E_n$, $F_n^k$, and $G_n$ are related to conditions \eqref{it:detLoc1}, \eqref{it:detLoc2}, and \eqref{it:detLoc3} of Theorem \ref{thm:detLoc},  respectively. 

It is clear that condition \eqref{it:detLoc1} holds on the complement of $E_n$ because $m_\mu(\xi) \neq 0$. For $n > \frac{9}{C_1^2\eps_n}$, \eqref{it:detLoc3} is true, on the complement of $G_n$, because in this case, $\sqrt{\frac{\eps_n}{n}} > \frac{3}{C_1n}$. The following lemma establishes condition \eqref{it:detLoc2}. 

\begin{lemma}
There exists a constant $C_{\mu,\xi} > 0$, depending only on $\mu$ and $\xi$, so that if $\eps_n \in (0,1)$, and
\[
n > \max\set{\left(\frac{8\rho_\xi}{C_1\eps_n}\right)^2, \left(\frac{8e^2}{C_1\rho_\xi}\right)^2, \frac{8}{C_1\rho_\xi}},
\]
then, on the complement of $\bigcup_{k=0}^{d_n}F_n^k \cup G_n$, any complex numbers 
$z,w \in \overline{B}\left(\xi,\frac{2}{C_1n}\right)$
satisfy
\[
\abs{\sum_{j=1}^n\frac{1}{(z-X_j)(w-X_j)}} \leq C_{\mu,\xi}\cdot\frac{n\ln{n}}{\eps_n^{3/2}}.
\]
\label{lem:inLip}
\end{lemma}

\begin{proof}
Fix $z,w \in \overline{B}\left(\xi,\frac{2}{C_1n}\right)$ and $1 \leq j \leq n$. By applying the triangle inequality several times, we obtain
\begin{align*}
\abs{z-X_j}\abs{w-X_j} &\geq \Big|\big(\abs{\xi - X_j} - \abs{z-\xi}\big)\big(\abs{\xi - X_j} - \abs{w-\xi}\big)\Big|\\
&\geq \abs{\xi-X_j}^2 - \abs{\xi-X_j}\left(\abs{z-\xi}+\abs{w-\xi}\right)\\
&\geq \abs{\xi-X_j}^2 - \abs{\xi-X_j}\frac{4}{C_1n}.
\end{align*}
Consequently, on the complement of $\bigcup_{k=0}^{d_n}F_n^k \cup G_n$,
\begin{align*}
\abs{\sum_{j=1}^n\frac{1}{(z-X_j)(w-X_j)}}
&\leq \sum_{j=1}^n\frac{1}{\abs{z-X_j}\cdot\abs{w-X_j}}\\
&\leq \sum_{j=1}^n\frac{1}{\abs{\xi-X_j}^2 -\abs{\xi-X_j}\frac{4}{C_1n}}\\
&\leq\sum_{\substack{1\leq j \leq n\ \text{s.t.}\\[1pt]\frac{\sqrt{\eps_n}}{\sqrt{n}} \leq \abs{X_j - \xi} < \frac{\rho_\xi}{\sqrt{n}}}}\frac{1}{\frac{\eps_n}{n} -\frac{\rho_\xi}{\sqrt{n}}\frac{4}{C_1n}}\\
&\qquad{}+ \sum_{k=1}^{d_n}\sum_{\substack{1\leq j \leq n\ \text{s.t.}\\X_j \in A_n^k}}\frac{1}{\frac{\rho_\xi^2e^{2k-2}}{n} -\frac{\rho_\xi e^k}{\sqrt{n}}\frac{4}{C_1n}}\\
&\qquad{}+\sum_{\substack{1\leq j \leq n\ \text{s.t.}\\ \abs{X_j -\xi}\geq \rho_\xi}}\frac{1}{\abs{\xi-X_j}^2 -\abs{\xi-X_j}\frac{4}{C_1n}}.
\end{align*}
We have split the sum over $1 \leq j \leq n$ into $d_n +2$ pieces. Notice that for $n> \left(\frac{8\rho_\xi}{C_1\eps_n}\right)^2$, 
\begin{align*}
0 < \frac{1}{\frac{\eps_n}{n} -\frac{\rho_\xi}{\sqrt{n}}\frac{4}{C_1n}} &\leq \frac{2n}{\eps_n}
\end{align*}
and for $n > \left(\frac{8e^2}{C_1\rho_\xi}\right)^2$,
\begin{align*}
0< \frac{1}{\frac{\rho_\xi^2e^{2k-2}}{n} -\frac{\rho_\xi e^k}{\sqrt{n}}\frac{4}{C_1n}} &\leq \frac{2n}{\rho_\xi^2e^{2k-2}},\ \text{for $1 \leq k \leq d_n$.}
\end{align*}
Additionally, if $n > \frac{8}{C_1\rho_\xi}$ and $\abs{X_j - \xi} \geq \rho_\xi$, then,
\[
0<\frac{1}{\abs{\xi-X_j}^2 -\abs{\xi-X_j}\frac{4}{C_1n}} \leq \frac{1}{\abs{\xi-X_j}\left(\rho_\xi-\frac{4}{C_1n}\right)} \leq \frac{2}{\rho_\xi^2}.
\]
It follows that if 
\[
n > \max\set{\left(\frac{8\rho_\xi}{C_1\eps_n}\right)^2,  \left(\frac{8e^2}{C_1\rho_\xi}\right)^2, \frac{8}{C_1\rho_\xi}},
\]
on the complement of $\bigcup_{k=0}^{d_n}F_n^k \cup G_n$, for all $z,w \in \overline{B}\left(\xi,\frac{2}{C_1n}\right)$, 
\begin{align*}
&\abs{\sum_{j=1}^n\frac{1}{(z-X_j)(w-X_j)}}\\
&\quad\qquad\leq N_n^0\cdot \frac{2n}{\eps_n} + \sum_{k=1}^{d_n} N_n^k\cdot \frac{2n}{\rho_\xi^2e^{2k-2}} + \frac{2n}{\rho_\xi^2}\\
&\quad\qquad\leq \frac{(\pi C_f\rho_\xi^2+1)}{\sqrt{\eps_n}}\cdot \frac{2n}{\eps_n} + \sum_{k=1}^{d_n}\left(\pi C_f\rho_\xi^2e^{2k} + \frac{e^k}{\sqrt{\eps_n}}\right)\frac{2n}{\rho_\xi^2e^{2k-2}} + \frac{2n}{\rho_\xi^2}\\
&\quad\qquad= O_{\mu,\xi}\left(\frac{n}{\eps_n^{3/2}}\right) + \sum_{k=1}^{d_n}O_{\mu,\xi}\left(\frac{n}{\sqrt{\eps_n}}\right)= O_{\mu,\xi}\left(\frac{n\ln{n}}{\eps_n^{3/2}}\right),
\end{align*}
which completes the proof.  
\end{proof}

It remains to find an upper bound on the probability of $E_n \cup \bigcup_{k=1}^dF_n^k \cup G_n$, which we accomplish in the next lemma.

\begin{lemma}
\[
\P\left(E_n \cup \bigcup_{k=0}^{d_n} F_n^k \cup G_n \right) = o_{\mu,\xi}(1) + O_{\mu,\xi}\left(\ln{n}\cdot\eps_n^{2}+ \eps_n\right) = o_{\mu, \xi}(1)
\]
\label{lem:eventsToZero}
\end{lemma}

\begin{proof}
To control $\P(E_n)$, apply the Weak Law of Large Numbers to the random variables $\frac{1}{\xi - X_j}$, which have finite expectation by Lemma \ref{lem:CSnice}. Next, consider that for large $n$,
\[
\P(G_n) \leq n\cdot\P\left(\abs{X_1 - \xi} \leq \sqrt{\frac{\eps_n}{n}}\right) \leq n \cdot 
\pi C_f\cdot\frac{\eps_n}{n} = \pi C_f\eps_n,
\]
which establishes $\P(G_n) = O_{\mu,\xi}(\eps_n)$.

We now turn our attention to the events $F_n^k$. For $0\leq k \leq d_n$ and $1 \leq j \leq n$, define the random variables 
\[
\chi_{j,k} := \ind_{\set{X_j \in A_n^k}},
\]
which, for a fixed $k$, are independent and identically distributed according to a Bernoulli distribution with parameter $p_k \leq \pi C_f\rho_\xi^2 e^{2k}/n$. Since $N_n^k = \sum_{j=1}^n \chi_{j,k}$ has expectation at most $\pi C_f\rho_\xi^2e^{2k}$, Markov's inequality yields 
\begin{equation}
\P\left(N_n^k  \geq \pi C_f \rho_\xi^2e^{2k} +\frac{e^k}{\sqrt{\eps_n}}\right) \leq \frac{\eps_n^{2}\E\left[\left(N_n^k - \E[N_n^k]\right)^4\right]}{e^{4k}}.
\label{eqn:NDev}
\end{equation}
In order to control the fourth central moment of $N_n^k$, recall that for two independent, real-valued random variables $X$ and $Y$, 
\begin{align*}
&\E\left[\left(X+Y - \E[X]-\E[Y]\right)^4\right]\\
&\qquad\qquad= \E\left[\left(X-\E[X]\right)^4\right] + \E\left[\left(Y-\E[Y]\right)^4\right] + 6\var(X)\var(Y).
\end{align*}
Since $\chi_{j,k}$ are iid, it follows by inductively applying the previous identity that 
\begin{align*}
\E\left[\left(N_n^k - \E[N_n^k]\right)^4\right] &= n\E\left[\left(\chi_{1,k}-\E[\chi_{1,k}]\right)^4\right] + 6\frac{n(n-1)}{2}\var(\chi_{1,k})^2\\
&\leq n\left(\E[\chi_{1,k}^4] + 6\var(\chi_{1,k})\left(\E[\chi_{1,k}]\right)^2\right) + 3n^2\var(\chi_{1,k})^2\\
&\leq n\left(\E[\chi_{1,k}] + 6\left(\E[\chi_{1,k}]\right)^2\right) + 3n^2\left(\E[\chi_{1,k}]\right)^2\\
&\leq n\left(\frac{\pi C_f \rho_\xi^2e^{2k}}{n} + 6\frac{\pi^2C_f^2\rho_\xi^4e^{4k}}{n^2}\right) + 3n^2\frac{\pi^2C_f^2\rho_\xi^4e^{4k}}{n^2}\\
&= O_{\mu,\xi}\left(e^{4k}\right).
\end{align*}
Consequently, \eqref{eqn:NDev} becomes 
\[
\P\left(N_n^k  \geq \pi C_f \rho_\xi^2e^{2k} +\frac{e^{k}}{\sqrt{\eps_n}}\right) = O_{\mu,\xi}(\eps_n^2),
\]
and by the union bound
\[
\P\left(\bigcup_{k=0}^{d_n}F_n^k\right) = O_{\mu,\xi}\left(\ln{n}\cdot\eps^2_n\right).
\]
The proof of Lemma \ref{lem:eventsToZero} is complete.
\end{proof}

We have established that $C_1$, $C_2$, and $k_\text{Lip}$ defined in \eqref{eqn:CkIn} satisfy conditions \eqref{it:detLoc1}, \eqref{it:detLoc2}, and \eqref{it:detLoc3} of Theorem \ref{thm:detLoc} for large $n$, on the complement of $E_n \cup \bigcup_{k=0}^{d_n}F_n^k \cup G_n$, a ``bad'' event whose probability tends to zero. 
Consequently, the conclusion of Theorem \ref{thm:detLoc} guarantees that with probability at least $1-o_{\mu,\xi}(1)$, the polynomial $p_n$ has a unique critical point $w_\xi^{(n)}$ that fulfills  \eqref{eqn:multiInLoc}. 

We now consider the case $s>1$. The argument in this more general situation is much the same as the one just presented for $s=1$, so we sketch the proof and point out the major differences. Consider each of the roots $\xi_l$, $1 \leq l \leq s$ separately and modify the argument above in the obvious ways. In particular, we replace the annuli $A_n^k$ with  
\begin{align*}
A_{l,n}^0 &:= \set{z \in \C:\abs{z - \xi_l} < \frac{\delta}{\sqrt{n}}},\ 1 \leq l \leq s;\\
A_{l,n}^k &:= \set{z \in \C: \frac{\delta e^{k-1}}{\sqrt{n}} \leq \abs{z-\xi_l} < \frac{\delta e^k}{\sqrt{n}}},\ 1 \leq k \leq d_n,\ 1 \leq l \leq s;
\end{align*}
where $\delta > 0$ is any real number such that $f$ is a density for $\mu$ in the balls $B(\xi_l, \delta)$ and so that $2\delta < \min_{1\leq j<l\leq s}\abs{\xi_j-\xi_l}$. 
Define the random variables $N_{l,n}^k$ accordingly, in addition to the modified ``bad'' events 
\begin{align*}
E_{l,n} &=\set{\abs{\frac{1}{n+1-s}\sum_{j=1}^{n+1-s}\frac{1}{\xi_l - X_j} - m_\mu(\xi_l)} \geq \frac{\abs{m_\mu(\xi_l)}}{3}},\ 1 \leq l \leq s;\\
F_{l,n}^k &= \set{N_{l,n}^k \geq \pi C_f\delta^2e^{2k} + \frac{\delta e^k}{\sqrt{\eps_n}}},\ 0\leq k \leq d_n,\ 1 \leq l \leq s;\\
G_{l,n} &= \set{\min_{1\leq j \leq n}\abs{X_j - \xi_l} < \sqrt{\frac{\eps_n}{n}}},\ 1 \leq l \leq s;
\end{align*}
and the modified constants
\[
C_1^l:= \frac{\abs{m_\mu(\xi_l)}}{2}, \quad C_2^l:= \frac{3\abs{m_\mu(\xi_l)}}{2}, \quad \text{and} \quad k_\text{Lip}^l :=C_{\mu,\vec{\xi}}^l\cdot\frac{\ln{n}}{\eps_n^{3/2}},\ 1\leq l \leq s.
\]
(Note that $C_{\mu,\vec{\xi}}^l$, $1 \leq l \leq s$, will be defined via lemmata similar to Lemma \ref{lem:inLip}.) On the complement of the union of the modified ``bad'' events, for each $l$, $1 \leq l \leq s$, conditions \eqref{it:detLoc1}, \eqref{it:detLoc2}, and \eqref{it:detLoc3} of Theorem \ref{thm:detLoc} hold for reasons similar to those given in the argument for $s=1$ above. (Notice that for $1 \leq l \leq s$,
\[
\abs{\frac{1}{n}\sum_{k=1,k\neq l}^{s}\frac{1}{\xi_l - \xi_k}} = o(1),
\]
so computations similar to \eqref{eqn:multiOutC2} and \eqref{eqn:multiOutC1} establish condition \eqref{it:detLoc1} of Theorem \ref{thm:detLoc}.) The fact that the union of the modified ``bad'' events occurs with probability at most $o(1)$ follows by an updated version of Lemma \ref{lem:eventsToZero} and the union bound (recall $s$ is fixed and finite). 

We now turn our attention to \eqref{eqn:multiInCLT} which describes the joint fluctuations of $w_l^{(n)}$, $1 \leq l \leq s$. This is considerably more difficult than our consideration of \eqref{eqn:multiOutCLT} because in the current situation, $(\xi_l - X_j)^{-1}$, are heavy-tailed random variables. In Appendix \ref{sec:appendixB}, we appeal to the Lindeberg exchange method with an appropriate truncation to establish Theorem \ref{thm:genCLT}, a CLT that we use to prove \eqref{eqn:multiInCLT} in a similar manner to our justification of \eqref{eqn:multiOutCLT}. 

To start, consider that with probability $1-o(1)$, $w_l^{(n)}$, $1 \leq l \leq s$ satisfy \eqref{eqn:multiInLoc}, so with inspiration from \eqref{eqn:outCLTCalc} and \eqref{eqn:YnlCalc}, we obtain with probability at least $1-o(1)$ that for $1 \leq l \leq s$,
\begin{align*}
&\frac{n^{3/2}}{\sqrt{\ln{n}}}\cdot m_\mu(\xi_l)^2\cdot\left(w_l^{(n)} - \xi_l + \frac{1}{n+1}\frac{1}{m_\mu(\xi_l)}\right)\\
&\hspace{2.5cm}=\sqrt{n}\left(\frac{1}{n}\sum_{j=1}^n\frac{1}{\xi_l-X_j} - m_\mu(\xi_l)\right) +o(1),
\end{align*}
where all of the implied constants depend on $\xi_1, \ldots, \xi_s$ and $\mu$, and we have used Slutsky's theorem several times. (We also used the heavy-tailed CLT, Theorem \ref{thm:genCLT} once.)  For the arbitrary constants $t_1, \ldots, t_s \in \C$, we have with probability at least $1-o(1)$,
\begin{align*}
&\Re\left(\sum_{l=1}^st_l\frac{n^{3/2}}{\sqrt{\ln{n}}}\cdot m_\mu(\xi_l)^2\cdot\left(w_l^{(n)} - \xi_l+ \frac{1}{n+1}\frac{1}{m_\mu(\xi_l)}\right)\right)\\
&\qquad\quad= \Re\left(\sqrt{\frac{n}{\ln{n}}}\left(\frac{1}{n}\sum_{j=1}^n\sum_{l=1}^st_l\left[\frac{1}{\xi_l-X_j} - m_\mu(\xi_l)\right]\right)\right) +o(1),
\end{align*}
which converges in distribution by Slutsky's theorem and Theorem \ref{thm:genCLT} to a normal distribution with with mean zero and variance $\sum_{l=1}^s\frac{\pi\abs{t_l}^2f(\xi_l)}{2}$. This is exactly the same distribution as the sum $\Re\left(\sum_{l=1}^s t_l N_l\right)$, where $N_l$ are defined as in \eqref{eqn:multiInCLT} with covariance structure \eqref{eqn:multiInCov}. 
By the Cram\'{e}r--Wold technique, this completes the proof of 
Theorem \ref{thm:multiInCLT}.

\subsection{Proof of Theorem \ref{thm:eig}} We conclude this section by using Theorem \ref{thm:detLoc} to prove Theorem \ref{thm:eig}.
\begin{proof}
	Conclusions \eqref{item:inside} and \eqref{item:outside} follow from \cite[Theorem 1.7]{Tout}.  We now use Theorem \ref{thm:detLoc} to establish \eqref{eq:eigcp}.  In particular, we will verify the three conditions of Theorem \ref{thm:detLoc} hold for some constants $C_1, C_2, k_\text{Lip} > 0$ which depend only on $\eps$ and $\lambda$.  In view of parts \eqref{item:inside} and \eqref{item:outside}, it suffices to work on the event where 
	\begin{equation} \label{eq:eigevent}
	\max_{1 \leq i \leq n-1} |X_i| \leq 1 + 2 \eps, \quad \min_{1 \leq i \leq n-1} |\xi - X_i| \geq \frac{\eps}{2}, \quad 1 + \frac{11}{4} \eps \leq |\xi| \leq |\lambda| + 1. 
	\end{equation}
	In fact, this event automatically guarantees the third condition from Theorem \ref{thm:detLoc} for all values of $n$ sufficiently large.  The second condition also follows for large n since, for $z,w \in \mathbb{C}$ with $|z|, |w| > 1 + 5/2 \eps$, we have
	\[ \left| \frac{1}{n} \sum_{i=1}^{n-1} \frac{1}{z - X_i} - \frac{1}{n} \sum_{i=1}^{n-1} \frac{1}{w - X_i}  \right| \leq \frac{|z - w|}{n} \sum_{i=1}^{n-1} \frac{1}{|z - X_i| |w - X_i| } \ll_{\eps} |z-w| \]
	on the same event.  The upper bound in the first condition of Theorem \ref{thm:detLoc} follows from a similar argument.  The lower bound, however, is slightly more involved.  Indeed, for any $\theta \in \mathbb{R}$, we have 
	\[ \left| \frac{1}{n} \sum_{i=1}^{n-1} \frac{1}{\xi - X_i} \right| = \left| \frac{1}{n} \sum_{i=1}^{n-1} \frac{1}{\xi e^{\sqrt{-1} \theta} - X_i e^{\sqrt{-1} \theta}} \right| \geq \frac{1}{n} \sum_{i=1}^{n-1} \frac{\Re(\xi e^{\sqrt{-1} \theta}) - \Re(X_i e^{\sqrt{-1} \theta}) }{|\xi - X_i|^2}. \]
	Choose $\theta \in \mathbb{R}$ so that $\xi e^{\sqrt{-1} \theta}$ is real-valued and positive.  This gives
	\[ \left| \frac{1}{n} \sum_{i=1}^{n-1} \frac{1}{\xi - X_i} \right| \geq \frac{1}{n} \sum_{i=1}^{n-1} \frac{\xi e^{\sqrt{-1} \theta} - \Re(X_i e^{\sqrt{-1} \theta}) }{|\xi - X_i|^2} \geq \frac{1}{n} \sum_{i=1}^{n-1} \frac{ |\xi| - |X_i|}{(|\xi| + |X_i| )^2}. \]
	Thus, on the event \eqref{eq:eigevent}, we conclude that
	\begin{equation} \label{eq:eiglower}
	\left| \frac{1}{n} \sum_{i=1}^{n-1} \frac{1}{\xi - X_i} \right| \gg_{\eps, \lambda} 1, 
	\end{equation}
	which completes the proof of the lower bound.  Hence, the three conditions of Theorem \ref{thm:detLoc} are satisfied.  Applying Theorem \ref{thm:detLoc}, we obtain \eqref{eq:eigcp}.  Lastly, \eqref{eq:eigo1} follows from \eqref{eq:eigcp} after applying conclusion \eqref{item:outside} and \eqref{eq:eiglower}.  
\end{proof}


\section{Proof of results in Section \ref{sec:locallaw}}\label{sec:locProof}

\subsection{Proof of Theorem \ref{thm:main}} \label{sec:proof:main}

This section is devoted to the proof of Theorem \ref{thm:main}.  We will need the following lemmata.   

\begin{lemma}[Monte Carlo sampling; Lemma 36 from \cite{TV}] \label{lemma:mc}
Let $(X, \mu)$ be a probability space, and let $F:X \to \mathbb{C}$ be a square-integrable function.  Let $m \geq 1$, let $x_1, \ldots, x_m$ be drawn independently at random from $X$ with distribution $\mu$, and let $S$ be the empirical average
\[ S := \frac{1}{m} ( F(x_1) + \cdots + F(x_m)). \]
Then $S$ has mean $\int_X F d \mu$ and variance $\frac{1}{m}\int_X | F - \int_X F d\mu |^2 d \mu$.  In particular, by Chebyshev's inequality, one has
\[ \Prob \left( \left| S - \int_X F d \mu \right| \geq t \right) \leq \frac{1}{m t^2} \int_X \left| F - \int_X F d \mu \right|^2 d \mu \]
for any $t > 0$, or equivalently, for any $\delta > 0$ one has with probability at least $1 - \delta$ that
\[ \left| S - \int_X F d \mu \right| \leq \frac{1}{\sqrt{m \delta}} \left( \int_X \left| F - \int_X F d \mu \right|^2 d \mu \right)^{1/2}. \]
\end{lemma}

\begin{lemma} \label{lemma:help}
Fix $C > 0$, and let $X_1, \ldots, X_n$ be complex-valued random variables (not necessarily independent nor identically distributed) such that, with overwhelming probability, 
\begin{equation} \label{eq:maxXi}
	\max_{1 \leq i \leq n} |X_i| \leq e^{n^C}. 
\end{equation}
Let $\varphi: \mathbb{C} \to \mathbb{R}$ be a twice continuously differentiable function (possibly depending on $n$) which satisfies the pointwise bound in \eqref{eq:pointwise} for all $z \in \mathbb{C}$.  Then, with overwhelming probability, 
\begin{align}
	\int_{B(0, n^C)} | \lap \varphi(z) |^2  \log^2 | p_n(z) | d^2 z \ll n^{2C} n^{O(1)}, \label{eq:pbnd}\\
	\int_{B(0, n^C)} | \lap \varphi(z) |^2  \log^2 | p_n'(z) | d^2 z \ll n^{2C} n^{O(1)}, \label{eq:p'bnd}
\end{align}
and
\begin{equation} \label{eq:lapbnd}
	\int_{B(0, n^{C})} | \lap \varphi(z) |^2 d^2 z \ll n^{4C}. 
\end{equation}
\end{lemma}
\begin{proof}
The bound in \eqref{eq:lapbnd} follows immediately from the pointwise bound in \eqref{eq:pointwise}.  In order to prove \eqref{eq:pbnd} it suffices, by the pointwise bound in \eqref{eq:pointwise}, to prove that with overwhelming probability
\[ \int_{B(0, n^C)} \log^2 | p_n(z) | d^2 z \ll n^{O(1)}. \]
By supposition, we now work on the event where $X_1, \ldots, X_n \in B(0, e^{n^C})$.  As
\[ \log^2 |p_n(z)| \ll n \sum_{i=1}^n \log^2 |z - X_i|, \]
it suffices to prove that
\[ \max_{1 \leq i \leq n} \int_B \log^2 |z - X_i| d^2 z \ll n^{O(1)}, \]
where $B := B(0, n^{C})$.  
Since $X_1, \ldots, X_n \in B(0, e^{n^C})$, it follows that
\[ \max_{1 \leq i \leq n} \int_{B \setminus B(X_i, 1 ) } \log^2 |z - X_i| d^2 z \ll  n^{2C} |B| \ll n^{O(1)}, \]
where $|B|$ is the Lebesgue measure of $B$, and $|B| = O(n^{2C})$.  Near each root, we have
\[ \max_{1 \leq i \leq n} \int_{B \cap B(X_i, 1)} \log^2 |z - X_i| d^2 z \leq \max_{1 \leq i \leq n} \int_{B(X_i, 1)} \log^2 |z - X_i| d^2 z \ll 1 \]
since $\log| \cdot |$ is locally square-integrable.  This completes the proof of \eqref{eq:pbnd}.  

For \eqref{eq:p'bnd}, we observe that on the event where \eqref{eq:maxXi} holds, the Gauss--Lucas theorem implies that
\[ \max_{1 \leq j \leq n-1} |w_j| \leq e^{n^C}, \]
where $w_1^{(n)}, \ldots, w_{n-1}^{(n)}$ are the critical points of $p_n$.  Working on this event, the proof follows from the same procedure as we used to prove \eqref{eq:pbnd}; we omit the details.  
\end{proof}

\begin{lemma}[Crude upper bound] \label{lemma:crude}
Fix $C > 0$, and let $X_1, \ldots, X_n$ be complex-valued random variables (not necessarily independent nor identically distributed).  Assume $Z$ is uniformly distributed on $B(0, n^C)$, independent of $X_1, \ldots, X_n$.  Then for every $a > 0$, there exits $b > 0$ such that
\[ \left| \sum_{i=1}^n \frac{1}{Z - X_i} \right| \leq n^{b} \]
with probability $1 - O_a(n^{-a})$.  
\end{lemma}
\begin{proof}
Conditioning on $X_1, \ldots, X_n$, we find that
\begin{align*}
	\Prob \left( \min_{1 \leq i \leq n} |Z - X_i| \leq \eps \right) &\leq \sum_{i=1}^n \Prob( Z \in B(X_i, \eps ) ) \ll n \frac{ \eps^2}{n^{2C} } 
\end{align*}
for all $\eps > 0$.  In addition, on the event where $\min_{1 \leq i \leq n} |Z - X_i| > \eps$, we have
\[ \left| \sum_{i=1}^n \frac{1}{Z - X_i} \right| \leq \frac{n}{\eps}. \]
In order to prove the claim, it suffices to assume $a > 2C$.  In this case, by taking $\eps := \sqrt{ \frac{ n^{2C} }{ n^{a + 1}} }$, the result follows from the estimates above.   
\end{proof}

We now prove Theorem \ref{thm:main}. 
\begin{proof}[Proof of Theorem \ref{thm:main}]
Let $B := B(0, n^C)$, and let $|B|$ denote its Lebesgue measure.  Fix $\alpha > 0$, and let $\beta \in \mathbb{N}$ be a large constant (depending on $C, c, \alpha$) to be chosen later.  

Using the log-transform of the empirical measures constructed from the roots and critical points of $p$, we obtain 
\begin{align*}
	\sum_{i=1}^n \varphi(X_i) &= \frac{1}{2 \pi} \int_B \lap \varphi(z) \log |p_n(z)| d^2 z, \quad
	\sum_{j=1}^{n-1} \varphi(w_j) &= \frac{1}{2 \pi} \int_B \lap \varphi(z) \log |p'_n(z))| d^2 z.  
\end{align*}
(These identities can also be found in a more general form in \cite[Section 2.4.1]{JHough}.)  Instead of working with the integrals on the right-hand sides, we will work with large empirical averages by applying Lemma \ref{lemma:mc}.  Indeed, let $m := n^\beta$, and let $Z_1, \ldots, Z_m$ be iid random variables uniformly distributed on $B$, independent of $X_1, \ldots, X_n$.  Taking $\beta$ sufficiently large and applying Lemmas \ref{lemma:mc} and \ref{lemma:help}, we conclude that
\begin{align}
	\frac{2 \pi}{|B|} \sum_{i=1}^n \varphi(X_i) &= \frac{1}{m} \sum_{l=1}^m \lap \varphi(Z_l) \log |p_n(Z_l)| + O(n^{-c - 2C}), \label{eq:loc:conc1}\\
	\frac{2 \pi}{|B|} \sum_{j=1}^{n-1} \varphi(w_j) &= \frac{1}{m} \sum_{l=1}^m \lap \varphi(Z_l) \log |p'_n(Z_l)| + O(n^{-c - 2C}), \label{eq:loc:conc2}\\
	\frac{1}{|B|} \int_B | \lap \varphi(z) | d^2 z &= \frac{1}{m} \sum_{l=1}^m |\lap \varphi(Z_l)| + O(n^{-c - 2C -1 }) \label{eq:loc:conc3}
\end{align}
with probability $1 - O(n^{-\alpha})$.  In addition, by \eqref{eq:anti}, Lemma \ref{lemma:crude}, and the union bound it follows that there exists $b > 0$ such that
\[ n^{-b} \leq \min_{1 \leq l \leq m} \left| \sum_{i=1}^n \frac{1}{Z_l - X_i } \right| \leq \max_{1 \leq l \leq m} \left| \sum_{i=1}^n \frac{1}{Z_l - X_i } \right| \leq n^{b} \]
with probability $1 - O(n^{-\alpha})$.  Thus, since 
$ \frac{p'_n(z)}{ p_n(z)} = \sum_{i=1}^n \frac{1}{z - X_i }, $
we obtain 
\begin{equation} \label{eq:loc:det}
	\sup_{1 \leq l \leq m} \left| \log |p_n(Z_l)| - \log |p'_n(Z_l)| \right| = O(\log n) 
\end{equation}
with probability $1 - O(n^{-\alpha})$.  

From \eqref{eq:loc:conc1} and \eqref{eq:loc:conc2}, we find
\begin{align*}
	&\left| \frac{2 \pi}{|B|} \sum_{i=1}^n \varphi(X_i) - \frac{2 \pi}{|B|} \sum_{j=1}^{n-1} \varphi(w_j) \right| \\
	&\qquad \leq \frac{1}{m} \sum_{l=1}^m | \lap \varphi(Z_l) |  \left| \log |p_n(Z_l)| - \log |p'_n(Z_l)| \right| + O(n^{-c -2C}) 
\end{align*}
with probability $1 - O(n^{-\alpha})$.  
Applying \eqref{eq:loc:conc3} and \eqref{eq:loc:det} yields
\begin{align*}
	&\left| \frac{2 \pi}{|B|} \sum_{i=1}^n \varphi(X_i) - \frac{2 \pi}{|B|} \sum_{j=1}^{n-1} \varphi(w_j) \right| \\
	&\qquad \ll (\log n) \frac{1}{m} \sum_{l=1}^m | \lap \varphi(Z_l)| + n^{-c-2C} \\
	&\qquad \ll (\log n) \frac{1}{|B|} \int_B |\lap \varphi(z)| d^2z + n^{-c-2C}
\end{align*}
with probability $1 - O(n^{-\alpha})$.  Since $|B| = \Theta(n^{2C})$, we rearrange to obtain
\begin{equation} \label{eq:loc:conc}
	\left|  \sum_{i=1}^n \varphi(X_i) -  \sum_{j=1}^{n-1} \varphi(w_j) \right| \ll (\log n) \| \lap \varphi \|_1 + n^{-c} 
\end{equation}
with probability $1 - O(n^{-\alpha})$.  The proof of the theorem is complete.  
\end{proof}


\subsection{Proof of Theorems \ref{thm:indep} and \ref{thm:clt}} \label{sec:proof:indep}

In order to prove Theorem \ref{thm:indep}, it suffices to show that $X_1, \ldots, X_n$ satisfy the two axioms of Theorem \ref{thm:main}.  This follows from Lemmas \ref{lemma:critup} and \ref{lemma:critanti}.  

We now turn to the proof of Theorem \ref{thm:clt}.  By Theorem \ref{thm:indep}, 
\[ \sum_{j=1}^{n-1} \varphi( w^{(n)}_j) = \sum_{i=1}^n \varphi(X_i) + O(\log n) \]
with probability $1 - O(n^{-100})$.  Since $\varphi$ is bounded, we obtain
\[ \E\left[\sum_{j=1}^{n-1}\varphi( w^{(n)}_j)\right] = \sum_{i=1}^n \E \varphi(X_i) + O(\log n). \]
Therefore, we conclude that
\[
\frac{1}{\sqrt{n}} \left( \sum_{j=1}^{n-1} \varphi( w^{(n)}_j) - \E\left[\sum_{j=1}^{n-1}\varphi( w^{(n)}_j)\right] \right) = \frac{1}{\sqrt{n}} \sum_{i=1}^n \left( \varphi(X_i) - \E \varphi(X_i) \right) + o(1)
\]
with probability $1 - O(n^{-100})$.  
The claim now follows by applying the classical CLT to the right-hand side.  


\subsection{Proof of Theorem \ref{thm:analytic}} \label{sec:proof:analytic}

We will need the following companion matrix result, which describes a matrix whose eigenvalues are the critical points of a given polynomial.  This result appears to have originally been developed in \cite{KI} (see \cite[Lemma 5.7]{KI}).  However, the same result was later rediscovered and significantly generalized by Cheung and Ng \cite{CN,CN2}.

\begin{theorem}[Lemma 5.7 from \cite{KI}; Theorem 1.2 from \cite{CN2}] \label{thm:companion}
Let $p(z) := \prod_{j=1}^n (z - z_j)$ for some complex numbers $z_1, \ldots, z_n$, and let $D$ be the diagonal matrix $D := \diag(z_1, \ldots, z_n)$.  Then
$$ \frac{1}{n} z p'(z) = \det \left(zI - D \left( I - \frac{1}{n} J \right) \right), $$
where $I$ is the $n \times n$ identity matrix and $J$ is the $n \times n$ all-one matrix.  
\end{theorem}

We will also need the Sherman--Morrison formula for computing the inverse of a rank one update to a matrix.  

\begin{lemma}[Sherman--Morrison formula] \label{lemma:SM}
Suppose $A$ is an invertible matrix and $u,v$ are column vectors.  If ${1 + v^\mathrm{T} A^{-1} u \neq 0}$, then
$$ (A + uv^\mathrm{T})^{-1} = A^{-1} - \frac{A^{-1} u v^\mathrm{T} A^{-1}}{1 + v^\mathrm{T} A^{-1} u}. $$
\end{lemma}

Lemma \ref{lemma:SM} can be found in \cite{B}; see also \cite[Section 0.7.4]{HJ} for a more general version of this identity known as the Sherman--Morrison--Woodbury formula.  
We also require the following consequence of \cite[Lemma 4.1]{OW}.

\begin{lemma} \label{lemma:OW}
Under the assumptions of Theorem \ref{thm:analytic}, there exists a constant $c' > 0$ (depending only on $C, c$, and $\eps$) such that
\[ \inf_{z \in \Gamma} \left| \frac{z}{n} \sum_{i=1}^n \frac{1}{z - X_i} \right| \geq c' \]
with overwhelming probability.  
\end{lemma}
\begin{proof}
Clearly $|z| = C+ \eps$ for all $z \in \Gamma$.  Thus, it suffices to prove that
\[ \inf_{z \in \Gamma} \left| \frac{1}{n} \sum_{i=1}^n \frac{1}{z - X_i} \right| \geq c' \]
with overwhelming probability.  The claim now follows from the uniform bound in \cite[Lemma 4.1]{OW} and the assumption on $m_\mu$ given in \eqref{eq:mmulowbnd}.  
\end{proof}

With Lemma \ref{lemma:OW} in hand, we are now prepared to present the proof of Theorem \ref{thm:analytic}.  

\begin{proof}[Proof of Theorem \ref{thm:analytic}]
Let $D$ be the diagonal matrix $D:=\diag(X_1, \ldots, X_n)$.  Using the notation from Theorem \ref{thm:companion}, we observe that $zI - D$ is invertible for all $z \in \Gamma$ since $X_1, \ldots, X_n \in B(0, C)$ by supposition.  In addition, by the Gauss--Lucas theorem and Theorem \ref{thm:companion}, it must be the case that the eigenvalues of $D \left( I - \frac{1}{n} J \right)$ are also contained in $B(0,C)$.  This implies that $zI - D \left( I - \frac{1}{n} J \right)$ is also invertible for every $z \in \Gamma$.   In view of these observations, we define the resolvents
\[ G(z) := (zI - D)^{-1}, \qquad R(z) := \left( zI - D \left( I - \frac{1}{n} J \right) \right)^{-1} \]
for $z \in \Gamma$.  

Thus, by Cauchy's integral formula 
\[ \sum_{i=1}^n \varphi(X_i) = \frac{1}{2 \pi \sqrt{-1}} \oint_{\Gamma} \varphi(z) \tr G(z) dz \]
and
\[ \sum_{j=1}^{n-1} \varphi(w_j) + \varphi(0) = \frac{1}{2 \pi \sqrt{-1}} \oint_{\Gamma} \varphi(z) \tr R(z) dz. \]
We now take the difference of these two equalities.  Since $|\varphi(0)| \ll \int_\Gamma |\varphi(z)| |dz|$, it suffices by the triangle inequality to show
\begin{equation} \label{eq:analyticshow}
	\sup_{z \in \Gamma} \left| \tr G(z) - \tr R(z) \right| = O(1) 
\end{equation}
with overwhelming probability.  

Since $J = \one \one^{\mathrm{T}}$, where $\one$ is the all-ones vector, the Sherman--Morrison formula (Lemma \ref{lemma:SM}) implies that 
\begin{equation} \label{eq:SM}
	R(z) = G(z) - \frac{ \frac{1}{n} G(z) D J G(z) }{1 + \frac{1}{n} \one^{\mathrm{T}} G(z) D \one } 
\end{equation}
provided $1 + \frac{1}{n} \one^{\mathrm{T}} G(z) D \one \neq 0$.  In view of Lemma \ref{lemma:OW}, there exists a constant $c' > 0$ (depending only on $C, c$, and $\eps$) such that
\begin{equation} \label{eq:gammalowbnd}
	\inf_{z \in \Gamma} \left| 1 + \frac{1}{n} \one^{\mathrm{T}} G(z) D \one \right| = \inf_{z \in \Gamma} \left| \frac{z}{n} \sum_{i=1}^n \frac{1}{z - X_i} \right| \geq c'
\end{equation}
with overwhelming probability.  Here, we have exploited the fact that $D$ and $G(z)$ are diagonal matrices, which implies that
$ \one^{\mathrm{T}} G(z) D \one = \sum_{i=1}^n \frac{X_i}{z - X_i}$. 
Using \eqref{eq:SM} and \eqref{eq:gammalowbnd}, we conclude that with overwhelming probability
\[ \sup_{z \in \Gamma} \left| \tr G(z) - \tr R(z) \right| \leq \frac{1}{n c'} \sup_{z \in \Gamma} \left| \tr [G(z) DJ G(z)] \right|. \]
To bound this last remaining term, we again exploit the fact that $J = \one \one^{\mathrm{T}}$.  Indeed, from the cyclic property of the trace, we have the deterministic bound 
\[ \left| \tr [G(z) DJ G(z)] \right| = \left| \one^{\mathrm{T}} G^2(z) D \one \right| = \left| \sum_{i=1}^n \frac{X_i}{(z - X_i)^2} \right| \leq \sum_{i=1}^n \frac{|X_i|}{|z - X_i|^2} \leq n\frac{C}{\eps^2} \]
for all $z \in \Gamma$.  Combining the bounds above, we obtain \eqref{eq:analyticshow}, and the proof is complete.  
\end{proof}

\section{Proof of Theorem \ref{thm:wasserstein}}\label{sec:wassProof}

This section is devoted to proving Theorem \ref{thm:wasserstein}. Our first lemma shows that Assumption \ref{assum:radSym} implies Assumption \ref{assum:subquad}.


\begin{lemma}[Sufficiency of Assumption \ref{assum:radSym}]
If $\mu$ satisfies Assumption \ref{assum:radSym}, then $\mu$ also satisfies Assumption \ref{assum:subquad}.
\label{lem:radSymAssump}
\end{lemma}
\begin{proof}
Without loss of generality, suppose $\mu$ is radially symmetric about $z=0$, and let $X \sim \mu$. By Lemma \ref{lem:radSymMu}, we can write
\[
\abs{m_\mu(z)} = \frac{\P(\abs{X} < \abs{z})}{\abs{z}},
\]
so the hypotheses guarantee that $\abs{m_\mu(z)}$ is continuous on $\C\setminus\set{0}$. (Indeed, $\P(\abs{X} < r)$ is the cumulative distribution function associated to the radial part of $\mu$, which has a continuous density.) Since $f(0) > 0$, there are $\delta, c > 0$ so that $\abs{z} \leq \delta$ implies $\abs{f(z)} \geq c > 0$. In particular, for $\abs{z} \leq \delta$, 
\begin{equation}
\abs{m_\mu(z)} = \frac{1}{\abs{z}}\int_0^{\abs{z}} r f(r)\,dr \geq  \frac{c}{\abs{z}}\int_0^{\abs{z}} r\,dr = \frac{c\abs{z}}{2}.
\label{eqn:CSsupLin}
\end{equation}
Let $r_{1/2}$ be any value for which $\P(\abs{X} < r_{1/2}) = 1/2$. By the extreme value theorem, $\abs{m_\mu(z)}$ achieves its minimum, $m_\text{min}$, on the closed, bounded annulus 
\[
A := \set{z\in \C: \delta \leq \abs{z} \leq r_{1/2}}.
\]
We know that $m_\text{min}$ is non-zero by \eqref{eqn:CSsupLin} and the fact that $\P(\abs{X} < r)$ is non-decreasing in $r$. This second fact additionally implies that for $\abs{z} \geq r_{1/2}$, 
\[
\abs{m_\mu(z)} = \frac{\P(\abs{X} < \abs{z})}{\abs{z}} \geq \frac{1}{2\abs{z}}.
\]
We conclude that for any $\eps \in (0, m_\text{min})$, 
\begin{equation}\label{eqn:radSubquad}
\P(\abs{m_\mu(X)} < \eps) \leq \P\left(\frac{c\abs{X}}{2} < \eps\right) + \P(m_\text{min} < \eps) + \P\left(\frac{1}{2\abs{X}} < \eps\right) \leq C\eps^2,
\end{equation}
for some $C>0$. (We have used the fact that $\mu$ has two finite absolute moments to bound the last probability.) It follows that $\mu$ satisfies Assumption \ref{assum:subquad} part \eqref{it:subquad}. 

To see that $\mu$ satisfies Assumption \ref{assum:subquad} part \eqref{it:maxBd}, let $X_1, \ldots, X_n$ be iid complex-valued random variables with distribution $\mu$, and observe that 
\[
\P\left(\max_j\abs{X_j} > \sqrt{n\ln{n}}\right) = 1 - \P\left(\abs{X_1} \leq \sqrt{n\ln{n}}\right)^n.
\]
By Markov's inequality, 
\[
 \P\left(\abs{X_1} \leq \sqrt{n\ln{n}}\right)^n \geq \left(1 - \frac{\E\abs{X_1}^2}{n\ln{n}}\right)^n \xrightarrow{n\to \infty} 1,
\]
which completes the argument.
\end{proof}

\subsection{Introduction to and motivation for the proof of Theorem \ref{thm:wasserstein}.}

The following proof of Theorem \ref{thm:wasserstein} is motivated by the illustration in Figure \ref{fig:loops} that depicts the roots (red dots) and critical points (blue crosses) of $p_n(z)$ when the roots, $X_1, \ldots, X_{150}$ are chosen independently and uniformly in the unit disk centered at the origin. The observer will notice two things:
\begin{enumerate}[1)]
\item since the $X_j$ are chosen uniformly at random, they tend to ``clump together,'' and
\item the roots further from the origin tend to ``pair'' more closely with nearby critical points than the roots near the origin.
\end{enumerate}

The first of these makes it difficult to use our strategy from Theorems \ref{thm:multiInCLT}, \ref{thm:multiOutCLT} and \ref{thm:detLoc}, where it was a simple matter to ``zoom in'' on a fixed root and ensure that no other roots were nearby. We address this concern by grouping the critical points that lie near each ``clump'' of roots and simultaneously considering all of the critical points that lie in the same group. We will show that each ``clump'' of roots (and its corresponding group of critical points) is far away from other ``clumps,'' for large $n$.

The second observation can be explained by Theorem \ref{thm:multiInCLT}, which suggests that the closest critical point, $w_j^{(n)}$, to a given root $X_j$ is at a distance $\frac{1}{n\abs{m_\mu(X_j)}}$ from $X_j$. For example, in the case where $\mu$ is uniform on the unit disk, $\abs{m_\mu(z)} = \abs{z}$ for $\abs{z} \leq 1$, so near the origin, it makes sense that the ``pairing'' phenomenon gets worse. We tackle this problem by counting the ``clumps'' of roots and critical points in exponentially widening, nested regions that avoid the zeros of $m_\mu$. (In Figure \ref{fig:loops}, these are the annuli delimited by concentric dashed circles.) Using this method, we can take advantage of the fact that the number of ``clumps'' that are a given distance from the zero set of $m_\mu$ is roughly proportional to the strength of the ``pairing'' within those ``clumps.'' The ``pairing'' phenomenon is quite unreliable near the zeros of $m_\mu$, so for any ``clumps'' that are sufficiently close to the zeros of $m_\mu$, we bound the distances between the roots and critical points using the Gauss--Lucas theorem.  (In fact, this is where we expect to find the ``extra,'' un-paired root that results because $p_n$ has a higher degree than $p_n'$). 

In order to synthesize these two ideas, we will form random, disjoint, simple closed curves to encircle each ``clump'' of roots and critical points. We will build the curves from the arcs of circles centered at the roots of $p_n$ and will use smaller circles for roots that are farther away from the zeros of $m_\mu$. See, for example, the boundaries of the gray domains depicted in Figure \ref{fig:loops}. We will conclude with an argument involving Rouch\'{e}'s theorem to count the number of critical points interior to each curve by comparing $p_n'$ to a simpler polynomial whose critical points can be located with Walsh's two circle theorem. Near the zeros of $m_\mu$, our method breaks down, and we use the Gauss--Lucas theorem for a bound on the distances between the critical points and roots of $p_n$. Luckily, there are few critical points near the zeros of $m_\mu$, a fact which follows in part from Assumptions \ref{assum:subquad} and \ref{assum:radSym}. 

\begin{figure}
\includegraphics[width =.7\columnwidth]{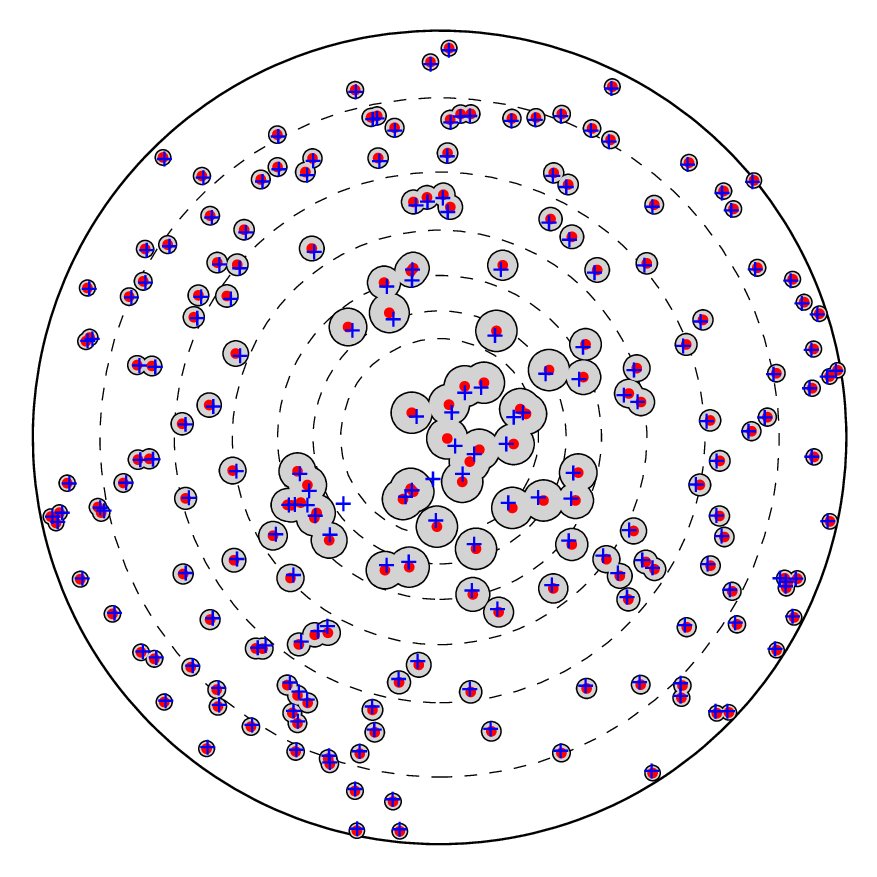}
\caption{An illustration motivating the strategy we use to prove Theorem \ref{thm:wasserstein}. The red circles and blue crosses represent the locations of the roots and critical points, respectively, of $p_{150}(z)$, where $\mu$ is the uniform distribution on the unit disk. Roughly speaking, the gray disks around the $X_j$ are of radius $\max\set{1/(n\abs{m_\mu(X_j)}),1/\sqrt{n}}$. The dashed concentric circles are meant to divide the unit disk into exponentially widening annuli.}
\label{fig:loops}
\end{figure} 

\subsection{Definitions}
In view of Lemma \ref{lem:radSymAssump}, we prove Theorem \ref{thm:wasserstein} under Assumption \ref{assum:subquad}.  Let $C_\mu > 0$ be larger than each of the constants in Assumption \ref{assum:subquad} and larger than the constant bounding the density associated to $\mu$. For each $n \in \mathbb{N}$, define the following sets which partition $\C$ into regions based on the size of $\abs{m_\mu(z)}$:
\begin{align*}
A_n^k &:= \set{z \in \C:  \abs{m_\mu(z)} < \frac{e^k}{\sqrt{n}}},\ k = \lfloor 4\ln(\ln{n})\rfloor,\\
A_{n}^k &:= \set{z \in \C: \frac{e^{k-1}}{\sqrt{n}} \leq \abs{m_\mu(z)} < \frac{e^k}{\sqrt{n}}},\ \lfloor 4\ln(\ln{n})\rfloor + 1 \leq k \leq \left\lfloor \ln\left(\sqrt{n}\right)\right\rfloor,\\
A_n &:= \set{z \in \C: \abs{m_\mu(z)} \geq \frac{e^{\left\lfloor \ln\left(\sqrt{n}\right)\right\rfloor}}{\sqrt{n}}}.
\end{align*}
Additionally, define the random variables 
\begin{align*}
N_{n}^k &:= \#\set{1 \leq j \leq n : X_j \in A_n^k},\ \lfloor 4\ln(\ln{n})\rfloor \leq k \leq \left\lfloor \ln\left(\sqrt{n}\right)\right\rfloor,\\
\zeta_{i,j}^{(n)} &:= \begin{cases}\displaystyle\frac{1}{X_i-X_j}\ind_{\abs{X_i - X_j} \geq \frac{(\ln{n})^2}{n\abs{m_\mu(X_i)}}},&\abs{m_\mu(X_i)} \neq 0\\\displaystyle 0, & \text{otherwise}
\end{cases},\ 1 \leq i, j \leq n,\ j\neq i,
\end{align*}
and let $\mathcal{N}_n$ be a $n^{-1/2}$-net of the closed disk $\overline{B}(0,n^{C_\mu})$ 
that satisfies:
\begin{enumerate}[(i)]
\item $\overline{B}(0,n^{C_\mu}) \subseteq \bigcup_{x \in \mathcal{N}_n} B(x,n^{-1/2})$,
\item if $x,y \in \mathcal{N}_n$, and $x \neq y$, then $\abs{x-y} \geq \frac{1}{2\sqrt{n}}$,
\item $\#\mathcal{N}_n = O_\mu(n^{1+2C_\mu})$.
\end{enumerate} 
Such a collection of points exists by e.g. Lemma 3.3 in \cite{OW}. 
Let $\delta > 0$ be a fixed real parameter to be chosen later. We will show that the conclusion of Theorem \ref{thm:wasserstein} holds on the complement of the union of the following ``bad'' events:
\begin{align*}
E_{n}^k &:= \set{N_n^k \geq 2C_\mu e^{2k}\ln(\ln{n})},\ \lfloor 4\ln(\ln{n})\rfloor \leq k \leq \left\lfloor \ln\left(\sqrt{n}\right)\right\rfloor;\\
F_n^i&:= \set{\abs{m_\mu(X_i)} \geq \frac{(\ln{n})^4}{\sqrt{n}},\ \abs{\frac{1}{n-1}\sum_{\substack{j = 1\\j\neq i}}^n\left(\zeta_{i,j}^{(n)} - \E[\zeta_{i,j}^{(n)}|X_i]\right)} \geq \frac{\abs{m_\mu(X_i)}}{2}},\\
&\quad \text{for $1\leq i \leq n$;}\\
G_n^\delta&:= \set{\exists x \in \mathcal{N}_n \cup \set{X_i}_{i=1}^n\ \text{s.t.}\ \#\set{ 1 \leq j \leq n : \abs{X_j-x} < \frac{1}{\sqrt{n}}}\geq 2 + \delta\ln{n}};\\
H_n &: = \set{\eta_n \geq n^{C_\mu}}.
 \end{align*}

For convenience, we use $\mathcal{E}_n^\text{bad}$ to denote the union of all of the ``bad'' events:
\[
\mathcal{E}_n^\text{bad} :=\bigcup_{k=\lfloor 4\ln(\ln{n})\rfloor}^{\left\lfloor \ln\left(c_\mu \sqrt{n}\right)\right\rfloor} E_n^k \cup \bigcup_{i=1}^nF_n^i \cup G_n^\delta \cup H_n.
\]

\subsection{The ``bad'' events are unlikely}
In this subsection, we establish that
\begin{equation}
\P\left(\mathcal{E}_n^\text{bad}\right) = o(1).
\label{eqn:WBadSmall}
\end{equation}
By assumption, $\P(H_n) = o(1)$, so it remains to bound the probabilities of the remaining events. 

\begin{lemma}
\[
\P\left(\bigcup_{k=\lfloor 4\ln(\ln{n})\rfloor}^{\left\lfloor \ln\left(\sqrt{n}\right)\right\rfloor} E_n^k\right) \leq \frac{1}{C_\mu [\ln(\ln{n})]^2} = o(1).
\]
\label{lem:WEprob}
\end{lemma}
\begin{proof}
Observe that for a fixed $n$ and $k$, $\lfloor 4\ln(\ln{n})\rfloor \leq k \leq \left\lfloor \ln\left(\sqrt{n}\right)\right\rfloor$, $N_n^k$ is a binomial random variable with parameters $n$ and $p_k \leq C_\mu e^{2k}/n$. 
By Markov's inequality, we have,
\begin{align*}
\P\left(N^k_n \geq 2C_\mu e^{2k} \ln(\ln{n})\right) &\leq \P\left(\abs{N_n^k -\E\left[N_n^k\right]} \geq C_\mu e^{2k}\ln(\ln{n})\right)\\
&\leq \frac{\var\left(N_n^k\right)}{C_\mu^2e^{4k}[\ln(\ln{n})]^2}\\
&= \frac{np_k(1-p_k)}{C_\mu^2e^{4k}[\ln(\ln{n})]^2}\\
&\leq \frac{1}{C_\mu e^{2k}[\ln(\ln{n})]^2}. 
\end{align*}
If we take the union over $k$, we obtain
\[
\P\left(\bigcup_{k=\lfloor 4\ln(\ln{n})\rfloor}^{\left\lfloor \ln\left(\sqrt{n}\right)\right\rfloor} E_n^k\right) \leq \sum_{k=1}^{\infty} \frac{1}{C_\mu e^{2k}[\ln(\ln{n})]^2} = \frac{1}{C_\mu (e^2-1)[\ln(\ln{n})]^2},
\]
which implies the desired result.
\end{proof}

\begin{lemma}
$\P\left(\bigcup_{i=1}^nF_n^i\right) = o(1).$
\label{lem:WFProb}
\end{lemma}
\begin{proof}
We will use the method of moments to control the probability of each $F_n^i$, $1 \leq i \leq n$. Since $F_n^i \subset \set{\abs{m_\mu(X_i)} \geq n^{-1/2}}$, we will often assume that $\abs{m_\mu(X_i)} \geq n^{-1/2}$ in our calculations. Recall from Lemma \ref{lem:CSnice}, part \eqref{it:CS1} that $\abs{m_\mu(X_i)}$ is almost surely bounded above by an absolute constant (that depends only on $\mu$).

First, consider that for complex-valued random variables $X$ and $Y$, where $Y$ has a finite fourth absolute moment,
\begin{equation}
\begin{aligned}
&\E\left[\abs{Y - \E\left[Y\mid X\right]}^4\;\middle\vert\; X\right]\\
&\qquad\qquad\leq \E\left[\abs{Y}^4\;\middle\vert\; X\right] + 6\left(\E\left[\abs{Y}^2\;\middle\vert\; X\right]\right)^2 + 4\E\left[\abs{Y}^3\;\middle\vert\; X\right]\cdot \E\left[ \abs{Y}\;\middle\vert\; X\right].
\end{aligned}
\label{eqn:WfourY}
\end{equation}
(This inequality could be derived by writing
\[
\abs{Y - \E\left[Y\mid X\right]}^4 = \left(\vphantom{\overline{\E[Y\mid X]}}Y-\E[ Y\mid X]\right)^2\left(\overline{Y}-\overline{\E[Y \mid X]}\right)^2,
\]
expanding the expression at right, and bounding each of the resulting terms with an appropriate term from the right side of \eqref{eqn:WfourY}.)

Now, for $X=X_i$ and $Y = \zeta_{i,j}^{(n)}$, where $1 \leq i,j \leq n$ with $j \neq i$,
\begin{align*}
\E\left[\abs{Y}^4\;\middle\vert\; X_i\right] &\leq \E\left[\frac{1}{\abs{X_i-X_j}^4}\ind_{\frac{(\ln{n})^2}{n\abs{m_\mu(X_i)}} \leq \abs{X_i-X_j}\leq 1} \;\middle\vert\; X_i\right]\\
&\qquad+ \E\left[\frac{1}{\abs{X_i-X_j}^4}\ind_{\abs{X_i-X_j}>1}\;\middle\vert\; X_i\right]\\
&\leq 2\pi C_\mu \int_{\frac{(\ln{n})^2}{n\abs{m_\mu(X_i)}}}^1\frac{r}{r^4}\,dr + 1\\
&= \frac{\pi C_\mu n^2\abs{m_\mu(X_i)}^2}{(\ln{n})^4}- \pi C_\mu + 1,
\end{align*}
and similarly, 
\begin{align*}
\E\left[\abs{Y}^3\;\middle\vert\; X_i\right] &\leq  \frac{2\pi C_\mu n\abs{m_\mu(X_i)}}{(\ln{n})^2}- 2\pi C_\mu + 1,\\
\E\left[\abs{Y}^2\;\middle\vert\; X_i\right] &\leq  2\pi C_\mu \ln\left(\frac{n\abs{m_\mu(X_i)}}{(\ln{n})^2}\right) + 1,\\
\E\left[\abs{Y}\;\middle\vert\; X_i\right] &\leq  2\pi C_\mu  + 1.
\end{align*}
Consequently, via \eqref{eqn:WfourY}, there are positive constants $C'_\mu$, $K_\mu$ that depend only on $\mu$ so  that if $n \geq K_\mu$, on the event $\abs{m_\mu(X_i)} \geq n^{-1/2}$,   
\begin{equation}
\E\left[\abs{\zeta_{i,j}^{(n)}-\E\left[\zeta_{i,j}^{(n)}\;\middle\vert\;X_i\right]}^4\;\middle\vert\; X_i\right] \leq \frac{C'_\mu\abs{m_\mu(X_i)}^2n^2}{(\ln{n})^4}.
\label{eqn:Wzeta4}
\end{equation}
Next, we show that there are constants $C''_\mu, K'_\mu >0$ that depend only on $\mu$, so that for $n\geq K'_\mu$ and any fixed $i$, $1 \leq i \leq n$, 
\begin{equation}
\ind_{\abs{m_\mu(X_i)} \geq \frac{1}{\sqrt{n}}}\cdot\E\left[\abs{\sum_{\substack{j=1\\j\neq i}}^n\left(\zeta_{i,j}^{(n)} - \E\left[\zeta_{i,j}^{(n)}|X_i\right]\right)}^4\;\middle\vert\; X_i\right] \leq\frac{C''_\mu\abs{m_\mu(X_i)}^2n^3}{(\ln{n})^4}.
\label{eqn:WzetaSum} 
\end{equation}
Write 
\begin{align*}
&\E\left[\abs{\sum_{\substack{j=1\\j\neq i}}^n\left(\zeta_{i,j}^{(n)} - \E\left[\zeta_{i,j}^{(n)}|X_i\right]\right)}^4\;\middle\vert\; X_i\right]\\
&\qquad= \E\left[\left(\sum_{\substack{j=1\\j\neq i}}^n\left(\zeta_{i,j}^{(n)} - \E\left[\zeta_{i,j}^{(n)}|X_i\right]\right)\right)^2\left(\overline{\sum_{\substack{j=1\\j\neq i}}^n\left(\zeta_{i,j}^{(n)} - \E\left[\zeta_{i,j}^{(n)}|X_i\right]\right)}\right)^2\;\middle\vert\; X_i\right],
\end{align*}
and observe that if we distribute the factors inside the expectation, the independence of $\set{X_j}_{j=1}^n$ implies that the only terms which contribute to a nonzero expectation are bounded by expectations of the form 
\[
\E\left[\abs{\zeta_{i,j}^{(n)} - \E\left[\zeta_{i,j}^{(n)}\;\middle\vert\;X_i\right]}^2\cdot\abs{\zeta_{i,k}^{(n)}-\E\left[\zeta_{i,k}^{(n)}\;\middle\vert\;X_i\right]}^2\;\middle\vert\; X_i\right],
\]
where $1 \leq j,k \leq n$ and $j,k \neq i$. By a routine counting argument and the fact that $\zeta_{i,j}^{(n)}$, $j \neq i$ are identically distributed, it follows that
\begin{align*}
&\E\left[\abs{\sum_{\substack{j=1\\j\neq i}}^n\left(\zeta_{i,j}^{(n)} - \E\left[\zeta_{i,j}^{(n)}\;\middle\vert\;X_i\right]\right)}^4 \;\middle\vert\; X_i\right]\\
&\hspace{1.2in}\leq (n-1)\E\left[\abs{\zeta_{i,l}^{(n)}-\E\left[\zeta_{i,l}^{(n)}\;\middle\vert\;X_i\right]}^4\;\middle\vert\; X_i\right]\\
&\hspace{1.2in}\qquad+ \binom{n-1}{2}\binom{4}{2}\left(\E\left[\abs{\zeta_{i,l}^{(n)}-\E\left[\zeta_{i,l}^{(n)}\;\middle\vert\;X_i\right]}^2 \;\middle\vert\; X_i\right]\right)^2,
\end{align*}
where $l \neq i$ is any fixed index. From \eqref{eqn:Wzeta4} and the bounds on $\E[\abs{Y^2} \mid X_i]$ and $\E\left[\abs{Y}\mid X_i\right]$ above, we can find $C''_\mu, K'_\mu > 0$ large enough so that $n \geq K'_\mu$ implies \eqref{eqn:WzetaSum}. (For the asymptotics, we are using that $n^{-1/2}\leq \abs{m_\mu(X_i)} = O_\mu(1)$, where the implied constant depends only on $\mu$.) 
Via Markov's inequality, it follows that for $n\geq K'_\mu$ and a fixed $i$, $1\leq i \leq n$, on the event $\abs{m_\mu(X_i)} \geq n^{-1/2}$,
\begin{equation}
\P\left(\abs{\frac{1}{n-1}\sum_{\substack{j=1\\j\neq i}}^n\left(\zeta_{i,j}^{(n)} - \E[\zeta_{i,j}^{(n)}|X_i]\right)} \geq \frac{\abs{m_\mu(X_i)}}{2} \;\middle\vert\; X_i\right) \leq\frac{C''_\mu}{n\abs{m_\mu(X_i)}^2(\ln{n})^4}.
\label{eqn:WXiNearMean}
\end{equation}
%
%
%
%
%
%
%
%
We conclude the proof by demonstrating that $\P(\cup_{i=1}^nF_n^i) = o(1)$. 
Indeed, for $n \geq K'_\mu$,
\begin{align*}
&\P\left(\bigcup_{i=1}^n F_n^i\right)\leq n \P\left(F_n^1\right)\\
&\quad= n \P(\emptyset) + n \sum_{k=\lfloor 4\ln(\ln{n})\rfloor+1}^{\left\lfloor\ln\left(\sqrt{n}\right)\right\rfloor}\P\left(\set{X_1 \in A_n^k} \cap F_n^1\right) + n\cdot \P\left(\set{X_1 \in A_n} \cap F_n^1\right)\\
&\quad= n \sum_{k=\lfloor 4\ln(\ln{n})\rfloor+1}^{\left\lfloor\ln\left(\sqrt{n}\right)\right\rfloor}\E\left(\ind_{\set{X_1 \in A_n^k}} \cdot \P(F_n^1\mid X_1)\right) + n\cdot \E\left(\ind_{\set{X_1 \in A_n}} \cdot \P(F_n^1\mid X_1)\right)\\
&\quad\leq n \sum_{k=\lfloor 4\ln(\ln{n})\rfloor+1}^{\left\lfloor\ln\left(\sqrt{n}\right)\right\rfloor}\E\left( \frac{C''_\mu\cdot\ind_{\set{X_1 \in A^k_n}}}{n\abs{m_\mu(X_1)}^2(\ln{n})^4}\right) + n\cdot \E\left(\frac{C''_\mu\cdot\ind_{\set{X_1 \in A_{n}}}}{n\abs{m_\mu(X_1)}^2(\ln{n})^4}\right)\\
&\quad\leq \sum_{k=\lfloor 4\ln(\ln{n})\rfloor+1}^{\left\lfloor\ln\left(\sqrt{n}\right)\right\rfloor}\frac{C''_\mu\cdot n^2\cdot \P(X_1 \in A^k_n)}{n(\ln{n})^4e^{2k-2}} + \frac{C''_\mu \cdot n^2\cdot \P(X_1 \in A_{n})}{n(\ln{n})^4e^{2\left\lfloor \ln\left(\sqrt{n}\right)\right\rfloor}},
\end{align*}
where we used \eqref{eqn:WXiNearMean} to bound $\P(F^1_n\mid X_1)$. 
Assumption \ref{assum:subquad} guarantees that
\[
\P(X_1 \in A^k_n) \leq C_\mu\cdot \frac{e^{2k}}{n},\ \lfloor 4\ln(\ln{n})\rfloor \leq k \leq \left\lfloor \ln\left(\sqrt{n}\right)\right\rfloor.
\]
We also have
\[
e^{2\left\lfloor \ln\left(\sqrt{n}\right)\right\rfloor} \geq e^{2\ln\left(\sqrt{n}\right)-2} = ne^{-2}.
\]
Hence, for large $n$, our calculation from above yields
\begin{align*}
\P\left(\bigcup_{i=1}^n F^i_n\right) &\leq \sum_{k=1}^{\left\lfloor\ln\left(\sqrt{n}\right)\right\rfloor}\frac{C''_\mu C_\mu e^2}{(\ln{n})^4} + \frac{C''_\mu e^2}{(\ln{n})^4}\cdot 1 = o(1).
\end{align*}

\end{proof}

\begin{lemma}
For a fixed $\delta \in \left(0,\frac{1}{2\pi C_\mu}\right)$,
\[
\P(G_n^\delta) = O_\mu\left(\frac{n^{2+2C_\mu}}{\left(1+\delta\ln{n}\right)^{(2 + \delta\ln{n})}}\right) = o_\delta(1).
\]
\label{lem:WGProb}
\end{lemma}
\begin{proof}
This is a straight-forward application of the Chernoff bound for binomial random variables. In particular, for each $x \in \mathcal{N}_n$, define the random variable \[N_{x}:= \sum_{j=1}^n\ind_{\abs{X_j - x} \leq \frac{1}{\sqrt{n}}},\] which has a binomial distribution with parameters $n$ and $p \leq \pi C_\mu/n$. The moment generating function for $N_x$ is \[\E[e^{tN_x}] = (1 + p(e^t-1))^n \leq e^{np(e^t-1)} \leq e^{\pi C_\mu(e^t-1)}.\] Choosing $t = \ln(1 + 1/(\pi C_\mu)\ln{n})$ establishes 
\[
\E\left[(1 + 1/(\pi C_\mu)\ln{n})^{N_x}\right] \leq n,
\] 
and by Markov's inequality, we obtain
\[
\P\left(N_x \geq 2 + \delta\ln{n}\right) \leq \frac{\E\left[(1 + 1/(\pi C_\mu)\ln{n})^{N_x}\right]}{(1+1/(\pi C_\mu)\ln{n})^{(2+\delta\ln{n})}}\leq \frac{n}{(1+1/(\pi C_\mu)\ln{n})^{(2+\delta\ln{n})}}.
\]
Note that the bound is independent of $x$, and that the argument can be easily modified (by conditioning on $X_i$) to show that for a fixed $1 \leq i \leq n$,
\[
\P\left(\sum_{\substack{{j=1}\\j\neq i}}\ind_{\abs{X_j - X_i} \leq \frac{1}{\sqrt{n}}} \geq 2 + \delta\ln{n}\right) \leq \frac{n}{(1+1/(2\pi C_\mu)\ln{n})^{(2+\delta\ln{n})}}.
\]
Hence, we can apply the the union bound over all $x \in \mathcal{N}_n$ and $X_1, \ldots, X_n$ to obtain the desired result. 
\end{proof}

Combining Lemmas \ref{lem:WEprob}, \ref{lem:WFProb}, and \ref{lem:WGProb} from this subsection establishes \eqref{eqn:WBadSmall}, so for the remainder of the proof, we work on the complements of the ``bad'' events.


\subsection{Constructing disjoint domains that partition the roots}

We will create disjoint domains which contain clusters of roots of $p_n(z)$ that are close to one another and show that inside each domain, the numbers of roots and critical points of $p_n(z)$ are the same. The domains will be disjoint to ensure that no roots or critical points are counted more than once (see Figure \ref{fig:loops} for reference). For technical reasons involving Rouch\'{e}'s theorem, we will require that the boundaries of the regions be simple, closed curves.  

Our strategy will be to make an open ball around each $X_i$, $1 \leq i \leq n$, and to consider the path-connected components of the union of these balls. Some of the resulting regions may not be simply connected, so we need to ``fill in the holes.'' To start, define the random collection of open balls
\[
\mathcal{C}_n := \set{B\left(x, \frac{(\ln{n})^3}{n \cdot \max\set{\abs{m_\mu(x)},\frac{(\ln{n})^4}{\sqrt{n}}}}\right): x \in \set{X_j}_{j=1}^n},
\]
and define on $\set{1, 2, \ldots, n}$ the equivalence relation given by the following rule: $i\sim j$ if and only if there is a collection \[\set{B_0, B_1, \ldots, B_l} \subset \mathcal{C}_n,\]
with
\begin{align*}
B_0 &= B\left(X_i, \frac{(\ln{n})^3}{n\cdot \max\set{\abs{m_\mu(X_i)},\frac{(\ln{n})^4}{\sqrt{n}}}}\right),\\
\intertext{and}
B_l &= B\left(X_j, \frac{(\ln{n})^3}{n\cdot\max\set{\abs{m_\mu(X_j)}, \frac{(\ln{n})^4}{\sqrt{n}}}}\right),
\end{align*}
such that $B_k \cap  B_{k+1} \neq \emptyset$ for $0 \leq k \leq l-1$. Let $\mathcal{P}_n$ be the set of equivalence classes induced by $\sim$. The idea is that for a fixed $P \in \mathcal{P}_n$,
\[
\mathcal{U}_{n,P} := \bigcup_{i \in P} B\left(X_i, \frac{(\ln{n})^3}{n\cdot \max\set{\abs{m_\mu(X_i)},\frac{(\ln{n})^4}{\sqrt{n}}}}\right)
\]
forms a connected component of $\cup_{B\in \mathcal{C}_n} B$. Each light gray region in Figure \ref{fig:loops} is one connected component, $\mathcal{U}_{n,P}$ for some $P \in \mathcal{P}_n$; a ``zoomed-in'' version is presented in Figure \ref{fig:walsh}. Notice that some of the $\mathcal{U}_{n,P}$, $P \in \mathcal{P}_n$ may not have simple, closed boundaries, and some could be ``nested'' inside ``holes'' formed by others. We address these concerns in the following discussion, where we demonstrate how to select a simple, closed component of the boundary of each $\mathcal{U}_{n,P}$, $P \in \mathcal{P}_n$, whose interior contains $\mathcal{U}_{n,P}$. 

More specifically, for each equivalence class $P \in \mathcal{P}_n$, we will create a simple closed curve, $\gamma_{n,P} \subset \partial\mathcal{U}_{n,P}$, such that each $X_j$, $j \in P$ is contained interior to the bounded component of $\C\setminus\gamma_{n,P}$. Furthermore, we will show that the interiors of the bounded regions defined by the curves $\set{\gamma_{n,P}}_{P \in \mathcal{P}_n}$ are partially ordered with respect to set inclusion.  This will allow us to combine ``nested'' regions.

To that end, fix an equivalence class $P \in \mathcal{P}_n$, and recall the definition of the open set $\mathcal{U}_{n,P}$ from above. For simplicity, write
$
\mathcal{U}_{n,P} = \bigcup_{i=1}^l B_i,
$ 
where $B_1, \ldots, B_l$ are \textit{distinct} open balls (in the definition of $\mathcal{U}_{n,P}$, some of the open balls could coincide if, for example $X_i = X_j$ for $i,j \in P$, $i \neq j$).  
We use $\mathcal{V}_{n,P}$ to denote the unique unbounded, path-connected component of the complement of $\overline{\mathcal{U}_{n,P}}$. (The complement of $\overline{\mathcal{U}_{n,P}}$ has a unique unbounded, path-connected component because $\overline{\mathcal{U}_{n,P}}$, a union of finitely  many closed disks, is compact.) 
By construction, the boundaries $\partial \mathcal{U}_{n,P}\supseteq \partial\mathcal{V}_{n,P}$ consist of arcs of the finitely many circles $\partial B_1, \ldots, \partial B_l$.
\begin{lemma}
	The curve $\gamma_{n,P}:=\partial \mathcal{V}_{n,P}$ is a simple, closed curve (i.e. a Jordan curve), and $\mathcal{U}_{n,P}$ is contained in the bounded component of $\C \setminus \gamma_{n,P}$. 
\end{lemma}

\begin{proof}
	There are several ways that one could proceed. One method is to construct a simple path starting on the boundary $\partial\mathcal{V}_{n,P}$ that follows circle arcs until it returns to the start. A second approach is to consider the genus of the region $\mathcal{U}_{n,P}$, find generators for its fundamental group, and ``close-off'' any ``holes.'' We present, in detail, a third method that relies on the following converse of the Jordan curve theorem due to Sch\"{o}nflies (see \cite{FI,TCar}, and the discussion on pp. 13 and 67 of \cite{W}). The theorem statement requires two definitions.
	
	
	A \textit{region} of the closed set $F \subset \C$ is defined as a path-connected component of $\C \setminus F$. A point $x$ in $F$ is \textit{accessible} from a region $\mathcal{R}$ if there is a point $y \in \mathcal{R}$ and a simple path from $y$ to $x$, whose intersection with $F$ is $\set{x}$.
	
	\begin{theorem}[Theorem 1 in \cite{TCar}; see also Theorem II 5.38 on p. 67 of \cite{W}]
		If $F$ is a compact set in $\C$ with precisely two regions such that every point of $F$ is accessible from each of those regions, then $F$ is a simple closed curve. 
		\label{thm:JordanCon}
	\end{theorem}

	Our goal is to show that the compact set $\gamma_{n,P} = \partial \mathcal{V}_{n,P}$ has precisely two regions from which $\gamma_{n,P}$ is accessible at every point. Define $\mathcal{U}'_{n,P} := \C \setminus \overline{\mathcal{V}_{n,P}}$. Observe that $\C \setminus \gamma_{n,P} = \mathcal{V}_{n,P} \cup \mathcal{U}'_{n,P}$, where the union is disjoint. It is clear that $\mathcal{V}_{n,P}$ is a region of $\gamma_{n,P}$; next, we argue that $\mathcal{U}'_{n,P}$ is also a region of $\gamma_{n,P}$.
	
	Since $\mathcal{U}'_{n,P}\subset\C$ is open, it suffices to show that $\mathcal{U}'_{n,p}$ is connected. Suppose, for a contradiction, that this is not the case. Then, there are disjoint, non-empty open sets $S,T\subset \C$ such that $S \cup T = \mathcal{U}'_{n,P}$. By construction, the open set $\mathcal{U}_{n,P} \subset \mathcal{U}'_{n,P}$ is path-connected, and hence connected, so $\mathcal{U}_{n,P}$ must be completely contained in either $S$ or $T$. Suppose, without loss of generality, that $\mathcal{U}_{n,P} \subset S$. Since $T$ is non-empty, there is some $x \in T$. We will demonstrate that a path whose image is contained entirely in $\mathcal{U}'_{n,P}$ connects $x$ to a point of $\mathcal{U}_{n,P} \subset S$, which results in a contradiction. We may assume that $x \notin \partial\mathcal{U}_{n,P}$ because otherwise $x$ lies on a one of the circles $\partial B_i$, $1 \leq i \leq l$, and there is a path in $\mathcal{U}'_{n,P}$ between $x$ and a point of $\mathcal{U}_{n,P} \subset S$.
	
	Since the (finitely many) circles $\partial B_1,\ldots, \partial B_l$ are distinct, there are only finitely many points of $\C$ that are contained in more than one circle. Consequently, we can choose a point $v \in \mathcal{V}_{n,P}$ such that the line segment $\overline{xv}$ does not contain any points of $\C$ that lie in the intersection of two or more distinct $B_i$, $1 \leq i \leq l$. (Indeed, choose a circle $\mathfrak{C}_x \subset \mathcal{V}_{n,P}$, centered at $x$, whose interior contains the compact set $\overline{\mathcal{U'}_{n,P}}$. Then, the collection $\set{\overline{xz} : z \in \mathfrak{C}_x}$ of line segments connecting $x$ to points of $\mathfrak{C}_x$ is infinite in number. Also, $x \notin \partial{U}_{n,P}$ by assumption.) 
	Define the path $\ell:[0,1]\to \C $ via  $t \mapsto tx +(1-t)v$, whose image is the line segment $\overline{xv}$. Since $\overline{xv}$ is 
	connected, it cannot be the case that $\overline{xv} \in \C \setminus \gamma_{n,P}$ (indeed, $\mathcal{U}'_{n,P} \cup \mathcal{V}_{n,P} = \mathcal{C}\setminus \gamma_{n,P}$ is a disjoint union of non-empty open sets). Consequently, $\overline{xv}$ contains a point of $\gamma_{n,P}$. Let $t^* := \min\set{t: \ell(t) \in \gamma_{n,P}}$ and set $y := \ell(t^*)$. Note that $t^* >0$ since $x\notin\mathcal{U}_{n,P}$. 
		
	\begin{figure}
		\begin{minipage}[b]{.4\linewidth}
			\centering
			\begin{tikzpicture}
			\draw[fill=gray!50] (2,0) arc (0:130:2);
			\fill[color=black] (1.414,1.414) circle (2pt);
			\draw[dashed] (1.414,1.414) circle (0.6);
			\draw (1.45,1.45)node[right]{$y$};
			\draw (-.2,1.2)node[above]{$\mathcal{U}_{n,P}$};
			\draw (-.2,2)node[above]{$\mathcal{V}_{n,P}$};
			\end{tikzpicture}
			\subcaption{Case 1: $y$ is on precisely one circle among $\set{\partial B_i}_{i=1}^l$. \label{fig:ballA}}
		\end{minipage}
	\qquad
		\begin{minipage}[b]{.4\linewidth}
			\centering
			\begin{tikzpicture}[x=1.5cm, y=1.5cm]
			\draw[] (1.732,-1) arc (-30:30:2);
			\draw[rotate around={65:(2,0)}] (1.732,-1) arc (-30:30:2);
			\draw[rotate around={110:(2,0)}] (1.732,-1) arc (-30:30:2);
			\fill[color=black] (2,0) circle (2pt);
			\draw[dashed] (2,0) circle (0.6);
			\draw (2.15,-.25) node {$y$};
			\end{tikzpicture}
			\subcaption{Case 2: $y$ is on more than one of the circles $\set{\partial B_i}_{i=1}^l$. \label{fig:ballB}}
		\end{minipage}
		\caption{The geometry near $y \in \gamma_{n,P}$. }
	\end{figure}
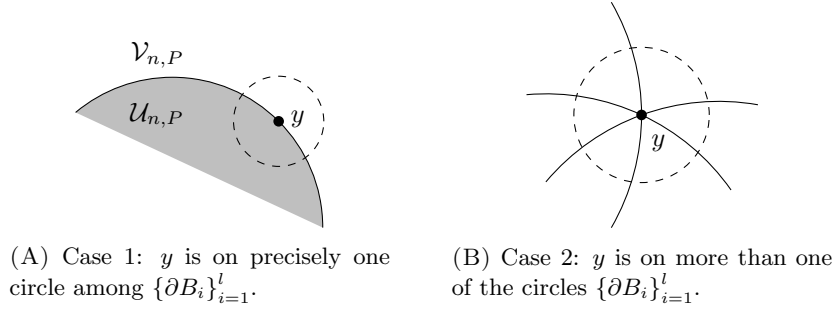	
	
	By construction, $y$ lies on precisely one of the circles $\set{\partial B_i}_{i=1}^l$; suppose, without loss of generality, that $y \in \partial B_1$. Hence, we can choose an open ball $\mathcal{B}_y\ni y$ small enough that $\mathcal{B}_y \setminus \partial B_1$ consists of exactly two disjoint, path-connected open regions (See Figure \ref{fig:ballA}). One of these regions must be a subset of $B_1 \subset \mathcal{U}_{n,P}$, and the other must be a subset of $\mathcal{V}_{n,P}$. (The second region is connected and open, contains no points of $\partial \mathcal{V}_{n,P}$, and must contain a point of $\mathcal{V}_{n,P}$ because $y \in \partial \mathcal{V}_{n,P}$.) 
	
	Choose $\eta>0$ small enough so that $t^*-\eta > 0$ and $\ell(t^*-\eta) \in \mathcal{B}_y$. It follows that the line segment $$L:=\set{\ell(t): 0\leq t \leq t^*-\eta}$$ is connected and disjoint from $\gamma_{n,P}$. We conclude that $L$ is contained entirely in $T$, for it contains $x \in T$. This means $L$ does not contain any points of $\mathcal{V}_{n,P}$, so $\ell(t^*-\eta) \in \mathcal{B}_y \cap B_1 \subset \mathcal{U}_{n,P} \subset S$. We have reached a contradiction since $S$ and $T$ are disjoint, so $\mathcal{U}'_{n,P}$ must be connected. 
	
	We have shown that $\gamma_{n,P}$ has precisely two regions,  $\mathcal{V}_{n,P}$ and $\mathcal{U}'_{n,P}$. It remains to show that every point of $\gamma_{n,P}$ is accessible from both of these regions. Suppose $y \in \gamma_{n,P}$. There are two cases: $y$ is contained in precisely one of $\partial B_i$, $1 \leq i\leq l$, or $y$ is contained in more than one of these circles. (See Figures \ref{fig:ballA} and \ref{fig:ballB}, respectively.)  
	
	If the first case is true, just as we did above, we can choose an open ball $\mathcal{B}_y\ni y$ small enough that $\mathcal{B}_y \setminus \partial B_1$ consists of the two disjoint, path-connected open regions $\mathcal{B}_y \cap \mathcal{U}_{n,P}$ and $\mathcal{B}_y \cap \mathcal{V}_{n,P}$. It is now clear that $y$ is accessible from both $\mathcal{V}_{n,P}$ and $\mathcal{U}'_{n,P} \supset \mathcal{U}_{n,P}$.
	
	On the other hand, suppose, without loss of generality, that $y$ is contained in the circles $\partial B_1, \partial B_2, \ldots, \partial B_j$. Then, we can choose an open ball $\mathcal{B}_y\ni y$ small enough that $\mathcal{B}_y \setminus \bigcup_{i=1}^j\partial B_i$ consists of $2j$ disjoint path-connected, open regions that do not contain points from $\gamma_{n,P}$ (see Figure \ref{fig:ballB}). Consequently, each of these regions must be entirely contained in one of the disjoint open sets $\mathcal{U}'_{n,P}$ or $\mathcal{V}_{n,P}$. Since $y \in \partial \mathcal{U}'_{n,P} = \partial \mathcal{V}_{n,P}$, at least one of the $2j$ regions must be contained in $\mathcal{U}'_{n,P}$ and at least one must be contained in $\mathcal{V}_{n,P}$. It follows that $y$ is accessible from both $\mathcal{V}_{n,P}$ and $\mathcal{U}'_{n,P}$.
	
	We conclude via Theorem \ref{thm:JordanCon} that $\gamma_{n,P}$ is a simple closed curve whose interior contains $\mathcal{U}_{n,P}$ because $\mathcal{U}'_{n,P}$ is the bounded component of $\C \setminus \gamma_{n,P}$, and  $\mathcal{U}_{n,P} \subset \mathcal{U}'_{n,P}$.
 \end{proof}

We have shown that there are simple, closed curves $\set{\gamma_{n,P}}_{P \in \mathcal{P}_n}$ so that for each $P \in \mathcal{P}_n$, $\gamma_{n,P} \subseteq \partial\mathcal{U}_{n,P}$ and $\mathcal{U}_{n,P}$ is contained in the interior of the bounded region defined by $\gamma_{n,P}$. Furthermore, the path-connected, open regions $\set{\mathcal{U}_{n,P}}_{P\in\mathcal{P}_n}$ are disjoint by the definition of the equivalence relation $\sim$. This means that no curve $\gamma_{n,P}$ can pass through the interior of any region $\mathcal{U}_{n,P}$, and as a result, we can  identify ``maximal'' curves which we will use in the remainder of the proof. 

\begin{definition}
We say that a simple, closed curve $\gamma_{n,P^*}$ among $\set{\gamma_{n,P}}_{P \in \mathcal{P}_n}$ is \textit{maximal} if whenever $\mathcal{U}_{n,P^*}$ is in the bounded component of $\C\setminus\gamma_{n,P}$ for some $P \in \mathcal{P}_n$, we have $P = P^*$. We use $\mathcal{M}_n$ to denote the collection of maximal curves. For each $\Gamma \in \mathcal{M}_n$, let $\mathcal{O}_\Gamma$ denote the bounded component of $\C \setminus \Gamma$, so that $\partial\mathcal{O}_\Gamma = \Gamma$. 
\end{definition}

Notice that the domains $\mathcal{O}_\Gamma$, $\Gamma \in \mathcal{M}_n$ are disjoint by construction and that each $X_j$, $1 \leq j \leq n$, is contained in precisely one $\mathcal{O}_\Gamma$. 
We conclude this subsection with two important lemmas that restrict the sizes of the equivalence classes $P$, $P \in \mathcal{P}_n$ and domains $\mathcal{O}_\Gamma$, $\Gamma \in \mathcal{M}_n$.

\begin{lemma}
Suppose $0< \delta < 1/3$. There exists $C_\delta>0$ so that for $n \geq C_\delta$, the following holds on the complement of $G_n^\delta$: for each $P \in \mathcal{P}_n$, $\abs{P} \leq \delta\ln{n} + 2$, and if $x, y \in \overline{\mathcal{U}_{n,P}}$, then,
$
\abs{x-y} < \frac{3\delta}{\sqrt{n}}.
$ 
\label{lem:WPsmall}
\end{lemma}

\begin{proof}
Assume, for a contradiction, that there is a $P \in \mathcal{P}_n$ for which $\abs{P} > \delta\ln{n} + 2$, and suppose, without loss of generality, that $1 \in P$. By the definition of $\mathcal{P}_n$, for each $i \in P \setminus \set{1}$, there are 
elements $B_0^i, B_1^i, \ldots B^i_{l_i} \in \mathcal{C}_n$, where
\begin{align*}
B_0^i &= B\left(X_1, \frac{(\ln{n})^3}{n\cdot\max\set{\abs{m_\mu(X_1)},\frac{(\ln{n})^4}{\sqrt{n}}}}\right),\\
B_{l_i}^i &=  B\left(X_i, \frac{(\ln{n})^3}{n\cdot\max\set{\abs{m_\mu(X_i)},\frac{(\ln{n})^4}{\sqrt{n}}}}\right), 
\end{align*}
$B_{k}^i \cap B_{k+1}^i \neq \emptyset\ \text{for}\ 0\leq k \leq l_i-1,$ and $B_0^i, \ldots, B_{l_i}^i$ are balls with radius at most ${(\ln{n})^{-1}n^{-1/2}}$. Notice that the distance between $X_1$ and any $X_i$, $i \in P\setminus\set{1}$ is bounded by $2 + 2(l_i-1)$ times this maximum radius (recall that $X_1$ and $X_i$, $i \in P\setminus\set{1}$ are the centers of $B_0^i$ and $B_{l_i}^i$, respectively). We consider two cases:
\begin{enumerate}[(i)]
\item \label{it:WPSmall1} for every $i \in P\setminus \set{1}$, $l_i < \delta\ln{n} + 2$
\item \label{it:WPSmall2} there is an $i^* \in P\setminus \set{1}$ for which $l_{i^*} \geq \delta\ln{n} + 2$.
\end{enumerate}
If case \eqref{it:WPSmall1} is true, then, for $n$ large enough to guarantee $\delta \ln{n} \geq 3$,
\[
\max_{i \in P\setminus\set{1}}\abs{X_1 - X_i} < \max_{i \in P\setminus\set{1}} \frac{2+ 2(l_i-1)}{\ln{n}\sqrt{n}}< \frac{2 + 2(\delta\ln{n}+1)}{\ln{n}\sqrt{n}} 
\leq \frac{3\delta}{\sqrt{n}} < \frac{1}{\sqrt{n}},
\]
so every $X_i$, $i \in P$ is in the ball of radius $n^{-1/2}$ centered at $X_1$, which is impossible on the complement of $G_n^\delta$. On the other hand, if case \eqref{it:WPSmall2} is true, then, for large $n$, 
\[
\bigcup_{k=0}^{\left\lceil\delta\ln{n}+2\right\rceil} B_k^{i^*} \subset B\left(X_1, \frac{1}{\sqrt{n}}\right).
\]
Indeed, $\set{B_k^{i^*}}_{k=0}^{\lceil\delta\ln{n}+2\rceil}$ are overlapping balls with radius at most $(\ln{n})^{-1}n^{-1/2}$, so if $n$ is large enough that $\delta\ln{n} \geq 7$ and $y \in  \bigcup_{k=0}^{\lceil\delta\ln{n}+2\rceil} B_k^{i^*}$, then, 
\[
\abs{y-X_1} \leq \frac{1 + 2\lceil\delta\ln{n} +2\rceil}{\ln{n}\sqrt{n}} < \frac{2\delta\ln{n} + 7}{\ln{n}\sqrt{n}} \leq \frac{3\delta}{\sqrt{n}} < \frac{1}{\sqrt{n}}.
\]
This is impossible on the complement of $G_n^\delta$ because it would imply too many roots among $\set{X_j}_{j=1}^n$ in the ball of radius $n^{-1/2}$ centered at $X_1$.

Now, suppose $x, y \in \overline{\mathcal{U}_{n,P}}$ and $n$ is large enough to guarantee that, on the complement of $G_n^\delta$, $\abs{P} \leq \delta\ln{n} + 2$  and $\delta\ln{n} > 4$. Since the path-connected set $\overline{\mathcal{U}_{n,P}}$ consists of $\abs{P}$ overlapping closed disks of radius at most $(\ln{n})^{-1}n^{-1/2}$, we have
\[
\abs{x-y} \leq \abs{P}\frac{2}{\ln{n}\sqrt{n}} \leq \frac{2(\delta\ln{n} + 2)}{\ln{n}\sqrt{n}} < \frac{3\delta}{\sqrt{n}}.
\]
\end{proof}

\begin{corollary}
Suppose  $0< \delta < 1/3$. There exists $C_\delta > 0$ such that for $n \geq C_\delta$, on the complement of $G_n^\delta$, each $\Gamma \in \mathcal{M}_n$ satisfies the following. There exist $x^*,y^* \in \Gamma$ so that if $x,y \in \overline{\mathcal{O}_\Gamma}$, then 
$
\abs{x-y} \leq\abs{x^*-y^*} <\frac{3\delta}{\sqrt{n}}.
$
\label{cor:WOsmall}
\end{corollary}
\begin{proof}
In view of Lemma \ref{lem:WPsmall}, it suffices to show that there exist $x^*, y^* \in \Gamma$ so that
\begin{equation}
\sup_{x,y \in \overline{\mathcal{O}_\Gamma}}\abs{x-y} \leq \abs{x^*-y^*}.
\label{eqn:diamBd}
\end{equation}
(Recall that there exists $P^* \in \mathcal{P}_n$ so that $\Gamma \subset \partial \overline{\mathcal{U}_{n,P^*}}$.)
Since $\overline{\mathcal{O}_\Gamma}$ is compact and $(x,y) \mapsto \abs{x-y}$ is continuous, the extreme value theorem guarantees the existence of $x^*, y^* \in \overline{\mathcal{O}_\Gamma}$ so that the supremum in \eqref{eqn:diamBd} is achieved when $x=x^*$ and $y=y^*$. Suppose, for a contradiction, that  $x^*\notin \Gamma$. Then, $x^*$ is in the open set $\mathcal{O}_\Gamma$, and there is a $\rho > 0$ so that $x^* \in B(x^*,\rho) \subset \mathcal{O}_\Gamma$. Consequently, the line segment $\overline{x^*y^*}$ can be extended along the line connecting $x^*$ and $y^*$ by length $\rho/2$ without leaving $\overline{\mathcal{O}_\Gamma}$. This contradicts the assumption that the supremum in \eqref{eqn:diamBd} is achieved for $x=x^*$, $y=y^*$. We conclude that $x^* \in \Gamma$. A similar argument shows that $y^* \in \Gamma$. 

\end{proof}


\subsection{Pairing of roots and critical points inside each domain}

We now show that on the complement of the ``bad'' events, the roots and critical points within most of the domains $\mathcal{O}_\Gamma$, $\Gamma \in \mathcal{M}_n$ are ``paired.'' The only domains for which this does not occur are those that contain roots of $p_n(z)$ that are ``too close'' to the zeros of $m_\mu$. (See Figure \ref{fig:loops} for reference; recall that $m_\mu(z) = 0$ precisely when $z=0$ in the case where $\mu$ is the uniform measure on the unit disk.) To make ``too close'' rigorous, we define the random collection of roots
\begin{align*}
R^{\text{pair}}_n &:= \set{X_j: 1 \leq j \leq n\ \text{and}\ X_j \in \C \setminus \left(A_n^{\lfloor 4\ln(\ln{n})\rfloor}\cup A_n^{\lfloor 4\ln(\ln{n})\rfloor+1}\right) }\\
&\subseteq \set{X_j: 1 \leq j \leq n\ \text{and}\ \abs{m_{\mu}(X_j)} > \frac{(\ln{n})^4}{\sqrt{n}}}.
\end{align*}

The following lemma is the main result of this subsection. 
\begin{lemma}
For a fixed $\delta>0$ chosen sufficiently small, there is a constant $C_\delta>0$ so that for $n\geq C_\delta$, on the complement of $\cup_{i=1}^nF_n^i \cup G_n^\delta \cup H_n$, the following conclusion holds. For each $\mathcal{O}_\Gamma$, $\Gamma \in \mathcal{M}_n$, such that $\mathcal{O}_\Gamma \cap R^{\text{pair}}_n \neq \emptyset$, the number of critical points of $p_n(z)$ that lie inside $\mathcal{O}_\Gamma$ is equal to the number of roots of $p_n(z)$ that lie inside $\mathcal{O}_\Gamma$ (where both counts include multiplicity). Furthermore, if $X \in \mathcal{O}_\Gamma \cap R^{\text{pair}}_n$ and $w \in \mathcal{O}_\Gamma$ is a critical point of $p_n(z)$, then, 
$
\abs{X-w} \leq \frac{(\ln{n})^4}{n\abs{m_\mu(X)}}.
$
\label{lem:Wmain}
\end{lemma}
\begin{proof}
The proof of this lemma is similar in flavor to the proofs of Theorems \ref{thm:multiInCLT} and \ref{thm:detLoc}, although the argument presented here is much more technical. Fix $n \in \mathbb{N}$, suppose $\mathcal{O}_\Gamma$, $\Gamma \in \mathcal{M}_n$ is such that $\mathcal{O}_\Gamma \cap R^{\text{pair}}_n \neq \emptyset$, and choose an $X \in \mathcal{O}_\Gamma \cap R^{\text{pair}}_n$ to be a distinguished root that will be a reference point in our calculations. We classify the roots $\set{X_j}_{j=1}^n$ into three groups based on their proximity to $X$ (see Figure \ref{fig:walsh}). To that end, define 
\begin{align*}
R_{\text{near}} &:= \set{j: 1 \leq j \leq n,\ \abs{X_j - X} <  \frac{(\ln{n})^2}{n\abs{m_\mu(X)}}}\\
R_{\text{med}} &:= \set{j: 1 \leq j \leq n,\ \abs{X_j - X} <  \frac{1}{\sqrt{n}}} \setminus R_{\text{near}}\\
R_{\text{far}} &:= \set{j: 1 \leq j \leq n,\ \abs{X_j - X} \geq  \frac{1}{\sqrt{n}}},
\end{align*}
and let 
\[
q_X(z) := \!\!\!\!\prod_{j \notin R_{\text{near}}}\!\!\!\!(z-X_j)\quad \text{and}\quad r_X(z):=\!\!\!\!\prod_{j \in R_{\text{near}}}\!\!\!\!(z-X_j),
\]
so that $p_n(z) = q_X(z)r_X(z)$. Note that $\abs{R_{\text{med}}}$ and $\abs{R_\text{near}}$ are of size at most $\delta\ln{n} + 2$ on the complement of $G_n^\delta$. We will compare the zeros of $p_n'(z)$ inside $\mathcal{O}_\Gamma$ to the zeros of the function  
\[
f_X(z) := q_X(z)\left(r_X'(z) + r_X(z)\frac{n-\abs{R_\text{near}}}{z-Y_X}\right)
\]
that are inside $\mathcal{O}_\Gamma$, where $Y_X$ is defined by
\[
Y_X := X - \frac{n-\abs{R_\text{near}}}{\sum_{j \notin R_\text{near}}\frac{1}{X-X_j}}.
\] The idea is that 
\[
\frac{f_X(z)}{p_n(z)} = \frac{r'_X(z)}{r_X(z)} + \frac{n-\abs{R_\text{near}}}{z-Y_X}
\]
is similar to the logarithmic derivative of $p_n(z)$ for $z$ near $X$. Furthermore, the number of roots of the equation
\[
0= r_X'(z) + r_X(z)\frac{n-\abs{R_\text{near}}}{z-Y_X}
\]
that are inside $\mathcal{O}_\Gamma$ will be easy to calculate since these are the same as the critical points of 
\[
\widetilde{p}_X(z):=r_X(z)\cdot (z-Y_X)^{n-\abs{R_\text{near}}}
\]
that lie inside $\mathcal{O}_\Gamma$ (we will show that $Y_X \notin \mathcal{O}_\Gamma$), and these can be located with Walsh's two circle theorem. 

The following lemma contains a few facts that we will frequently reference for the remainder of the proof of Lemma \ref{lem:Wmain}. 

\begin{lemma}
Suppose $\delta < 1/3$. There is a constant $K_{\mu,\delta} \in \mathbb{N}$, depending only on $\mu$ and $\delta$ (and not on $X, P, \Gamma$, etc...), so that $n \geq K_{\mu,\delta}$ implies the following. On the complement of $\cup_{i=1}^nF_n^i \cup G_n^\delta$, if $X \in \mathcal{O}_\Gamma \cap R^{\text{pair}}_n$ and $z\in \overline{\mathcal{O}_\Gamma}$, then  
\begin{enumerate}[(i)]
\item $\displaystyle \abs{z-X} \leq \frac{4\delta(\ln{n})^4}{n\abs{m_\mu(X)}}$, and $\abs{z-X} \geq \frac{(\ln{n})^3}{n\abs{m_\mu(X)}}$ if $z \in \Gamma$; \label{WbdzX}

\item 
$\displaystyle\frac{\abs{m_\mu(X)}}{4} \leq \abs{\frac{1}{n-\abs{R_\text{near}}}\sum_{j \notin R_\text{near}}\frac{1}{X-X_j}} \leq 2\abs{m_\mu(X)}$; \label{WbdCLT}

\item 
$\displaystyle \frac{1}{4\abs{m_\mu(X)}} \leq \abs{z-Y_X} \leq \frac{5}{\abs{m_\mu(X)}},$
so in particular, $f_X(z)$ is analytic in $\mathcal{O}_\Gamma$. \label{WbdzY}
\end{enumerate}
\label{lem:Wfacts}
\end{lemma}
\begin{proof}
Much of this proof relies on the fact that $m_\mu(\cdot)$ is nearly Lipschitz (see Lemma \ref{lem:CSnice} part \eqref{it:CS2}). To establish \eqref{WbdzX}, we first observe that for large $n$, on the complement of $G_n^\delta$, if $\xi \in \overline{\mathcal{O}_\Gamma}$, then
\begin{equation}
\frac{\abs{m_\mu(X)}}{2} \leq \abs{m_\mu(\xi)}\leq \frac{3\abs{m_\mu(X)}}{2}.
\label{eqn:Wmclose}
\end{equation}
Indeed, via Corollary \ref{cor:WOsmall}, $\abs{\xi - X} < \frac{3\delta}{\sqrt{n}} < \frac{1}{\sqrt{n}}$ for large $n$, on the complement of $G_n^\delta$, so as long as we also have $\frac{1}{\sqrt{n}} < \min\set{\eps_\mu,e^{-1}}$, Lemma \ref{lem:CSnice} guarantees that 
\[
\abs{m_\mu(\xi)-m_\mu(X)} \leq \kappa_\mu\frac{3\delta}{\sqrt{n}}\ln\left(\frac{\sqrt{n}}{3\delta}\right).
\]
(We have used the fact that on the interval $[0,e^{-1}]$, the function $-x\ln{x}$ is increasing.) It follows that for $n \geq 5$ and larger than some constant depending on $\mu$ and $\delta$, on the complement of $G_n^\delta$,
\[
\abs{m_\mu(\xi)-m_\mu(X)} \leq \frac{(\ln{n})^2}{\sqrt{n}} \leq \frac{(\ln{n})^4}{2\sqrt{n}} \leq \frac{\abs{m_\mu(X)}}{2},
\]
which implies equation \eqref{eqn:Wmclose}. (The last inequality follows since $X \in R_n^\text{pair}$.) We will use this inequality to compute $\abs{z - X}$, for $z \in \overline{\mathcal{O}_\Gamma}$, in a way that references the balls that we started with when we constructed $\Gamma$.

Let $n$ be large enough to establish \eqref{eqn:Wmclose} and the conclusion of Corollary \ref{cor:WOsmall} on the complement of $G_n^\delta$. Since, $z, X \in \overline{\mathcal{O}_\Gamma}$, Corollary \ref{cor:WOsmall} guarantees the existence of $w_1, w_2 \in \Gamma$ for which $\abs{z-X} \leq \abs{w_1-w_2}$. Recall that $\Gamma \subseteq \partial \mathcal{U}_{n,P^*}$ for some $P^* \in \mathcal{P}_n$, so there are $i_1, i_2 \in P^*$ for which, $X_{i_1}, X_{i_2} \in \mathcal{O}_\Gamma$, and
\begin{align*}
w_1 &\in \partial B\left(X_{i_1}, \frac{(\ln{n})^3}{n \max\set{\abs{m_\mu(X_{i_1})},\frac{(\ln{n})^4}{\sqrt{n}}}}\right)\\
\intertext{and}
w_2 &\in \partial B\left(X_{i_2}, \frac{(\ln{n})^3}{n \max\set{\abs{m_\mu(X_{i_2})},\frac{(\ln{n})^4}{\sqrt{n}}}}\right).
\end{align*}
Furthermore, since $i_1$ and $i_2$ are related by the equivalence that defines $\mathcal{P}_n$, there are open balls $B_0, B_1, \ldots B_{l} \in \mathcal{C}_n$, of the form
\[
B\left(X_j, \frac{(\ln{n})^3}{n\cdot\max\set{\abs{m_\mu(X_j)},\frac{(\ln{n})^4}{\sqrt{n}}}}\right),\ j \in P^*,\ X_j\in \mathcal{O}_\Gamma,
\]
where
\begin{align*}
B_0 &= B\left(X_{i_1}, \frac{(\ln{n})^3}{n\cdot\max\set{\abs{m_\mu(X_{i_1})},\frac{(\ln{n})^4}{\sqrt{n}}}}\right),\\
B_{l} &=  B\left(X_{i_2}, \frac{(\ln{n})^3}{n\cdot\max\set{\abs{m_\mu(X_{i_2})},\frac{(\ln{n})^4}{\sqrt{n}}}}\right), 
\end{align*}
and $B_{k} \cap B_{k+1} \neq \emptyset\ \text{for}\ 0\leq k \leq l-1$. Notice that on the complement of $G_n^\delta$, equation \eqref{eqn:Wmclose} guarantees that the radii of these balls are bounded by $\frac{2(\ln{n})^3}{n\abs{m_\mu(X)}}$ (recall that $X \in R^{\text{pair}}_n$), and if $n$ is large enough to guarantee the conclusion of Lemma \ref{lem:WPsmall}, the number of balls, $l$, is less than $\abs{P^*} \leq \delta\ln{n} + 2$. It follows that for $n$ larger than a constant depending on $\delta$, on the complement of $G_n^\delta$, 
\[
\abs{z-X} \leq \abs{w_1 - w_2} \leq \abs{P^*}\cdot 2\frac{2(\ln{n})^3}{n\abs{m_\mu(X)}}\leq \frac{4(\delta\ln{n} + 2)(\ln{n})^3}{n\abs{m_\mu(X)}} \leq \frac{4\delta(\ln{n})^4}{n\abs{m_\mu(X)}}.
\]
We have established the first half of \eqref{WbdzX}. To see the second inequality, simply recall that $\Gamma$ does not pass through $\mathcal{U}_{n,P}$ for any $P \in \mathcal{P}_n$, so if $z \in \Gamma$, then  
\[
\abs{z-X_j} \geq \frac{(\ln{n})^3}{n\cdot\max\set{\abs{m_\mu(X_j)},\frac{(\ln{n})^4}{\sqrt{n}}}}
\]
for any root $X_j$, $1 \leq j \leq n$. In particular, this is true for $X \in R^{\text{pair}}_n$, which satisfies $\abs{m_\mu(X)} \geq \frac{(\ln{n})^4}{\sqrt{n}}$, so we obtain the second part of \eqref{WbdzX}.

Inequality \eqref{WbdCLT} holds for large $n$ on the complement of $\cup_{i=1}^nF_n^i\cup G_n^\delta$ after several interpolations. 
For each $i$, $1 \leq i \leq n$, the random variables $\E[\zeta_{i,j}^{(n)}\mid X_i]$, $1 \leq j \leq n$, $j \neq i$ are identically distributed, so
\begin{equation}
\begin{aligned}
\abs{\frac{1}{n-1}\sum_{\substack{j=1\\j\neq i}}^n\zeta^{(n)}_{i,j}} &\leq  \abs{\frac{1}{n-1}\sum_{\substack{j=1\\j\neq i}}^n\left(\zeta^{(n)}_{i,j}-\E[\zeta^{(n)}_{i,j}\mid X_i]\right)}\\
&\qquad+ \abs{\E[\zeta^{(n)}_{i,l}\mid X_i] - m_\mu(X_i)} + \abs{m_\mu(X_i)},
\end{aligned}
\label{eqn:WbdCLT2}
\end{equation}
where $l$ is any index different from $i$. 
Since the $X_j$ are iid, we have
\begin{align*}
\abs{\E[\zeta^{(n)}_{i,l}\mid X_i] - m_\mu(X_i)} &= \abs{\E\left[\frac{1}{X_i - X_l}\ind_{\abs{X_i - X_l}< \frac{(\ln{n})^2}{n\abs{m_\mu(X_i)}}}\;\middle\vert\;X_i\right]}\\
&\leq 2\pi C_\mu \int_0^{\frac{(\ln{n})^2}{n\abs{m_\mu(X_i)}}}\frac{1}{r}\cdot r\,dr\\
&= 2\pi C_\mu \frac{(\ln{n})^2}{n\abs{m_\mu(X_i)}},
\end{align*}
so equation \eqref{eqn:WbdCLT2} implies that for any $i$, $1\leq i \leq n$,
\[
\abs{\frac{1}{n-1}\sum_{\substack{j=1\\j\neq i}}^n\zeta^{(n)}_{i,j}} \leq  \abs{\frac{1}{n-1}\sum_{\substack{j=1\\j\neq i}}^n\left(\zeta^{(n)}_{i,j}-\E[\zeta^{(n)}_{i,j}\mid X_i]\right)} + \frac{2\pi C_\mu(\ln{n})^2}{n\abs{m_\mu(X_i)}} +\abs{m_\mu(X_i)}.
\]
Now, $X = X_{i_X}$ for some $i_X$, $1 \leq i_X \leq n$, and $X \in R^{\text{pair}}_n$, so on the complement of $\cup_{i=1}^nF_n^i$,
\begin{equation}
\begin{aligned}
\abs{\frac{1}{n-\abs{R_\text{near}}}\sum_{j \notin R_\text{near}}\frac{1}{X-X_j}} &= \frac{n-1}{n-\abs{R_\text{near}}}\abs{\frac{1}{n-1}\sum_{\substack{j=1\\j\neq i_X}}^n\zeta^{(n)}_{i_X,j}}\\
&\leq \frac{n-1}{n-\abs{R_\text{near}}}\left(\frac{3}{2}\abs{m_\mu(X_{i_X})} + \frac{2\pi C_\mu(\ln{n})^2}{n\abs{m_\mu(X_{i_X})}}\right)\\
&\leq\frac{n-1}{n-\abs{R_\text{near}}}\left(\frac{3}{2}\abs{m_\mu(X)} + \frac{2\pi C_\mu}{\sqrt{n}(\ln{n})^2} \right).
\end{aligned}
\label{eqn:WbdCLT1}
\end{equation}
On the complement of $G_n^\delta$, $\abs{R_\text{near}}$ is at most $\delta\ln{n} + 2$, so for large $n$, on the complement of $\cup_{i=1}^nF_n^i \cup G_n^\delta$ inequality \eqref{eqn:WbdCLT1} establishes the upper bound in \eqref{WbdCLT}. (We have used that $X \in R^{\text{pair}}_n$ to bound $\frac{2\pi C_\mu}{\sqrt{n}(\ln{n})^2}$ above by, say, $1/4\abs{m_\mu(X)}$ for large $n$.) The lower bound in \eqref{WbdCLT} is achieved similarly by using the reverse triangle inequality to obtain
\begin{equation*}
\begin{aligned}
\abs{\frac{1}{n-1}\sum_{\substack{j=1\\j\neq i}}^n\zeta^{(n)}_{i,j}} &\geq  \abs{m_\mu(X_i)}-\abs{\frac{1}{n-1}\sum_{\substack{j=1\\j\neq i}}^n\left(\zeta^{(n)}_{i,j}-\E[\zeta^{(n)}_{i,j}\mid X_i]\right)}\\
&\qquad- \abs{\E[\zeta^{(n)}_{i,l}\mid X_i] - m_\mu(X_i)},
\end{aligned}
\end{equation*}
in place of \eqref{eqn:WbdCLT2}.


We conclude by establishing \eqref{WbdzY} as a consequence of \eqref{WbdzX} and \eqref{WbdCLT}. Indeed, via the triangle inequality, we have for large $n$, on the complement of $\cup_{i=1}^nF_n^i \cup G_n^\delta$, that
\[
\abs{z-Y_X} \leq \abs{z-X} + \abs{\frac{n-\abs{R_\text{near}}}{\sum_{j\notin R_\text{near}\frac{1}{X-X_j}}}} \leq \frac{4\delta(\ln{n})^4}{n\abs{m_\mu(X)}} + \frac{4}{\abs{m_\mu(X)}} \leq \frac{5}{\abs{m_\mu(X)}},
\]
where the rightmost inequality holds for large $n$. The lower bound in \eqref{WbdzY} follows for similar reasons, and $f_X$ is analytic because $\abs{m_\mu(X)}$ is almost surely bounded above by an constant that depends only on $\mu$ (apply Lemma \ref{lem:CSnice}, part \eqref{it:CS1} with $\xi = 0$ and $\rho = +\infty$). 
\end{proof}
%
%
%
%
%
%


The next Lemma justifies our choice of $f_X(z)$ as an intermediate comparison between $p_n(z)$ and $p_n'(z)$ because it establishes that under the right conditions, $f_X(z)$ and $p_n(z)$ have the same number of roots in the domain $\mathcal{O}_\Gamma$.  Consider Figure \ref{fig:walsh} which provides a visual aid to the argument.

\begin{figure}
\includegraphics[width =.85\columnwidth]{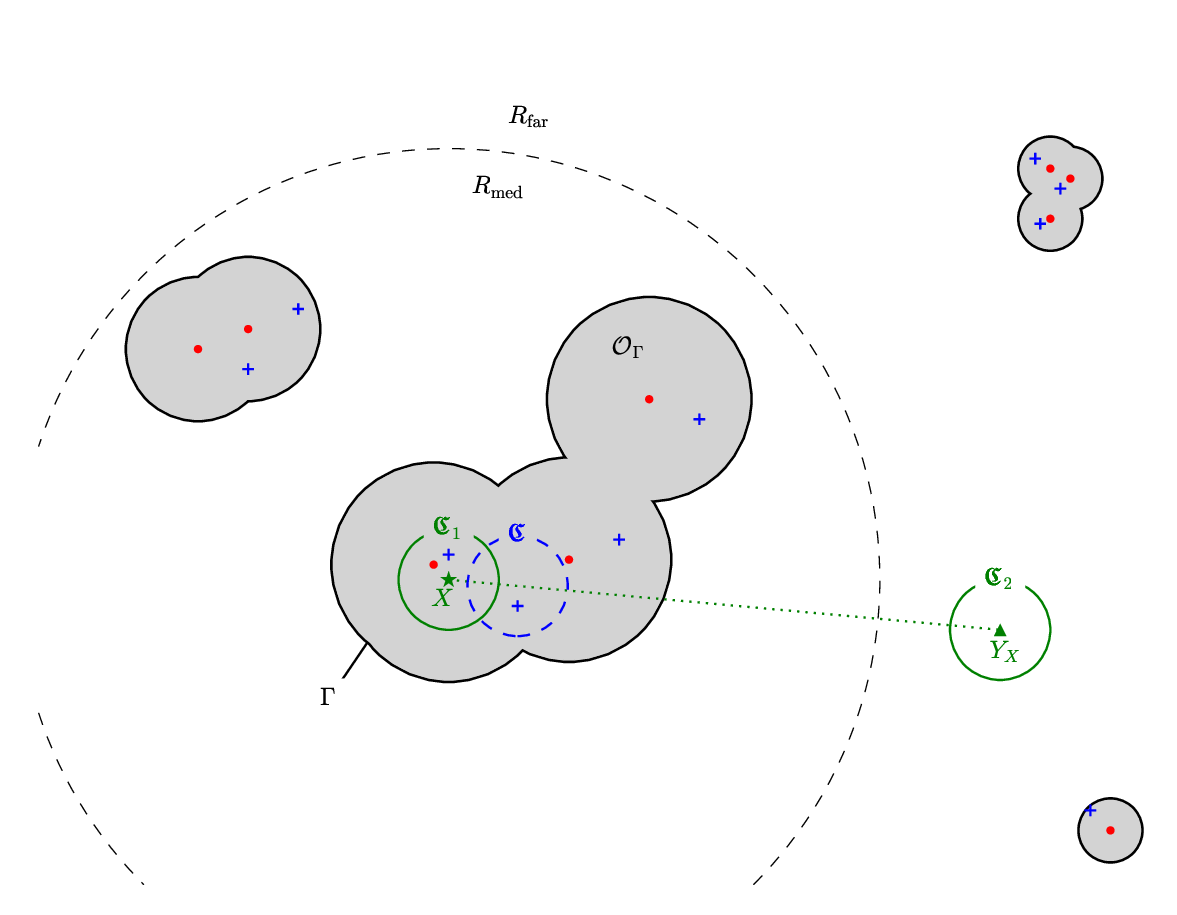}
\caption{A diagram to illustrate Lemma \ref{lem:Wwalsh} and its proof. The red dots and blue crosses are meant to represent roots and critical points, respectively, of $p_n$ that lie in a region near $X$, which is denoted by a green star. The large dashed circle is intended to be on the order of $n^{-1/2}$. Note that indices $1\leq j \leq n$ in $R_\text{near}$ correspond to roots $X_j$ that lie interior to $\mathfrak{C}_1$. This figure is neither to scale nor the result of a simulation.}
\label{fig:walsh}
\end{figure} 

\begin{lemma}
Suppose $\delta < 1/3$. For large $n$, on the complement of $\cup_{i=1}^nF_n^i \cup G_n^\delta$, the polynomial $\widetilde{p}_X(z) = r_X(z)(z-Y_X)^{n-\abs{R_\text{near}}}$ has $\abs{R_\text{near}}$ critical points inside $B\left(X, \frac{5(\ln{n})^2}{n\abs{m_\mu(X)}}\right) \subset \mathcal{O}_\Gamma$, and none of these is $Y_X \notin \mathcal{O}_\Gamma$. In particular, under these conditions, $f_X(z)$ has the same number of roots inside $\mathcal{O}_\Gamma$ as $p_n(z)$ does. 
\label{lem:Wwalsh}
\end{lemma}
\begin{proof}
This follows from Walsh's two circle theorem (see e.g. Theorem 4.1.1 in \cite{RS}.) 
First, we will show that $r_X(z)$ and $\widetilde{p}'_X(z)$ have the same number of roots, $\abs{R_\text{near}}$, inside $\mathcal{O}_\Gamma$ by using Walsh's two circle theorem, and then, we will use this fact to compare the roots of $p_n(z)$ and $f_X(z)$ inside $\mathcal{O}_\Gamma$.

To that end, choose $n$ large enough so that the statements in Lemma \ref{lem:Wfacts} hold on the complement of $\cup_{i=1}^nF_n^i \cup G_n^\delta$, and define the circular domains
\[
\mathfrak{C}_1 := B\left(X, \frac{(\ln{n})^2}{n\abs{m_\mu(X)}}\right) \quad \text{and} \quad \mathfrak{C}_2 := B\left(Y_X, \frac{(\ln{n})^2}{n\abs{m_\mu(X)}}\right).
\]
Note that $\mathfrak{C}_1$ and $\mathfrak{C}_2$ are disjoint for large $n$ on the complement of $\cup_{i=1}^nF_n^i \cup G_n^\delta$ by inequality \eqref{WbdzY} of Lemma \ref{lem:Wfacts}:
\[
\abs{X-Y_X}\geq \frac{1}{4\abs{m_\mu(X)}} > \frac{(\ln{n})^2}{n\abs{m_\mu(X)}}.
\]
In fact, for $n$ large enough, 
\[
\frac{1}{4\abs{m_\mu(X)}}> \frac{4\delta(\ln{n})^4}{n\abs{m_\mu(X)}} + \frac{(\ln{n})^2}{n\abs{m_\mu(X)}},
\]
so on the complement of $\cup_{i=1}^nF_n^i \cup G_n^\delta$, Lemma \ref{lem:Wfacts} part \eqref{WbdzX} guarantees that $\mathfrak{C}_2$ is disjoint from $\mathcal{O}_\Gamma$.

Next, observe that all of the roots of $\widetilde{p}_X(z)$ lie in $\mathfrak{C}_1 \cup \mathfrak{C_2}$, so by Walsh's two circle theorem, the critical points of $\widetilde{p}_X$ lie in $\mathfrak{C}_1 \cup \mathfrak{C}_2 \cup \mathfrak{C}$, where $\mathfrak{C}$ is the open ball 
\[
\mathfrak{C} := B\left(\frac{n-\abs{R_\text{near}}}{n}X + \frac{\abs{R_\text{near}}}{n}Y_X,\frac{(\ln{n})^2}{n\abs{m_\mu(X)}}\right).
\]
By Lemma \ref{lem:Wfacts}, for large $n$, on the complement of $\cup_{i=1}^nF_n^i \cup G_n^\delta$, $x \in \mathfrak{C}$ implies 
\begin{align*}
\abs{x-X} &\leq \abs{\frac{n-\abs{R_\text{near}}}{n}X + \frac{\abs{R_\text{near}}}{n}Y_X - X} +  \frac{(\ln{n})^2}{n\abs{m_\mu(X)}}\\
&= \frac{\abs{R_\text{near}}}{n}\abs{\frac{n-\abs{R_\text{near}}}{\sum_{j \notin R_\text{near}}\frac{1}{X-X_j}}} + \frac{(\ln{n})^2}{n\abs{m_\mu(X)}}\\
&\leq \frac{\abs{R_\text{near}}}{n}\frac{4}{\abs{m_\mu(X)}} + \frac{(\ln{n})^2}{n\abs{m_\mu(X)}}\\
&\leq \frac{4(\delta\ln{n}+2) + (\ln{n})^2}{n\abs{m_\mu(X)}}\\
&< \frac{5(\ln{n})^2}{n\abs{m_\mu(X)}},
\end{align*}
where the last inequality holds for large $n$. It follows that for large $n$, on the complement of $\cup_{i=1}^nF_n^i \cup G_n^\delta$,
\[
\mathfrak{C} \subseteq B\left(X, \frac{5(\ln{n})^2}{n\abs{m_\mu(X)}}\right) \subseteq B\left(X, \frac{(\ln{n})^3}{n\cdot\max\set{\abs{m_\mu(X)},\frac{(\ln{n})^4}{\sqrt{n}}}}\right)
\]
(recall $X \in R^{\text{pair}}_n$), so in particular, $\mathfrak{C} \cup \mathfrak{C}_1$ is contained in $\mathcal{O}_\Gamma$, and this union is disjoint from $\mathfrak{C}_2$. Consequently, by the Supplement Theorem 4.1.1 in \cite{RS},  for large $n$, on the complement of $\cup_{i=1}^nF_n^i \cup G_n^\delta$, $\widetilde{p}'_X(z)$ has $\abs{R_\text{near}}$ roots inside $\mathcal{O}_\Gamma$, just like $r_X(z)$ does. Under these conditions, $f_X(z)$ has the same roots as $q_X(z)\widetilde{p}'_X(z)$ inside $\mathcal{O}_\Gamma$ because $Y_X \notin \mathcal{O}_\Gamma$, so it follows that $f_X(z)$ and $p_n(z)= q_X(z)r_X(z)$ have the same number of roots inside $\mathcal{O}_\Gamma$. 
\end{proof}

We conclude this subsection with two lemmas and an application of Rouch\'{e}'s theorem to establish that $f_X(z)$ and $p_n'(z)$ have the same numbers of zeros in $\mathcal{O}_\Gamma$. This will imply via Lemma \ref{lem:Wwalsh} that $p_n(z)$ and $p_n'(z)$ have the same numbers of zeros in $\mathcal{O}_\Gamma$.

\begin{lemma}
Suppose $\delta < 1/8$. There exist positive constants $\widetilde{C}_\mu$, dependent only on $\mu$,  and $C_{\mu,\delta}$, dependent only on $\mu$ and $\delta$ (and not on $X$, $\Gamma$, etc...), so that for $n \geq C_{\mu,\delta}$, on the complement of $\cup_{i=1}^n F_n^i\cup G_n^\delta \cup H_n$, if $z \in \Gamma$, 
\begin{equation}
\abs{p_n'(z) - f_X(z)} \leq \abs{p_n(z)}\widetilde{C}_\mu\delta^2n\abs{m_\mu(X)}
\label{eqn:Wlbd}
\end{equation}
(here, $\widetilde{C}_\mu$ is independent of $\delta$).
\label{lem:RoucheLB}
\end{lemma}
\begin{proof}
Fix $z \in \Gamma$. By the definition of $f_X(z)$ and the triangle inequality, we have
\begin{equation}
\begin{aligned}
&\abs{p'_n(z)-f_X(z)}\\
&\quad= \abs{p_n(z)}\cdot\abs{\sum_{j=1}^n\frac{1}{z-X_j} - \frac{r'_X(z)}{r_X(z)} - \frac{n-\abs{R_\text{near}}}{z-Y_X}}\\
&\quad= \abs{p_n(z)}\cdot\abs{\sum_{j \notin R_\text{near}}\hspace{-2mm}\frac{1}{z-X_j}- \frac{n-\abs{R_\text{near}}}{z-Y_X}}\\
&\quad\leq  \abs{p_n(z)}\left(\abs{\sum_{j \notin R_\text{near}}\hspace{-2mm}\frac{1}{z-X_j} - \sum_{j \notin R_\text{near}}\hspace{-2mm}\frac{1}{X-X_j}} + \abs{\sum_{j \notin R_\text{near}}\hspace{-2mm}\frac{1}{X-X_j} - \frac{n-\abs{R_\text{near}}}{z-Y_X}}\right).
\end{aligned}
\label{eqn:leftRoucheStart}
\end{equation}
We find upper bounds for the two terms at right separately. First, factor $\abs{ \sum_{j \notin R_\text{near}}\frac{1}{X-X_j}}$ from the second term, and combine the resulting fractions to obtain
\begin{align*}
&\abs{\sum_{j \notin R_\text{near}}\hspace{-2mm}\frac{1}{X-X_j} - \frac{n-\abs{R_\text{near}}}{z-Y_X}}\\
&\quad=\abs{\sum_{j \notin R_\text{near}}\frac{1}{X-X_j}}\abs{\frac{(z-X)\sum_{j \notin R_\text{near}}\frac{1}{X-X_j}}{(z-X)\sum_{j \notin R_\text{near}}\frac{1}{X-X_j}+n - \abs{R_\text{near}}}}\\
&\quad= \abs{z-X}\abs{\frac{1}{n-\abs{R_\text{near}}}\sum_{j \notin R_\text{near}}\hspace{-2mm}\frac{1}{X-X_j}}^2\abs{\frac{n-\abs{R_\text{near}}}{(z-X)\frac{1}{n - \abs{R_\text{near}}}\sum_{j \notin R_\text{near}}\frac{1}{X-X_j}+1}}.
\end{align*}
Lemma \ref{lem:Wfacts} and an application of the reverse triangle inequality to the denominator of the rightmost factor yield for large $n$, on the complement of $\cup_{i=1}^nF_n^i\cup G^\delta_n$,
\begin{equation}
\begin{aligned}
&\abs{\sum_{j \notin R_\text{near}}\hspace{-2mm}\frac{1}{X-X_j} - \frac{n-\abs{R_\text{near}}}{z-Y_X}}\\
&\qquad\leq\frac{4\delta(\ln{n})^4}{n\abs{m_\mu(X)}}\cdot 4\abs{m_\mu(X)}^2\cdot \frac{n-\abs{R_\text{near}}}{1-\abs{z-X}\abs{\frac{1}{n-\abs{R_\text{near}}}\sum_{j \notin R_\text{near}}\frac{1}{X-X_j}}}\\
&\qquad= O_\delta\left((\ln{n})^4\abs{m_\mu(X)}\right).
\end{aligned}
\label{eqn:Wkllbd}
\end{equation}
We now find a bound on the first term in \eqref{eqn:leftRoucheStart}. Combining the summands gives 
\begin{equation*}  
\begin{aligned}
&\abs{\sum_{j \notin R_\text{near}}\hspace{-2mm}\frac{1}{z-X_j} - \sum_{j \notin R_\text{near}}\hspace{-2mm}\frac{1}{X-X_j}}\\
&\qquad= \abs{z-X}\abs{\sum_{j \notin R_\text{near}}\hspace{-1mm}\frac{1}{(z-X_j)(X-X_j)}}\\
&\qquad\leq \abs{z-X}\left(\abs{\sum_{j \in R_\text{med}}\frac{1}{(z-X_j)(X-X_j)}} + \abs{\sum_{j \in R_\text{far}}\frac{1}{(z-X_j)(X-X_j)}}\right),
\end{aligned}
\end{equation*}
so in view of \eqref{eqn:Wkllbd} and the upper bound Lemma \ref{lem:Wfacts} gives for $\abs{z-X}$, we can establish \eqref{eqn:Wlbd} by showing  
%
%
there exist positive constants $\widetilde{C}'_\mu$, $C'_{\mu,\delta}$ satisfying the following: on the complement of $\cup_{i=1}^nF_n^i\cup G^\delta_n$, $n \geq C'_{\mu,\delta}$ implies 
\begin{equation}
\abs{\sum_{j \in R_\text{med}}\frac{1}{(z-X_j)(X-X_j)}} + \abs{\sum_{j \in R_\text{far}}\frac{1}{(z-X_j)(X-X_j)}} \leq \frac{\widetilde{C}'_\mu\delta n^2\abs{m_\mu(X)}^2}{(\ln(n))^{4}}.
\label{eqn:Wkbd}
\end{equation}

We will bound each term on the left separately. By construction of the sets $\set{\mathcal{O}_\Gamma}_{\Gamma \in \mathcal{M}_n}$, recall that the curves $\Gamma \in \mathcal{M}_n$ do not intersect the interiors of the open balls forming $\mathcal{U}_{n,P}$, $P \in \mathcal{P}_n$. Hence, for $j \in R_\text{med}$,
\[
\abs{z - X_j} \geq \frac{(\ln{n})^3}{n\cdot\max\set{\abs{m_\mu(X_j)},\frac{(\ln{n})^4}{\sqrt{n}}}}.
\]
By Lemma \ref{lem:CSnice}, it follows that for large $n$, $\abs{m_\mu(X_j)} \leq 2\abs{m_\mu(X)}$ (Recall that for $j \in R_\text{med}$, $\abs{X-X_j} < \frac{1}{\sqrt{n}}$ and $X \in R^{\text{pair}}_n$). Consequently, for large $n$, on the complement of $\cup_{i=1}^nF_n^i \cup G_n^\delta$, 
\[
\abs{z-X_j} \geq \frac{(\ln{n})^3}{n\cdot\max\set{2\abs{m_\mu(X)},\frac{(\ln{n})^4}{\sqrt{n}}}} \geq \frac{(\ln{n})^3}{2n\abs{m_\mu(X)}}.
\]
In addition, for $j \in R_\text{med}$, $\abs{X-X_j} \geq \frac{(\ln{n})^2}{n\abs{m_\mu(X)}}$. Hence, for $n$ large, on the complement of $\cup_{i=1}^nF_n^i \cup G_n^\delta$, 
\begin{equation}
\abs{\sum_{j \in R_\text{med}}\frac{1}{(z-X_j)(X-X_j)}} \leq \abs{R_\text{med}} \frac{2n^2\abs{m_\mu(X)}^2}{(\ln{n})^5} \leq (\delta\ln{n}+2)\frac{2n^2\abs{m_\mu(X)}^2}{(\ln{n})^5}.
\label{eqn:Wkbd1}
\end{equation}

We now turn our attention to the second term on the left side of \eqref{eqn:Wkbd}. Since $\mu$ is absolutely continuous with respect to the Lebesgue measure on $\C$, we expect that the number of $X_j$, $j \in R_\text{far}$ within a given distance of $X$ is roughly proportional to the square of that distance, and hence, the sum over $j \in R_\text{far}$ in \eqref{eqn:Wkbd} should be roughly on the order of $n$. To take advantage of this intuition, we will split the sum into pieces by grouping terms $[(z-X_j)(X-X_j)]^{-1}$ according to the distance between $X_j$ and $X$. To that end, define for $1 \leq k \leq \sqrt{n}-1$, the annuli
\begin{align*}
D_{k,n} &:= \set{w \in \C : \frac{k}{\sqrt{n}} \leq \abs{w-X} \leq \frac{k+1}{\sqrt{n}}}\\
D_{k,n}' &:= \set{w \in \C : \frac{k-1}{\sqrt{n}} \leq \abs{w-X} \leq \frac{k+2}{\sqrt{n}}}\\
D_{k,n}'' &:= \set{w \in \C : \frac{k-2}{\sqrt{n}} \leq \abs{w-X} \leq \frac{k+3}{\sqrt{n}}}
\end{align*}
and the random variables
\[
\#_{k,n} := \#\set{j: 1\leq j \leq n,\ X_j \in D_{k,n}}.
\]
(Note that $D'_{1,n}, D''_{1,n}, D''_{2,n}$ are disks.) On the complement of $H_n$, each $X_j$, $1 \leq j \leq n$ is within $n^{-1/2}$ of some $x_j \in \mathcal{N}_n$, and on the complement of $G^\delta_n$, there are at most $2 + \delta\ln{n}$ roots $X_l$, $1\leq l \leq n$ within $n^{-1/2}$ of $x_j$. It follows that
\begin{equation}
\#_{k,n} \leq \abs{\mathcal{N}_n \cap D'_{k,n}} \cdot (\delta\ln{n} +2).
\label{eqn:WnumAbd}
\end{equation}
We will argue that due to the fact that any distinct $x,y \in \mathcal{N}_n$ are separated by at least $\frac{1}{2\sqrt{n}}$, the size of $\mathcal{N}_n \cap D_{k,n}'$ is bounded by $16^2k$. Indeed, for any distinct $x,y \in \mathcal{N}_n$, the balls $B(x,n^{-1/2}/4)$ and $B(y,n^{-1/2}/4)$ are disjoint, and if $x \in D_{k,n}'$, then 
$
B(x,n^{-1/2}/4) \subset D''_{k,n}.
$
The area of $D''_{k,n}$ for $k \geq 2$ is $\frac{\pi}{n}(10k + 5)$, so at most $16(10k+5)$ disjoint balls of radius $n^{-1/2}/4$ can fit in $D''_{k,n}$. Similarly, at most $16^2$ balls of radius $n^{-1/2}/4$ can fit in $D''_{1,n}$. Combining this with equation \eqref{eqn:WnumAbd} establishes that, on the complement of $G_n^\delta \cup H_n$, we have $\#_{k,n} \leq 16^2k(\delta\ln{n} + 2)$.

We can now find an upper bound for the second term on the left of \eqref{eqn:Wkbd} by breaking this sum into pieces that correspond to the annuli $D_{k,n}$, $1 \leq k \leq \sqrt{n}-1$. To start, observe that
\begin{equation}
\begin{aligned}
&\sum_{j \in R_\text{far}}\frac{1}{\abs{z-X_j}\abs{X-X_j}}\\
&\leq \sum_{k=1}^{\sqrt{n}-1} \sum_{j: X_j \in D_{k,n}}\frac{1}{\abs{z-X_j}\abs{X-X_j}} + \sum_{j: \abs{X-X_j} \geq 1} \frac{1}{\abs{z-X_j}}\\
&\leq \sum_{k=1}^{\sqrt{n}-1}\sum_{j: X_j \in D_{k,n}}\frac{1}{\left(\abs{X-X_j}-\abs{z-X}\right)\abs{X-X_j}}+ \sum_{j: \abs{X-X_j} \geq 1} \frac{1}{\abs{X-X_j}-\abs{z-X}},
\end{aligned}
\label{eqn:WkFar1}
\end{equation}
where the last line follows from the reverse triangle inequality (the next equation justifies why $\abs{z-X}$ is smaller than $\abs{X-X_j}$). 
If $\delta < 1/8$ and $n$ is large enough to guarantee the conclusions of Lemma \ref{lem:Wfacts}, then on the complement of $\cup_{i=1}^nF_n^i \cup G_n^\delta \cup H_n$, for $j \in R_\text{far}$, we have 
\[
\abs{X-X_j} \geq \frac{1}{\sqrt{n}} > \frac{8\delta}{\sqrt{n}} \geq \frac{8\delta(\ln{n})^4}{n\abs{m_\mu(X)}} \geq 2\abs{z-X}
\] 
(note that $\abs{m_\mu(X)} \geq \frac{(\ln{n})^4}{\sqrt{n}}$). Substituting $\abs{X-X_j}/2$ for $\abs{z-X}$ into the last line of \eqref{eqn:WkFar1} establishes that for large $n$,
on the complement of $\cup_{i=1}^nF_n^i \cup G_n^\delta\cup H_n$, 
\[
\sum_{j \in R_\text{far}}\frac{1}{\abs{z-X_j}\abs{X-X_j}} \leq \sum_{k=1}^{\sqrt{n}-1}\sum_{j: X_j \in D_{k,n}}\frac{2}{\abs{X-X_j}^2} + \sum_{j: \abs{X-X_j} \geq 1} \frac{2}{\abs{X-X_j}},
\]
at which point we can use the fact that $\abs{X_j - X} \geq \frac{k}{\sqrt{n}}$ for $X_j \in D_{k,n}$, $1 \leq k \leq \sqrt{n}-1$ to obtain
\[
\sum_{j \in R_\text{far}}\frac{1}{\abs{z-X_j}\abs{X-X_j}}\leq \sum_{k=1}^{\sqrt{n}-1}\#_{k,n}\frac{2n}{k^2}+2n \leq \sum_{k=1}^{\sqrt{n}-1}\frac{2\left(16^2k(\delta\ln{n} +2)\right)n}{k^2}+2n.
\]
By approximating $\sum_{k-1}^{\sqrt{n}-1}k^{-1}$ with $1 + \int_{1}^{\sqrt{n}-1}x^{-1}\,dx$, we see that this last expression is on the order of $O\left(\delta n(\ln{n})^2\right)$,
where the implied constant is independent of $\delta$. 
Together with \eqref{eqn:Wkbd1}, this establishes equation \eqref{eqn:Wkbd} since $\abs{m_\mu(X)} \geq \frac{(\ln{n})^4}{\sqrt{n}}$. 
\end{proof}

The last lemma in this subsection establishes a lower bound on $\abs{f_X(z)}$ that will combine with \eqref{eqn:Wlbd} to fulfill the hypotheses of Rouch\'{e}'s theorem on the boundary $\Gamma$ of the domain $\mathcal{O}_\Gamma$.

\begin{lemma}
For fixed $\delta>0$, there is a constant $\check{C}_{\mu,\delta}$ depending only on $\mu$ and $\delta$ so that when $n \geq \check{C}_{\mu,\delta}$, on the complement of $\cup_{i=1}^nF_n^i \cup G_n^\delta$, if $z \in \Gamma$, 
\begin{equation}
\abs{f_X(z)} \geq \abs{p_n(z)}n\abs{m_\mu(X)}\cdot e^{-9}.
\label{eqn:Wubd}
\end{equation}
\label{lem:RoucheUB}
\end{lemma}
\begin{proof}
We have
\[
\abs{r_X(z)} = \prod_{X_j \in R_\text{near}}\abs{z-X_j}\leq \left(\abs{z-X} + \frac{(\ln{n})^2}{n\abs{m_\mu(X)}}\right)^{\abs{R_\text{near}}}\\
\]
and
\[
\abs{\frac{f_X(z)}{q_X(z)}} = \frac{1}{\abs{z-Y_X}}\abs{(z-Y_X)r_X'(z) + r_X(z)\left(n-\abs{R_\text{near}}\right)}.
\]
By Lemma \ref{lem:Wwalsh}, for large $n$, on the complement of $\cup_{i=1}^nF_n^i \cup G_n^\delta$, the polynomial expression
\[
(z-Y_X)r_X'(z) + r_X(z)\left(n-\abs{R_\text{near}}\right) = \frac{\widetilde{p}'_X(z)}{(z-Y_X)^{n-\abs{R_\text{near}}-1}}
\]
has degree $\abs{R_\text{near}}$, leading coefficient $n$, and $\abs{R_\text{near}}$ roots in $B\left(X, \frac{5(\ln{n})^2}{n\abs{m_\mu(X)}}\right) \subset \mathcal{O}_\Gamma$. It follows that 
under these conditions,
\begin{align*}
\abs{\frac{f_X(z)}{q_X(z)}} 
&= \frac{n}{\abs{z-Y_X}}\prod_{\substack{w\in \mathcal{O}_\Gamma\\ \widetilde{p}'_X(w)=0}}\abs{z-w}\\
&\geq \frac{n}{\abs{z-Y_X}}\left(\abs{z-X} - \frac{5(\ln{n})^2}{n\abs{m_\mu(X)}}\right)^{\abs{R_\text{near}}},
\end{align*}
where the critical points of $\widetilde{p}_X(z)$ that index the product are considered with multiplicity.


 If additionally, $\delta < 1$ and $n$ is large enough to guarantee the bounds on $\abs{z-X}$ in Lemma \ref{lem:Wfacts}, 
we have that on the complement of $\cup_{i=1}^nF_n^i \cup G_n^\delta$ and for $z \in \Gamma$,
\[
\frac{(\ln{n})^2}{n\abs{m_\mu(X)}} \leq \frac{\abs{z-X}}{\delta\ln{n}}.
\]
Hence, if $n$ is large enough, on the complement of $\cup_{i=1}^nF_n^i \cup G_n^\delta$, 
for $z \in \Gamma$,  
\begin{align*}
\abs{f_X(z)} &= \abs{p_n(z)}\cdot\frac{1}{\abs{r_X(z)}} \cdot \abs{\frac{f_X(z)}{q_X(z)}}\\
&\geq \abs{p_n(z)} \cdot \frac{n}{\abs{z-Y_X}} \cdot \left(\frac{\abs{z-X}\left(1- \frac{5}{\delta\ln{n}}\right)}{\abs{z-X}\left(1+\frac{1}{\delta\ln{n}}\right)}\right)^{\abs{R_\text{near}}}\\
&\geq \abs{p_n(z)}\cdot \frac{n\abs{m_\mu(X)}}{5} \cdot \left(\frac{1- \frac{5}{\delta\ln{n}}}{1+\frac{1}{\delta\ln{n}}}\right)^{\delta\ln{n}+2}\\
&\geq \abs{p_n(z)}n\abs{m_\mu(X)}\cdot e^{-9}.
\end{align*}
We have used Lemma \ref{lem:Wfacts} to bound $\abs{z-Y_X}$, and the last inequality holds for large $n$ and comes from the fact that
\[
\left(1 +\frac{x}{\delta\ln{n}}\right)^{\delta\ln{n}+2} \xrightarrow{n \to\infty} e^{ x}.
\]
(Note that the rate of convergence possibly depends on $\delta$.) We have achieved \eqref{eqn:Wubd} as was desired.
\end{proof}


We have now established both \eqref{eqn:Wlbd} and \eqref{eqn:Wubd}, where the inequalities are independent of $X$, $\Gamma$, and $z\in \Gamma$. Since $\widetilde{C}_\mu$ is independent of $\delta$, we can choose $\delta \in (0,1/8)$ small enough that $\widetilde{C}_\mu\delta^2 < e^{-9}$. For such a $\delta$, by Lemmas \ref{lem:RoucheLB} and \ref{lem:RoucheUB}, for large $n$, on the complement of $\cup_{i=1}^nF_n^i\cup G_n^\delta\cup H_n$, any $z \in \Gamma$ satisfies
\[
\abs{p'_n(z) - f_X(z)} < \abs{f_X(z)}.
\]
It follows by Rouch\'{e}'s theorem that for large $n$, on the complement of $\cup_{i=1}^nF_n^i\cup G_n^\delta\cup H_n$, $p'_n(z)$ and $f_X(z)$ have the same number of zeros inside $\mathcal{O}_\Gamma$, and by Lemma \ref{lem:Wwalsh}, we conclude that $p'_n(z)$ and $p_n(z)$ have the same number of zeros inside $\mathcal{O}_\Gamma$. The inequality in the conclusion of Lemma \ref{lem:Wmain} follows directly from this and Lemma \ref{lem:Wfacts} part \eqref{WbdzX} (note $\delta \leq 1/4$).

In the argument above, the particular curve $\Gamma \in \mathcal{M}_n$ and the root $X \in \mathcal{O}_\Gamma \cap R_n^\text{pair}$ were arbitrary, and all of the constants involved were independent of $\Gamma$, so we have proved Lemma \ref{lem:Wmain}.
\end{proof}

\subsection{Bounding the Wasserstein distance}

In this subsection, we use Lemma \ref{lem:Wmain} to prove Theorem \ref{thm:wasserstein}. Let $w_1^{(n)}, \ldots, w_{n-1}^{(n)}$ denote the (not necessarily distinct) critical points of $p_n(z)$, and recall the definitions of the empirical measures, $\mu_n$ and $\mu'_n$ (see \eqref{eq:def:mun}) 
Since the numbers of roots and critical points of a polynomial differ by one, we first compare the measure $\mu_n'$ to the intermediate measure
\[
\tilde{\mu}'_n := \frac{1}{n}\left(\delta_{\overline{X}} + \sum_{j=1}^{n-1}\delta_{w_j^{(n)}}\right),\ \text{where}\ \overline{X} = \frac{1}{n}\sum_{j=1}^nX_j.
\]
The following lemma justifies our choice of $\tilde{\mu}'_n$.
\begin{lemma}
Let $\mu'_n$, $\tilde{\mu}'_n$, and $\eta_n:=\max_{1\leq j \leq n}\abs{X_j}$ be defined as above. Then, with probability $1$,
$
W_1(\mu_n', \tilde{\mu}'_n) \leq \frac{2\eta_n}{n}.
$
\label{lem:WWhyTilde}
\end{lemma}
\begin{proof}
Let $\widetilde{\pi}$ be the measure on $\C \times \C$ given by
\[
\widetilde{\pi} := \frac{1}{n} \sum_{j=1}^{n-1} \delta_{(w_j^{(n)}, w_j^{(n)})} + \frac{1}{n(n-1)}\sum_{j=1}^{n-1}\delta_{(w_j^{(n)},\overline{X})},
\]
whose marginal distributions are easily seen to be $\mu'_n$ and $\tilde{\mu}'_n$. It follows from the definition of the $L_1$-Wasserstein metric that, almost surely,
\begin{align*}
W_1(\mu'_n, \tilde{\mu}'_n) &\leq \frac{1}{n}\sum_{j=1}^{n-1}\abs{w_j^{(n)} - w_j^{(n)}} + \frac{1}{n(n-1)}\sum_{j=1}^{n-1}\abs{w_j^{(n)} - \overline{X}}\leq 0 + \frac{1}{n} \cdot 2\eta_n,
\end{align*}
where the last inequality follows from the Gauss--Lucas theorem. 
\end{proof}

The next result is an $L_1$-Wasserstein comparison between $\mu_n$ and $\tilde{\mu}'_n$ that we will use in conjunction with Lemma \ref{lem:WWhyTilde} and the triangle inequality to prove Theorem \ref{thm:wasserstein}.
\begin{lemma}
Let $X_1, \ldots, X_n$ be iid, complex random variables with distribution $\mu$ that has a bounded density and satisfies Assumption \ref{assum:subquad}. Then, there is a constant $C$, depending only on $\mu$, so that 
with probability $1 - o(1)$,
$
W_1(\mu_n, \tilde{\mu}'_n) \leq \frac{C\eta_n(\ln{n})^{9}}{n},
$
where $\mu_n$, $\tilde{\mu}'_n$, and $\eta_n$ are defined as above.
%
\label{lem:WassMuTilde}
\end{lemma}
\begin{proof}

Suppose $w_1^{(n)},\ldots, w_{n-1}^{(n)}$ are critical points of $p_n(z)$ defined above, and define $w_n^{(n)} := \overline{X}$. Then, for any permutation $\sigma_n$ of $\set{1,2,\ldots, n}$, the measure
\[
\pi_{\sigma_n} := \frac{1}{n}\sum_{j=1}^n\delta_{(X_j, w_{\sigma_n(j)}^{(n)})}
\]
has marginal distributions $\mu_n$ and $\tilde{\mu}'_n$, so 
\begin{equation}
W_1(\mu_n, \tilde{\mu}_n') \leq \int\abs{x-y}\,d\pi_{\sigma_n}(x,y) = \frac{1}{n}\sum_{j=1}^n\abs{X_j - w_{\sigma_n(j)}^{(n)}}.
\label{eqn:WUpBd}
\end{equation}
We will now make a judicious choice of $\sigma_n$ in order to take advantage of the ``clumping'' behavior of the roots and critical points of $p_n(z)$ proclaimed in the conclusion of Lemma \ref{lem:Wmain}.

To start, define the index sets $S_\Gamma$, $\Gamma \in \mathcal{M}_n$ by
\[
S_\Gamma := \set{1 \leq j \leq n : X_j \in \mathcal{O}_\Gamma}.
\] 
For large $n$, on the complement of $\mathcal{E}_n^\text{bad}$, Lemma \ref{lem:Wmain} guarantees that each $\mathcal{O}_\Gamma$, $\Gamma \in \mathcal{M}_n$ satisfying $\mathcal{O}_\Gamma \cap R_n^\text{pair} \neq \emptyset$ contains the same number of critical points and roots of $p_n(z)$. Consequently, we can choose $\sigma_n$ so that for each $\Gamma \in \mathcal{M}_\Gamma$ satisfying $\mathcal{O}_\Gamma \cap R_n^\text{pair}\neq \emptyset$, we have
\[
\sigma_n(S_\Gamma) = \set{1 \leq j \leq n-1: w_j^{(n)} \in \mathcal{O}_\Gamma} 
\] 
(recall that $\mathcal{O}_\Gamma$, $\Gamma \in \mathcal{M}_n$ are pairwise disjoint). For the remaining indices whose images under $\sigma_n$ we haven't specified, arbitrarily assign them from among the remaining choices. (There is at least one index $1 \leq i \leq n$ for which $\sigma_n(i)$ is still undefined because the number of roots and critical points of $p_n(z)$ differs by $1$. Recall that we have added $w_n^{(n)} = \overline{X}$ to account for this fact.)

Based on our construction of $\sigma_n$, Lemma \ref{lem:Wmain} also implies that for large $n$, on the complement of $\mathcal{E}_n^\text{bad}$,
\begin{equation}
\abs{X_j - w_{\sigma_n(j)}^{(n)}} \leq \frac{(\ln{n})^4}{n\abs{m_\mu(X_j)}}\ \text{when $X_j \in R_n^\text{pair}$},\ 1\leq j \leq n
\label{eqn:WPairBd}
\end{equation}
(Indeed, $X_j \in R_n^\text{pair}$ implies that $X_j \in \mathcal{O}_\Gamma$, with $\mathcal{O}_\Gamma\cap R_n^\text{pair} = \emptyset$ for some $\Gamma \in \mathcal{M}_n$.)
%
%
%
%
%
%
By the Gauss--Lucas theorem, each critical point (and each root) of $p_n(z)$ is in the convex hull of the set $\set{X_j}_{j=1}^n$ of roots of $p_n(z)$, so for any $X_j \notin R_n^\text{pair}$, we have the trivial bound
\begin{equation}
\abs{X_j - w_{\sigma_n(j)}^{(n)}} \leq 2\eta_n.
\label{eqn:WTrivialBd}
\end{equation}
We now find an upper bound for $W_1(\mu_n, \tilde{\mu}'_n)$ by splitting the sum on the right side of \eqref{eqn:WUpBd} into many pieces. To start, we classify the terms based on whether or not $X_j \in R_n^\text{pair}$ and apply \eqref{eqn:WPairBd} and \eqref{eqn:WTrivialBd} to obtain
\begin{equation}
W_1(\mu_n, \tilde{\mu}_n')\leq \frac{1}{n}\sum_{j: X_j \notin R_n^\text{pair}}2\eta_n + \frac{1}{n}\sum_{j: X_j \in R_n^\text{pair}} \frac{(\ln{n})^4}{n\abs{m_\mu(X_j)}},
\label{eqn:WDistSplit}
\end{equation}
for large $n$, on the complement of $\mathcal{E}_n^\text{bad}$. According to the definition of $R_n^\text{pair}$, any $X_j\notin R_n^\text{pair}$ is an element of $A^{\lfloor 4\ln(\ln{n})\rfloor} \cup A^{\lfloor 4\ln(\ln{n})\rfloor+1}$, so there are $N_n^{\lfloor 4\ln(\ln{n})\rfloor} + N_n^{\lfloor 4\ln(\ln{n})\rfloor+1}$ terms in the leftmost sum. It follows that for large $n$, on the complement of $\mathcal{E}_n^\text{bad}$,
\begin{equation}
\sum_{j: X_j \notin R_n^\text{pair}}\hspace{-2mm}2\eta_n \leq 4\eta_nC_\mu\ln(\ln{n})\left( e^{2\lfloor 4\ln(\ln{n})\rfloor} + e^{2\lfloor 4\ln(\ln{n})\rfloor+2} \right) \ll_\mu (\ln{n})^9\eta_n.
\label{eqn:WUnpairBd}
\end{equation}
It remains to control the rightmost sum in \eqref{eqn:WDistSplit}, which we accomplish by grouping indices $j$, $1 \leq j \leq n$ that correspond to $X_j$ that fall in the same region among $A_n$ and $A_n^k$, $\lfloor 4\ln(\ln{n})\rfloor+2\leq k \leq \lfloor \ln(\sqrt{n})\rfloor$. We have,
\begin{align*}
\sum_{j: X_j \in R_n^\text{pair}} \frac{(\ln{n})^4}{n\abs{m_\mu(X_j)}}&=\sum_{k = \lfloor 4\ln(\ln{n})\rfloor+2}^{\lfloor \ln(\sqrt{n})\rfloor}\sum_{j: X_j \in A_n^k}\frac{(\ln{n})^4}{n\abs{m_\mu(X_j)}} + \sum_{j: X_j \in A_n}\frac{(\ln{n})^4}{n\abs{m_\mu(X_j)}}\\
&\leq\sum_{k = \lfloor 4\ln(\ln{n})\rfloor+2}^{\lfloor \ln(\sqrt{n})\rfloor}N_n^k\cdot\frac{(\ln{n})^4\sqrt{n}}{ne^{k-1}} + n\cdot \frac{(\ln{n})^4\sqrt{n}}{ne^{\lfloor \ln(\sqrt{n})\rfloor}},
%
%
\end{align*}
where the last inequality follows after applying the lower bounds on $\abs{m_\mu(X_j)}$ that are given in the definitions of $A_n$ and $A_n^k$ (note that there are $N_n^k$ terms in each inner sum, and the very last sum has at most $n$ terms). For large $n$, on the complement of $\mathcal{E}_n^\text{bad}$, we have $N_n^k < 2C_\mu e^{2k}\log(\log{n})$, for $\lfloor 4\log(\log{n})\rfloor\leq k \leq \lfloor\log(\sqrt{n})\rfloor$, so, continuing from above,
\begin{align*}
\sum_{j: X_j \in R_n^\text{pair}} \frac{(\ln{n})^4}{n\abs{m_\mu(X_j)}} &\leq 2C_\mu\ln(\ln{n})\frac{e(\ln{n})^4}{\sqrt{n}}\sum_{k = \lfloor 4\ln(\ln{n})\rfloor+2}^{\lfloor \ln(\sqrt{n})\rfloor}e^k + \frac{(\ln{n})^4\sqrt{n}}{e^{\lfloor \ln(\sqrt{n})\rfloor}}\\
&\leq 2C_\mu\ln(\ln{n})\frac{e(\ln{n})^4}{\sqrt{n}}\cdot \log(\sqrt{n})e^{\log(\sqrt{n})} + \frac{e(\ln{n})^4\sqrt{n}}{\sqrt{n}}\\
&\ll_\mu (\log{n})^6.
\end{align*}
In view of the fact that  $\P(\mathcal{E}_n^\text{bad}) = o(1)$, we complete the proof of Lemma \ref{lem:WassMuTilde} by combining the last asymptotic with \eqref{eqn:WDistSplit} and \eqref{eqn:WUnpairBd}. 
(Note that with probability $1-o(1)$, $\eta_n\ln{n} \geq 1$.)
%
\end{proof}

We conclude this subsection by remarking that Theorem \ref{thm:wasserstein} follows from Lemmas \ref{lem:radSymAssump}, \ref{lem:WWhyTilde}, and \ref{lem:WassMuTilde} and the triangle inequality for the $L_1$-Wasserstein metric. 

\appendix

\section{Proof of assorted results from Section \ref{sec:results}}\label{sec:appendixA}
\subsection{Computation of $m_\mu$ for radially symmetric distributions}

The following lemma is useful for computing the Cauchy--Stieltjes transforms of radially symmetric distributions, which can expedite the verification of Assumptions \ref{assum:subquad} and \ref{assum:radSym} in a variety of situations. We note that Lemma \ref{lem:radSymMu} also appears as Proposition 3.1 in \cite{KS}. 

\begin{lemma}
	Suppose $\mu$ has a density $f(r,\theta) = f(r)$ that is radially symmetric about the origin. Then, $m_\mu(0) = 0$, and for $\xi \neq 0$, 
	\[
	m_\mu(\xi) = \frac{2\pi}{\xi}\int_0^{\abs{\xi}}rf(r)\,dr = \frac{1}{\xi}\P(\abs{X} < \abs{\xi}),
	\]
	where $X \sim \mu$.
	\label{lem:radSymMu}
\end{lemma}

\begin{proof} For $\xi \neq 0$, we can use polar coordinates and Laurent series to obtain
	\begin{align*}
	m_\mu(\xi) &= \int_0^{2\pi} \int_0^\infty \frac{f(r)}{\xi - re^{i\theta}}\cdot r\,dr\,d\theta\\
	&= \frac{1}{\xi}\int_0^{\abs{\xi}} r f(r) \int_0^{2\pi}\frac{1}{1 - \frac{r}{\xi}e^{i\theta}}\,d\theta\,dr -  \frac{1}{\xi}\int_{\abs{\xi}}^\infty r f(r) \int_0^{2\pi}\frac{\frac{\xi}{r}e^{-i\theta}}{1-\frac{\xi}{r}e^{-i\theta}}\,d\theta\,dr\\
	&= \frac{1}{\xi}\int_0^{\abs{\xi}} r f(r) \sum_{j=0}^\infty\int_0^{2\pi}\left(\frac{r}{\xi}e^{i\theta}\right)^j\,d\theta\,dr\\
	&\qquad\qquad -  \frac{1}{\xi}\int_{\abs{\xi}}^\infty r f(r) \sum_{j=0}^\infty\frac{\xi}{r}\int_0^{2\pi}e^{-i\theta}\left(\frac{\xi}{r}e^{-i\theta}\right)^j\,d\theta\,dr.
	\end{align*}
	The only nonzero integral occurs when the power on the exponential is $0$, so we obtain
	\[
	m_\mu(\xi) = \frac{2\pi}{\xi}\int_0^{\abs{\xi}}rf(r)\,dr 
	\] 
	as is desired. Finally, observe that
	\[
	m_\mu(0) = \int_0^\infty\int_0^{2\pi} \frac{-f(r)}{re^{i\theta}}\cdot r\,d\theta\,dr = -\int_0^\infty f(r)\int_0^{2\pi} e^{-i\theta}\,d\theta\,dr = 0.
	\]
\end{proof}

\subsection{Calculations for Example \ref{ex:twoDisks}}

In this subsection, we establish that $\mu$ from Example \ref{ex:twoDisks} satisfies part \eqref{it:subquad} of Assumption \ref{assum:subquad}. Setting each of the three branches in \eqref{eqn:twoCirc} to zero shows that the only zeros of $m_\mu(z)$ are when $z = 0, \pm \sqrt{3}$. We claim that there is a $C>0$ such that if $X\sim \mu$, and $\eps >0$ is small, then, $\P(\abs{m_\mu(X)} < \eps) \leq C\eps^2$. To start, consider that for $\abs{z+2} < 1$,
\begin{align*}
\abs{m_\mu(z)} &= \frac{1}{2(z-2)}\abs{(z-2)(\overline{z+2}) + 1}\\
&=\frac{1}{2\abs{z-2}}\abs{\left(z+\sqrt{3}- \sqrt{3} - 2\right)\left(\overline{z + \sqrt{3}} - \sqrt{3} + 2\right) +1}\\
&=\frac{1}{2\abs{z-2}} \abs{\abs{z+\sqrt{3}}^2 + (2-\sqrt{3})(z+\sqrt{3}) - (2+\sqrt{3})(\overline{z+\sqrt{3}})}\\
&=\frac{\abs{z+\sqrt{3}}}{2\abs{z-2}}\abs{\overline{(z+\sqrt{3})} + 2-\sqrt{3} - (2+\sqrt{3})\frac{\overline{z+\sqrt{3}}}{z+\sqrt{3}}}.
\end{align*}
Since $\abs{z + 2} < 1$, it follows that $\abs{z-2} < 6$ and also, by the triangle inequality,
\[
\abs{z+\sqrt{3}} \leq \abs{z+2} + \abs{\sqrt{3}-2} \leq 1 +2-\sqrt{3} = 3-\sqrt{3}.
\]
Hence, the reverse triangle inequality transforms our previous calculation into
\begin{align*}
\abs{m_\mu(z)} &\geq \frac{\abs{z+\sqrt{3}}}{12}\left(\abs{(2+\sqrt{3})\frac{\overline{z+\sqrt{3}}}{z+\sqrt{3}}} - \abs{z+\sqrt{3}} - \abs{2-\sqrt{3}}\right)\\
&\geq \frac{\abs{z+\sqrt{3}}}{12} \left((2+\sqrt{3}) - (3-\sqrt{3}) - (2-\sqrt{3})\right)\\
&= \frac{\sqrt{3}-1}{4}\abs{z+\sqrt{3}},
\end{align*}
for $\abs{z+2} < 1$. Similarly, $\abs{z-2}<1$ implies that
\[
\abs{m_\mu(z)} \geq \frac{\sqrt{3}-1}{4}\abs{z-\sqrt{3}}.
\]
Since the random variable $X$ can only take values $z$ for which $\abs{z\pm 2} < 1$, it follows that 
\[
\P\left(\abs{m_\mu(X)} < \eps\right) \leq \P\left(\abs{X+\sqrt{3}} < c\eps\right) + \P\left(\abs{X-\sqrt{3}} < c\eps\right) \leq \frac{c^2\eps^2}{2},
\]
where $c= 4/(\sqrt{3}-1)$ and $\eps > 0$ is small enough that $B(\sqrt{3}, c\eps) \subset B(2,1)$. 

\subsection{Proof of Lemma \ref{lemma:critanti}}

\begin{proof}
	Fix $a > 0$, and let $b > 0$ be a large constant (depending on $C$ and $a$) to be chosen later.  Since $Z$ is independent of $X_1, \ldots, X_n$ it follows that, with probability $1$, $Z \not\in \{X_1, \ldots, X_n\}$.  Hence the sum
	$ \sum_{i=1}^n \frac{1}{Z - X_i} $
	is well-defined and finite.  By conditioning on the values of $X_2, \ldots, X_n$ and $Z$, it suffices to prove that
	\[ \sup_{w \in \mathbb{C}} \sup_{z \in B(0, n^C)} \Prob \left( \left| \frac{1}{z - X_1} - w \right| \leq n^{-b} \right) \ll_a n^{-a}. \]
	The claim now follows from Lemma \ref{lemma:singleanti} below by taking $\eps := n^{-b}$ and choosing $b$ sufficiently large in terms of $C$ and $a$.  
\end{proof}

\begin{lemma} \label{lemma:singleanti}
	Fix $C > 0$, and let $X$ be a complex-valued random variable that is absolutely continuous (with respect to Lebesgue measure on $\mathbb{C}$) and which has density bounded by $n^C$.  If $\E|X| \leq n^C$, then for every $a > 0$ and $0 < \eps < 1$, 
	\[ \sup_{w \in \mathbb{C} } \sup_{z \in B(0, n^C)} \Prob \left( \left| \frac{1}{z - X} - w \right| \leq \eps \right) \leq 4 \sqrt{\eps} n^C + 4 \pi \eps n^{3C + 2a}  + n^{-a}. \]
\end{lemma}
\begin{proof}
	Fix $w \in \mathbb{C}$ and $z \in B(0, n^C)$.  We consider two cases.  If $|w| \leq \sqrt{\eps}$, then
	\begin{align*}
	\Prob \left( \left| \frac{1}{z - X} - w \right| \leq \eps \right) &\leq \Prob \left( \left| \frac{1}{z - X} \right| \leq 2 \sqrt{\eps} \right) \\
	&\leq \Prob \left( |X - z| \geq \frac{1}{2 \sqrt{\eps} } \right) \\
	&\leq 2 \sqrt{\eps} \left( \E|X| + n^C \right) \\
	&\leq 4 \sqrt{\eps} n^C
	\end{align*}
	by Markov's inequality.  
	
	We now consider the case where $|w| > \sqrt{\eps}$.  Define the event
	$ \mathcal{E} := \{ |X| \leq n^{C+a} \}. $
	By Markov's inequality, it follows that
	$ \Prob(\mathcal{E}^c) \leq n^{-a}. $
	Thus, we obtain
	\begin{align*}
	\Prob \left( \left| \frac{1}{z - X} - w \right| \leq \eps \right) &\leq \Prob \left( \left| \frac{1}{w} - (z - X) \right| \leq \sqrt{\eps} |z - X| \right) \\
	&\leq \Prob \left( \left| \frac{1}{w} - (z - X) \right| \leq \sqrt{\eps} |z - X| \;\middle\vert\; \mathcal{E} \right) \Prob(\mathcal{E}) + \Prob(\mathcal{E}^c) \\
	&\leq \Prob \left( \left| \frac{1}{w} - (z - X) \right| \leq 2\sqrt{\eps} n^{C+a} \right) + n^{-a} \\
	&\leq \Prob \left( X \in B(w^{-1} - z, 2 \sqrt{\eps} n^{C+a}) \right) + n^{-a} \\
	&\leq \pi n^C \left( 2 \sqrt{\eps} n^{C+a} \right)^2 + n^{-a}. 
	\end{align*}
	
	Combining the bounds above yields 
	\[ \Prob \left( \left| \frac{1}{z - X} - w \right| \leq \eps \right) \leq 4 \sqrt{\eps} n^C + 4 \pi \eps n^{3C + 2a}  + n^{-a} \]
	for any $w \in \mathbb{C}$ and $z \in B(0, n^C)$.  The proof of the lemma is complete.  
\end{proof}

\section{A heavy-tailed CLT} \label{sec:appendixB}

In this subsection, we prove Theorem \ref{thm:genCLT}, a CLT for ``heavy-tailed'' random variables that have the same distribution as $Y := \frac{1}{\xi-X}$, where $X\sim \mu$ and $\mu$ has a continuous density $f$ in a neighborhood of $\xi$. Notice that $\E\abs{Y}^p < \infty$ for $p \in [0, 2)$, but $\E\abs{Y}^2 = \infty$. Many results demonstrate that $Y$ is in the domain of attraction of a normal random variable (see e.g. Section XVII.5 in \cite{F}, Theorem 11 in Section 6.4 of \cite{Gal},  and Theorem 3.10 in \cite{P}), however, our implementation of Theorem \ref{thm:genCLT} requires specific information about the parameters of the limiting normal distribution; we include an explicit statement and proof for clarity.


\begin{theorem}
Let $X_1, X_2, \ldots $ be iid, complex-valued random variables with common distribution $\mu$, fix $s,k \in \mathbb{N}$, and suppose  $\xi_1, \ldots, \xi_s, t_1, \ldots, t_k \in \C$ are deterministic values with  $\xi_1, \ldots, \xi_s$ distinct.  In addition, assume that $\mu$ has a bounded density $f$ in a neighborhood of each $\xi_l$, $1\leq l\leq s$, that is continuous at these points. Then,
\[
\frac{1}{\sqrt{n\ln{n}}}\sum_{j=1}^n\sum_{k=1}^st_k\left[\frac{1}{\xi_k-X_j} - m_\mu(\xi)\right] \longrightarrow N
\]
in distribution as $n\to \infty$, where $N$ is a complex random variable with mean zero whose real and imaginary parts have a joint Gaussian distribution that has covariance matrix 
\begin{equation}
\Sigma:= \sum_{k=1}^s \frac{\pi\abs{t_k}^2f(\xi_k)}{2}I. 
\label{eqn:genCLT}
\end{equation}
(Here, $I$ denotes the $2 \times 2$ identity matrix.)
\label{thm:genCLT}
\end{theorem}
\begin{proof}
We proceed by Lindeberg's exchange method \cite{L}. (See also \cite{C}. Similar methods have been applied to problems in random matrix theory; see e.g. \cite{TV2011}, \cite{TVpoly}.) To that end, let $N, N_1, N_2, \ldots$ be a sequence of iid complex random variables independent of $\set{X_j}$, whose components have a joint Gaussian distribution with mean zero and covariance matrix $\Sigma$, defined in \eqref{eqn:genCLT}, and let $g:\C \to \R$ be a smooth test function with compact support. We will show that
\begin{equation}
\abs{\E\left[g\left(\frac{1}{\sqrt{n\ln{n}}}\sum_{j=1}^n\sum_{k=1}^st_k\left[\frac{1}{\xi_k-X_j} - m_\mu(\xi_k)\right]\right)\right] - \E\left[g\left(\frac{1}{\sqrt{n}}\sum_{j=1}^nN_j\right)\right]}\to 0,
\label{eqn:portmanteau}
\end{equation}
as $n \to \infty$, which implies convergence of the corresponding measures in the vague topology. Convergence in distribution follows because for each $n$, $n^{-1/2}\sum_{j=1}^n N_j$ has the same distribution as the random variable $N$. (See e.g. Exercise 1.1.25 of \cite{Ttop}, pages 23-33.)

 Since the random variables $\sum_{k=1}^s\frac{t_k}{\xi_k-X_j}$ are heavy-tailed, we initially need to truncate them. Let $\eps \in (0,1)$ be fixed, and define 
\begin{align*}
\zeta_j &:= \sum_{k=1}^s\frac{t_k}{\xi_k - X_j}\ind_{\{\abs{\xi_k-X_j}^{-1}<\eps\sqrt{n\ln{n}} \}},\\
\widetilde{\zeta}_j &:= \zeta_j - \E[\zeta_j].
\end{align*}
(Be aware that this notation suppresses the dependence of $\zeta_j$ and $\widetilde{\zeta}_j$ on $\eps$ and $n$.)
\begin{lemma}
There is a constant $C_{\mu,s,\vec{t}}>0$, depending only on $\mu$, $s$, and $t_1, \ldots, t_k$, and there is a natural number $K_{\mu,g,\eps}$ so that $n \geq K_{\mu,g,\eps}$ implies
\[
\abs{\E\left[g\left(\frac{1}{\sqrt{n\ln{n}}} \sum_{j=1}^n\widetilde{\zeta}_j\right)\right]- \E\left[g\left(\frac{1}{\sqrt{n}}\sum_{j=1}^nN_j\right)\right]} \leq C_{\mu,s,\vec{t}}\eps.
\]
\label{lem:genCLTtrunc}
\end{lemma}

\begin{proof}
By Taylor's theorem applied to the Taylor series for $g$ centered at \[A_{1,n} := \frac{1}{\sqrt{n\ln{n}}}\sum_{j=2}^n\widetilde{\zeta}_j,\] we have
\begin{align*}
g\left(\frac{1}{\sqrt{n\ln{n}}}\sum_{j=1}^n\widetilde{\zeta}_j\right) &= g\left(A_{1,n}\right)+ \frac{g_x\left(A_{1,n}\right)}{\sqrt{n\ln{n}}}\Re\left(\widetilde{\zeta}_1\right) + \frac{g_y\left(A_{1,n}\right)}{\sqrt{n\ln{n}}}\Im\left(\widetilde{\zeta}_1\right)\\
&+ \frac{g_{xx}\left(A_{1,n}\right)}{2n\ln{n}}\Re\left(\widetilde{\zeta}_1\right)^2 + \frac{g_{yy}\left(A_{1,n}\right)}{2n\ln{n}}\Im\left(\widetilde{\zeta}_1\right)^2\\
&+  \frac{g_{xy}\left(A_{1,n}\right)}{n\ln{n}}\Re\left(\widetilde{\zeta}_1\right)\Im\left(\widetilde{\zeta}_1\right) + R_3\left(\frac{1}{\sqrt{n\ln{n}}}\sum_{j=1}^n\widetilde{\zeta}_j\right),
\end{align*}
where 
\[
\abs{R_3\left(\frac{1}{\sqrt{n\ln{n}}}\sum_{j=1}^n\widetilde{\zeta}_j\right)} \leq \frac{8\cdot C_g}{2!\cdot (n\ln{n})^{3/2}}\abs{\widetilde{\zeta}_1}^3,
\]
and $C_g$ is any constant that is an upper bound for the mixed partial derivatives of $g$ up to and including order three (which are compactly supported and thus bounded). Taking the expectation of both sides yields (by independence and the fact that $\widetilde{\zeta}_j$ are centered) 
\begin{align*}
&\E\left[g\left(\frac{1}{\sqrt{n\ln{n}}}\sum_{j=1}^n\widetilde{\zeta}_j\right)\right]\\ &\qquad= \E[g(A_{1,n})] + \frac{\E[g_{xx}\left(A_{1,n}\right)]}{2n\ln{n}}\E\left[\Re\left(\widetilde{\zeta}_1\right)^2\right] + \frac{\E[g_{yy}\left(A_{1,n}\right)]}{2n\ln{n}}\E\left[\Im\left(\widetilde{\zeta}_1\right)^2\right]\\
&\qquad+ \frac{\E[g_{xy}\left(A_{1,n}\right)]}{n\ln{n}}\E\left[\Re\left(\widetilde{\zeta}_1\right)\Im\left(\widetilde{\zeta}_1\right)\right] + \E\left[R_3\left(\frac{1}{\sqrt{n\ln{n}}}\sum_{j=1}^n\widetilde{\zeta}_j\right)\right].
\end{align*}
Similarly, we have
\begin{align*}
&\E\left[g\left(\frac{N_1}{\sqrt{n}} + \frac{1}{\sqrt{n\ln{n}}}\sum_{j=2}^n\widetilde{\zeta}_j \right)\right]\\ &\qquad= \E[g(A_{1,n})] + \frac{\E[g_{xx}\left(A_{1,n}\right)]}{2n}\E\left[\Re\left(N_1\right)^2\right] + \frac{\E[g_{yy}\left(A_{1,n}\right)]}{2n}\E\left[\Im\left(N_1\right)^2\right]\\
&\qquad+ \frac{\E[g_{xy}\left(A_{1,n}\right)]}{n}\E\left[\Re\left(N_1\right)\Im\left(N_1\right)\right] + \E\left[R_3\left(\frac{N_1}{\sqrt{n}}+\frac{1}{\sqrt{n\ln{n}}}\sum_{j=2}^n\widetilde{\zeta}_j\right)\right],
\end{align*}
where \[
\abs{R_3\left(\frac{N_1}{\sqrt{n}}+\frac{1}{\sqrt{n\ln{n}}}\sum_{j=2}^n\widetilde{\zeta}_j\right)} \leq \frac{8\cdot C_g}{2! \cdot n^{3/2}}\abs{N_1}^3.
\]
The difference between these two equations is bounded by 
\begin{align*}
&\abs{\E\left[g\left(\frac{1}{\sqrt{n\ln{n}}}\sum_{j=1}^n\widetilde{\zeta}_j\right)\right] - \E\left[g\left(\frac{N_1}{\sqrt{n}} + \frac{1}{\sqrt{n\ln{n}}}\sum_{j=2}^n\widetilde{\zeta}_j \right)\right]} \\
&\qquad\qquad \leq\frac{C_g}{2n}\abs{\frac{1}{\ln{n}}\E\left[\Re(\widetilde{\zeta}_1)^2\right] - \E\left[\Re\left(N_1\right)^2\right]}\\
&\qquad\qquad +\frac{C_g}{2n}\abs{\frac{1}{\ln{n}}\E\left[\Im(\widetilde{\zeta}_1)^2\right] - \E\left[\Im\left(N_1\right)^2\right]}\\
&\qquad\qquad +\frac{C_g}{n}\abs{\frac{1}{\ln{n}}\E\left[\Re(\widetilde{\zeta}_1)\Im(\widetilde{\zeta}_1)\right] - \E\left[\Re\left(N_1\right)\Im\left(N_1\right)\right]}\\
&\qquad\qquad + \frac{4C_g}{(n\ln{n})^{3/2}}\E\left[\abs{\widetilde{\zeta}_1}^3\right] + \frac{4C_g}{n^{3/2}}\E\left[\abs{N_1}^3\right].
\end{align*}
If we continue, for $2 \leq k \leq n$, the process of computing the second order Taylor polynomials of $g$ centered at 
\[
A_{k,n} := \frac{1}{\sqrt{n}}\sum_{j=1}^{k-1}N_j + \frac{1}{\sqrt{n\ln{n}}}\sum_{j=k+1}^n\widetilde{\zeta}_j
\]
and evaluating them at both
\[\frac{1}{\sqrt{n}}\sum_{j=1}^{k-1}N_j + \frac{1}{\sqrt{n\ln{n}}}\sum_{j=k}^n\widetilde{\zeta}_j \quad \text{and}\quad \frac{1}{\sqrt{n}}\sum_{j=1}^{k}N_j + \frac{1}{\sqrt{n\ln{n}}}\sum_{j=k+1}^n\widetilde{\zeta}_j,\]
we find that
\begin{align*}
&\left|\E\left[g\left(\frac{1}{\sqrt{n}}\sum_{j=1}^{k-1}N_j + \frac{1}{\sqrt{n\ln{n}}}\sum_{j=k}^n\widetilde{\zeta}_j\right)\right]\right.\\
&\qquad\qquad\qquad- \left.\E\left[g\left(\frac{1}{\sqrt{n}}\sum_{j=1}^{k}N_j + \frac{1}{\sqrt{n\ln{n}}}\sum_{j=k+1}^n\widetilde{\zeta}_j \right)\right]\right| \\
& \leq\frac{C_g}{2n}\abs{\frac{1}{\ln{n}}\E\left[\Re(\widetilde{\zeta}_k)^2\right] - \E\left[\Re\left(N_k\right)^2\right]}\\
&\qquad\qquad +\frac{C_g}{2n}\abs{\frac{1}{\ln{n}}\E\left[\Im(\widetilde{\zeta}_k)^2\right] - \E\left[\Im\left(N_k\right)^2\right]}\\
&\qquad\qquad +\frac{C_g}{n}\abs{\frac{1}{\ln{n}}\E\left[\Re(\widetilde{\zeta}_k)\Im(\widetilde{\zeta}_k)\right] - \E\left[\Re\left(N_k\right)\Im\left(N_k\right)\right]}\\
&\qquad\qquad + \frac{4C_g}{(n\ln{n})^{3/2}}\E\left[\abs{\widetilde{\zeta}_k}^3\right] + \frac{4C_g}{n^{3/2}}\E\left[\abs{N_k}^3\right].
\end{align*}
Now, repeatedly applying the triangle inequality and using the fact that the $\widetilde{\zeta_j}$ and $N_j$ are iid gives
\begin{equation}
\begin{aligned}
&\left|\E\left[g\left(\frac{1}{\sqrt{n\ln{n}}}\sum_{j=1}^n\widetilde{\zeta}_j\right)\right]- \E\left[g\left(\frac{1}{\sqrt{n}}\sum_{j=1}^{n}N_j \right)\right]\right|\\
&\qquad\qquad\leq\frac{C_g}{2}\abs{\frac{1}{\ln{n}}\E\left[\Re(\widetilde{\zeta}_1)^2\right] - \E\left[\Re\left(N\right)^2\right]}\\
&\qquad\qquad\qquad +\frac{C_g}{2}\abs{\frac{1}{\ln{n}}\E\left[\Im(\widetilde{\zeta}_1)^2\right] - \E\left[\Im\left(N\right)^2\right]}\\
&\qquad\qquad\qquad +C_g\abs{\frac{1}{\ln{n}}\E\left[\Re(\widetilde{\zeta}_1)\Im(\widetilde{\zeta}_1)\right] - \E\left[\Re\left(N\right)\Im\left(N\right)\right]}\\
&\qquad\qquad\qquad + \frac{4C_g}{\ln{n}\sqrt{n\ln{n}}}\E\left[\abs{\widetilde{\zeta}_1}^3\right] + \frac{4C_g}{\sqrt{n}}\E\left[\abs{N}^3\right].
\end{aligned}
\label{eqn:gen:nInterp}
\end{equation}
In order to establish Lemma \ref{lem:genCLTtrunc}, we need to show that each of the terms on the right side of \eqref{eqn:gen:nInterp} is dominated by $\eps$ as $n \to \infty$. It is in these computations that we use the fact that $f$ is continuous at $\xi_1, \ldots, \xi_s$. To take advantage of this hypothesis, fix $\eta > 0$, and note that there is a $\delta = \delta(\eta) > 0$ such that $\delta < \frac{1}{2}\min_{1 \leq k < l \leq s}\abs{\xi_k - \xi_l}$ and for which $\abs{z-\xi_k} < \delta$ implies $\abs{f(z) - f(\xi_k)} < \eta$ for $1 \leq k \leq s$. We have
\begin{align*}
\E\left[\Re^2(\widetilde{\zeta}_1)\right] &= \E\left[\Re^2(\zeta_1)\right] - \left(\E\left[\Re(\zeta_1)\right]\right)^2 \leq \E\left[\Re^2(\zeta_1)\right]\\
&\leq \E\left[\Re^2\left(\sum_{k=1}^s\frac{t_k}{\xi_k-X_1}\right) \prod_{k=1}^s\ind_{\abs{\xi_k-X_1} \geq \delta}\right]\\
&\quad {}+ \sum_{k=1}^s\E\left[\Re^2\left(\sum_{l=1}^s\frac{t_l}{\xi_l-X_1}\right) \ind_{1/(\eps\sqrt{n\ln{n}} < \abs{\xi_k-X_1} < \delta}\prod_{l\neq k}\ind_{\abs{\xi_l-X_1} \geq \delta}\right]\\
&\leq \left(\sum_{k=1}^s\frac{\abs{t_k}}{\delta}\right)^2 + \sum_{k=1}^s\E\left[\Re^2\left(\frac{t_k}{\xi_k-X_1}\right)\ind_{1/(\eps\sqrt{n\ln{n}} < \abs{\xi_k-X_1} < \delta}\right]\\
&\quad {} + 2\sum_{k=1}^s\sum_{l\neq k} \frac{\abs{t_l}}{\delta}\E\abs{\frac{t_k}{\xi_k - X_1}} + \sum_{k=1}^s\left(\sum_{l\neq k} \frac{\abs{t_l}}{\delta}\right)^2,
\end{align*}
where the last inequality follows from the fact that 
\[
\Re^2(z + w) = \Re^2(z) + 2\Re(z)\Re(w) + \Re^2(w)  \leq \Re^2(z) + 2\abs{z}\abs{w} + \abs{w}^2.
\]
Since $\E\abs{\frac{1}{\xi_k-X_1}}$ is bounded by a constant that depends only on $\mu$ (see Lemma \ref{lem:CSnice}), there is a constant $C_{s,\vec{t},\delta}$ depending on $s$, $t_1, \ldots, t_s$ and $\delta$ so that, continuing from above,
\begin{align*}
\E\left[\Re^2(\widetilde{\zeta}_1)\right] & \leq C_{s,\vec{t},\delta} + \sum_{k=1}^s\E\left[\frac{\Re^2((\xi_k - X_1)/t_k)}{\abs{(\xi_k - X_1)/t_k}^4}\ind_{1/(\eps\sqrt{n\ln{n}})<\abs{\xi_k-X_1} <\delta}\right]\\
&\leq C_{s,\vec{t},\delta} + \sum_{k=1}^s (f(\xi_k) + \eta)\abs{t_k}^2\int_0^{2\pi}\int_{1/(\abs{t_k}\eps\sqrt{n\ln{n}})}^{\delta/\abs{t_k}} \frac{r^2\cos^2{\theta}}{r^4}r\,dr\,d\theta\\
&\leq C_{s,\vec{t},\delta} + \sum_{k=1}^s\pi(f(\xi_k) + \eta)\abs{t_k}^2\ln(\delta \eps \sqrt{n\ln{n}}).
\end{align*}
Dividing both sides by $\ln{n}$ yields
\begin{equation}
\frac{1}{\ln{n}}\E\left[\Re^2(\widetilde{\zeta}_1)\right]\leq \sum_{k=1}^s\abs{t_k}^2\frac{\pi(f(\xi) + \eta)}{2} + o(1).
\label{eqn:gen:reUp}
\end{equation}
On the other hand, similar to above,
\begin{align*}
&\E\left[\Re^2(\widetilde{\zeta}_1)\right]\\
&\quad= \E\left[\Re^2(\zeta_1)\right] - \left(\E\left[\Re(\zeta_1)\right]\right)^2\\
&\quad\geq \sum_{k=1}^s\E\left[\Re^2\left(\sum_{l=1}^s\frac{t_l}{\xi_l-X_1}\right) \ind_{1/(\eps\sqrt{n\ln{n}} < \abs{\xi_k-X_1} < \delta}\prod_{l\neq k}\ind_{\abs{\xi_l-X_1} \geq \delta}\right]  - o(\ln{n})\\
&\quad\geq \sum_{k=1}^s\E\left[\frac{\Re^2((\xi_k - X_1)/t_k)}{\abs{(\xi_k - X_1)/t_k}^4}\ind_{1/(\eps\sqrt{n\ln{n}})<\abs{\xi-X_1} <\delta}\right] - o(\ln{n})\\
&\quad\geq \sum_{k=1}^s(f(\xi_k) - \eta)\abs{t_k}^2\int_0^{2\pi}\int_{1/(\abs{t_k}\eps\sqrt{n\ln{n}})}^{\delta/\abs{t_k}} \frac{r^2\cos^2{\theta}}{r^4}r\,dr\,d\theta - o(\ln{n})\\
&\quad\geq \sum_{k=1}^s\pi(f(\xi_k) - \eta)\abs{t_k}^2\ln(\delta \eps \sqrt{n\ln{n}}) - o(\ln{n}),
\end{align*}
and dividing by $\ln{n}$ yields
\begin{equation}
\frac{\E\left[\Re^2(\widetilde{\zeta}_1)\right]}{\ln{n}}\geq \sum_{k=1}^s\frac{\pi(f(\xi_k) - \eta)\abs{t_k}^2}{2} - o(1).
\label{eqn:gen:reLow}
\end{equation}
If we combine inequalities \eqref{eqn:gen:reUp} and \eqref{eqn:gen:reLow} and first take $\limsup_{n\to \infty}$ (respectively $\liminf_{n\to \infty}$) of both sides and then take $\eta \to 0$, we see that
\begin{equation}
\lim_{n\to \infty}\frac{1}{\ln{n}}\E\left[\Re^2(\widetilde{\zeta}_1)\right] = \sum_{k=1}^s\frac{\pi f(\xi_k)\abs{t_k}^2}{2} = \E[\Re^2(N)].
\label{eqn:gen:reVar}
\end{equation} (Note: $f$ is bounded, and by Lemma \ref{lem:CSnice}, the expectation of $\abs{\xi-X_1}^{-1}$ is uniformly bounded, so the limit in $n$ is uniform in $\xi$.) Nearly identical arguments 
show that 
\begin{equation}
\lim_{n\to \infty}\frac{1}{\ln{n}}\E\left[\Im^2(\widetilde{\zeta}_1)\right] = \sum_{k=1}^s\frac{\pi f(\xi_k)\abs{t_k}^2}{2} = \E[\Im^2(N)]
\label{eqn:gen:imVar}
\end{equation}
and 
\begin{equation}
\lim_{n\to \infty}\frac{1}{\ln{n}}\E\left[\Re(\widetilde{\zeta_1})\Im(\widetilde{\zeta}_1)\right] = 0 = \E[\Re(N)\Im(N)],
\label{eqn:gen:coVar}
\end{equation}
with the only difference being that when we find bounds for $\E\left[\Re(\widetilde{\zeta_1})\Im(\widetilde{\zeta}_1)\right]$, we consider separately the cases where the integrand is positive and negative. 

In our quest to prove Lemma \ref{lem:genCLTtrunc}, we next show that 
\begin{equation}
\limsup_{n\to \infty} \frac{1}{\ln{n}\sqrt{n\ln{n}}}\E\left[\abs{\widetilde{\zeta_1}}^3\right] \leq O_{\mu,s,\vec{t}}(\eps).
\label{eqn:gen:thirdMo}
\end{equation}
Note that 
\[
\E\left[\abs{\widetilde{\zeta}_1}^3\right] \leq 2\cdot\E\left[\abs{\zeta_1}^3\right] + 6\cdot\E\left[\abs{\zeta_1}^2\right] \cdot \E\abs{\zeta_1} \leq 8\eps\sqrt{n\ln{n}}\sum_{k=1}^s\abs{t_k}\E\left[\abs{\zeta_1}^2\right],
\]
where the last inequality comes from using the fact that, almost surely, $\abs{\zeta_1} \leq \eps\sqrt{n\ln{n}}\sum_{k=1}^s\abs{t_k} $. Choose $\delta_1>0$ so that $\delta_1 < \frac{1}{2}\min_{1\leq k < l \leq s}\abs{\xi_k-\xi_l}$ and that for $\abs{z-\xi_k} < \delta_1$, $1 \leq k \leq s$, we have $\abs{f(z) - f(\xi_k)} < 1$. Then, it follows that for $n$ large enough to ensure $\frac{1}{\eps\sqrt{n\ln{n}}} \leq \delta_1$,
\begin{align*}
\frac{\E\left[\abs{\widetilde{\zeta_1}}^3\right]}{\ln{n}\sqrt{n\ln{n}}} & \leq \sum_{k=1}^s\frac{8\eps\abs{t_k}}{\ln{n}}\E\left[\abs{\zeta_1}^2\right]\\
&\leq\sum_{k=1}^s\frac{8\eps\abs{t_k}}{\ln{n}}\left(\E\left[\left(\sum_{k=1}^s\frac{t_k}{\xi_k-X_1}\right)^2 \prod_{k=1}^s\ind_{\abs{\xi_k-X_1} \geq \delta_1}\right]\right.\\
&\quad {}+ \left.\sum_{k=1}^s\E\left[\left(\sum_{l=1}^s\frac{t_l}{\xi_l-X_1}\right)^2 \ind_{1/(\eps\sqrt{n\ln{n}} < \abs{\xi_k-X_1} < \delta_1}\prod_{l\neq k}\ind_{\abs{\xi_l-X_1} \geq \delta_1}\right]\right)\\
&\leq \sum_{k=1}^s\frac{8\eps\abs{t_k}}{\ln{n}}\left(\sum_{l=1}^s\frac{\abs{t_k}}{\delta_1}\right)^2\\
&\quad {} + \sum_{k=1}^s\frac{8\eps\abs{t_k}}{\ln{n}}\sum_{l=1}^s\E\left[\abs{\frac{t_l}{\xi_l-X_1}}^2\ind_{1/(\eps\sqrt{n\ln{n}} < \abs{\xi_l-X_1} < \delta_1}\right]\\
&\quad {} + \sum_{k=1}^s\frac{8\eps\abs{t_k}}{\ln{n}}\left(\sum_{l=1}^s2\E\abs{\frac{t_l}{\xi_l - X_1}}\sum_{j\neq l} \frac{\abs{t_j}}{\delta_1} + \sum_{l=1}^s\left(\sum_{j\neq l} \frac{\abs{t_j}}{\delta_1}\right)^2\right),
\end{align*}
where we have used the fact that
$\abs{z+w}^2 \leq \abs{z}^2 + 2\abs{z}\abs{w} + \abs{w}^2.$
Continuing from above, where the second sum is the only one of non-negligible order, we have
\begin{align*}
&\frac{1}{\ln{n}\sqrt{n\ln{n}}}\E\left[\abs{\widetilde{\zeta_1}}^3\right]\\
&\qquad \leq \sum_{k=1}^s\frac{8\eps\abs{t_k}}{\ln{n}}\sum_{l=1}^s\E\left[\abs{\frac{t_l}{\xi_l-X_1}}^2\ind_{1/(\eps\sqrt{n\ln{n}} < \abs{\xi_l-X_1} < \delta_1}\right] + o(1)\\
&\qquad\leq \sum_{k,l = 1}^s\frac{8\eps\abs{t_k}}{\ln{n}}8\eps(f(\xi_l)+1)\abs{t_l}^2\int_0^{2\pi}\int_{1/(\eps\sqrt{n\ln{n}})}^{\delta_1}\frac{1}{r^2}r\,dr\,d\theta + o(1)\\
&\qquad= \sum_{k,l = 1}^s\frac{16\pi\eps(f(\xi_l) + 1)\abs{t_k}\abs{t_l}^2\ln(\delta_1\eps\sqrt{n\ln{n}})}{\ln{n}} + o(1)
\end{align*}
and taking $\limsup_{n\to \infty}$ establishes \eqref{eqn:gen:thirdMo}.
We conclude the proof of Lemma \ref{lem:genCLTtrunc} by combining equations \eqref{eqn:gen:nInterp}, \eqref{eqn:gen:reVar}, \eqref{eqn:gen:imVar}, \eqref{eqn:gen:coVar}, and \eqref{eqn:gen:thirdMo} in view of the facts that $\abs{N}$ has a finite third moment and $f(z)$ and $\eps$ are bounded. 
\end{proof}

In order to establish \eqref{eqn:portmanteau}, we still need to remove the truncation, which we will accomplish through a series of interpolations. We have
\begin{align*}
&\frac{1}{\sqrt{n\ln{n}}}\sum_{j=1}^n\sum_{k=1}^st_k\left(\frac{1}{\xi_k-X_j}- m_\mu(\xi_k)\right)\\
&\hspace{1in}= \frac{1}{\sqrt{n\ln{n}}}\sum_{j=1}^n\left(\sum_{k=1}^s\frac{t_k}{\xi_k-X_j} - \zeta_j\right) + \frac{1}{\sqrt{n\ln{n}}}\sum_{j=1}^n\widetilde{\zeta}_j\\
&\hspace{1in}\qquad + \frac{1}{\sqrt{n\ln{n}}}\sum_{j=1}^n\left(\E[\zeta_j] - \sum_{k=1}^s t_km_\mu(\xi_k)\right).
\end{align*}
For $n$ large enough to guarantee that the density, $f$, is well-defined and bounded by a constant, $C_f$, on \[\bigcup_{k=1}^sB\left(\xi_k,\frac{1}{\eps\sqrt{n\ln{n}}}\right),\]
it follows that
\[
\E\abs{ \frac{1}{\sqrt{n\ln{n}}}\sum_{j=1}^n\left(\sum_{k=1}^s\frac{t_k}{\xi_k-X_j} - \zeta_j\right)}\quad\text{and}\quad\abs{\frac{1}{\sqrt{n\ln{n}}}\sum_{j=1}^n\left(\E[\zeta_j] - \sum_{k=1}^st_k m_\mu(\xi_k)\right)}
\]
are both less than
\begin{align*}
\frac{n}{\sqrt{n\ln{n}}}\sum_{k=1}^s &\abs{t_k}\E\left[\frac{1}{\abs{\xi_k-X_1}}\ind_{\abs{\xi_k-X_1}^{-1} \geq \eps\sqrt{n\ln{n}}}\right]\\
&\leq \frac{ns\max_{1 \leq k \leq s}\abs{t_k}\cdot C_f}{\sqrt{n\ln{n}}}\int_0^{2\pi}\int_0^{1/(\eps\sqrt{n\ln{n}})}\frac{1}{r}r\,dr\,d\theta\\
&\leq \frac{ns\max_{1 \leq k \leq s}\abs{t_k}\cdot C_f}{\sqrt{n\ln{n}}}\frac{2\pi}{\eps\sqrt{n\ln{n}}}\\
&=o(1).
\end{align*}
Consequently,
\[
\E\abs{\frac{1}{\sqrt{n\ln{n}}}\sum_{j=1}^n\sum_{k=1}^st_k\left(\frac{1}{\xi-X_j}- m_\mu(\xi)\right) - \frac{1}{\sqrt{n\ln{n}}}\sum_{j=1}^n\widetilde{\zeta}_j} = o(1).
\]
We can take advantage of the fact that $g$ is Lipshitz (indeed, $g$ is smooth with compact support, so it has bounded partial derivatives), to obtain 
\[
\E\abs{g\left(\frac{1}{\sqrt{n\ln{n}}}\sum_{j=1}^n\sum_{k=1}^st_k\left(\frac{1}{\xi-X_j}- m_\mu(\xi)\right)\right) - g\left( \frac{1}{\sqrt{n\ln{n}}}\sum_{j=1}^n\widetilde{\zeta}_j\right)} = o(1).
\]
Lemma \ref{lem:genCLTtrunc} now implies that for $n$ larger than a constant depending on $\mu, g$, $\eps$, $s$, and $t_1, \ldots, t_s$,
\begin{align*}
&\abs{\E\left[g\left(\frac{1}{\sqrt{n\ln{n}}}\sum_{j=1}^n\sum_{k=1}^st_k\left[\frac{1}{\xi-X_j} - m_\mu(\xi)\right]\right)\right] - \E\left[g\left(\frac{1}{\sqrt{n}}\sum_{j=1}^nN_j\right)\right]}\\
&\qquad \leq \abs{\E\left[g\left(\frac{1}{\sqrt{n\ln{n}}}\sum_{j=1}^n\widetilde{\zeta}_j\right)\right] - \E\left[g\left(\frac{1}{\sqrt{n}}\sum_{j=1}^nN_j\right)\right]} +o(1)\\
&\qquad = O_{\mu,s,\vec{t},g}(\eps),
\end{align*}
so taking $\eps \to 0$ yields equation \eqref{eqn:portmanteau}. The conclusion of Theorem \ref{thm:genCLT} follows since our choice of $g$ was arbitrary.
\end{proof}


\bibliography{localCP24}

\begin{thebibliography}{10}

\bibitem{B}
M.~S. Bartlett.
\newblock An inverse matrix adjustment arising in discriminant analysis.
\newblock {\em Ann. Math. Statistics}, 22:107--111, 1951.

\bibitem{Bill}
P.~Billingsley.
\newblock {\em Probability and measure}.
\newblock Wiley Series in Probability and Statistics. John Wiley \& Sons, Inc.,
  Hoboken, NJ, 2012.
\newblock Anniversary edition [of MR1324786], With a foreword by Steve Lalley
  and a brief biography of Billingsley by Steve Koppes.

\bibitem{BLR}
S.-S. Byun, J.~Lee, and T.~R. Reddy.
\newblock Zeros of random polynomials and its higher derivatives.
\newblock Available at arXiv:1801.08974, 2018.

\bibitem{C}
S.~Chatterjee.
\newblock A generalization of the {L}indeberg principle.
\newblock {\em Ann. Probab.}, 34(6):2061--2076, 2006.

\bibitem{CN}
W.~S. Cheung and T.~W. Ng.
\newblock A companion matrix approach to the study of zeros and critical points
  of a polynomial.
\newblock {\em J. Math. Anal. Appl.}, 319(2):690--707, 2006.

\bibitem{CN2}
W.~S. Cheung and T.~W. Ng.
\newblock Relationship between the zeros of two polynomials.
\newblock {\em Linear Algebra Appl.}, 432(1):107--115, 2010.

\bibitem{DH}
M.~R. Dennis and J.~H. Hannay.
\newblock Saddle points in the chaotic analytic function and {G}inibre
  characteristic polynomial.
\newblock {\em J. Phys. A}, 36(12):3379--3383, 2003.
\newblock Random matrix theory.

\bibitem{FI}
I.~F\'{a}ry and E.~M. Isenberg.
\newblock On a converse of the {J}ordan curve theorem.
\newblock {\em Amer. Math. Monthly}, 81:636--639, 1974.

\bibitem{F}
W.~Feller.
\newblock {\em An introduction to probability theory and its applications.
  {V}ol. {II}}.
\newblock Second edition. John Wiley \& Sons, Inc., New York-London-Sydney,
  1971.

\bibitem{Gal}
J.~Galambos.
\newblock {\em Advanced probability theory}, volume~10 of {\em Probability:
  Pure and Applied}.
\newblock Marcel Dekker, Inc., New York, second edition, 1995.

\bibitem{G}
A.~Gut.
\newblock {\em Probability: a graduate course}.
\newblock Springer Texts in Statistics. Springer, New York, 2005.

\bibitem{H1}
B.~Hanin.
\newblock Correlations and pairing between zeros and critical points of
  {G}aussian random polynomials.
\newblock {\em Int. Math. Res. Not. IMRN}, 2015(2):381--421, 2015.

\bibitem{H2}
B.~Hanin.
\newblock Pairing of zeros and critical points for random meromorphic functions
  on {R}iemann surfaces.
\newblock {\em Math. Res. Lett.}, 22(1):111--140, 2015.

\bibitem{H3}
B.~Hanin.
\newblock Pairing of zeros and critical points for random polynomials.
\newblock {\em Ann. Inst. Henri Poincar\'e Probab. Stat.}, 53(3):1498--1511,
  2017.

\bibitem{HJ}
R.~A. Horn and C.~R. Johnson.
\newblock {\em Matrix analysis}.
\newblock Cambridge University Press, Cambridge, second edition, 2013.

\bibitem{JHough}
J.~B. Hough, M.~Krishnapur, Y.~Peres, and B.~Vir\'ag.
\newblock {\em Zeros of {G}aussian analytic functions and determinantal point
  processes}, volume~51 of {\em University Lecture Series}.
\newblock American Mathematical Society, Providence, RI, 2009.

\bibitem{K}
Z.~Kabluchko.
\newblock Critical points of random polynomials with independent identically
  distributed roots.
\newblock {\em Proc. Amer. Math. Soc.}, 143(2):695--702, 2015.

\bibitem{KS}
Z.~Kabluchko and H.~Seidel.
\newblock Distances between zeroes and critical points for random polynomials
  with i.i.d. zeroes.
\newblock Available at arXiv:1807.02140, 2018.

\bibitem{KI}
N.~L. Komarova and I.~Rivin.
\newblock Harmonic mean, random polynomials and stochastic matrices.
\newblock {\em Adv. in Appl. Math.}, 31(2):501--526, 2003.

\bibitem{L}
J.~W. Lindeberg.
\newblock Eine neue {H}erleitung des {E}xponentialgesetzes in der
  {W}ahrscheinlichkeitsrechnung.
\newblock {\em Math. Z.}, 15(1):211--225, 1922.

\bibitem{M}
M.~Marden.
\newblock {\em Geometry of polynomials}.
\newblock Second edition. Mathematical Surveys, No. 3. American Mathematical
  Society, Providence, R.I., 1966.

\bibitem{O}
S.~O'Rourke.
\newblock Critical points of random polynomials and characteristic polynomials
  of random matrices.
\newblock {\em Int. Math. Res. Not. IMRN}, 2016(18):5616--5651, 2016.

\bibitem{OW}
S.~O'Rourke and N.~Williams.
\newblock Pairing between zeros and critical points of random polynomials with
  independent roots.
\newblock {\em Trans. Amer. Math. Soc.}, 371(4):2343--2381, 2019.

\bibitem{OWo}
S.~O'Rourke and P.~M. Wood.
\newblock Spectra of nearly {H}ermitian random matrices.
\newblock {\em Ann. Inst. Henri Poincar\'e Probab. Stat.}, 53(3):1241--1279,
  2017.

\bibitem{PR}
R.~Pemantle and I.~Rivin.
\newblock The distribution of zeros of the derivative of a random polynomial.
\newblock In {\em Advances in combinatorics}, pages 259--273. Springer,
  Heidelberg, 2013.

\bibitem{P}
V.~V. Petrov.
\newblock {\em Limit theorems of probability theory}, volume~4 of {\em Oxford
  Studies in Probability}.
\newblock The Clarendon Press, Oxford University Press, New York, 1995.
\newblock Sequences of independent random variables, Oxford Science
  Publications.

\bibitem{RS}
Q.~I. Rahman and G.~Schmeisser.
\newblock {\em Analytic theory of polynomials}, volume~26 of {\em London
  Mathematical Society Monographs. New Series}.
\newblock The Clarendon Press, Oxford University Press, Oxford, 2002.

\bibitem{TRR}
T.~R. Reddy.
\newblock On critical points of random polynomials and spectrum of certain
  products of random matrices.
\newblock Available at arXiv:1602.05298, 2016.

\bibitem{SS}
S.~Steinerberger.
\newblock A stability version of the gauss-lucas theorem and applications.
\newblock Available at arXiv:1805.10454, 2018.

\bibitem{S}
S.~D. Subramanian.
\newblock On the distribution of critical points of a polynomial.
\newblock {\em Electron. Commun. Probab.}, 17:no. 37, 9, 2012.

\bibitem{Ttop}
T.~Tao.
\newblock {\em Topics in random matrix theory}, volume 132 of {\em Graduate
  Studies in Mathematics}.
\newblock American Mathematical Society, Providence, RI, 2012.

\bibitem{Tout}
T.~Tao.
\newblock Outliers in the spectrum of iid matrices with bounded rank
  perturbations.
\newblock {\em Probab. Theory Related Fields}, 155(1-2):231--263, 2013.

\bibitem{TV2011}
T.~Tao and V.~Vu.
\newblock Random matrices: universality of local eigenvalue statistics.
\newblock {\em Acta Math.}, 206(1):127--204, 2011.

\bibitem{TVpoly}
T.~Tao and V.~Vu.
\newblock Local universality of zeroes of random polynomials.
\newblock {\em Int. Math. Res. Not. IMRN}, 2015(13):5053--5139, 2015.

\bibitem{TV}
T.~Tao and V.~Vu.
\newblock Random matrices: universality of local spectral statistics of
  non-{H}ermitian matrices.
\newblock {\em Ann. Probab.}, 43(2):782--874, 2015.

\bibitem{TCar}
C.~Thomassen.
\newblock The converse of the {J}ordan curve theorem and a characterization of
  planar maps.
\newblock {\em Geom. Dedicata}, 32(1):53--57, 1989.

\bibitem{V}
C.~Villani.
\newblock {\em Optimal transport}, volume 338 of {\em Grundlehren der
  Mathematischen Wissenschaften [Fundamental Principles of Mathematical
  Sciences]}.
\newblock Springer-Verlag, Berlin, 2009.
\newblock Old and new.

\bibitem{W}
R.~L. Wilder.
\newblock {\em Topology of manifolds}.
\newblock American Mathematical Society Colloquium Publications, Vol. XXXII.
  American Mathematical Society, Providence, R.I., 1963.

\end{thebibliography}
\bibliographystyle{abbrv}

\end{document}